\documentclass[a4paper]{amsart}
\usepackage[utf8]{inputenc}
\usepackage[english]{babel}
\usepackage{amsthm}
\usepackage{amsmath}
\usepackage{amsfonts}
\usepackage{mathtools}
\usepackage{enumitem}
\usepackage{xcolor}
\usepackage{amssymb}
\usepackage{csquotes}
\usepackage{stmaryrd}
\usepackage[capitalise, noabbrev]{cleveref}
\usepackage{cancel}
\usepackage{ textcomp }

\setlist[enumerate,1]{label = (\roman*)}

\theoremstyle{definition}
\newtheorem{definition}{Definition}[section]
\newtheorem{proposition}[definition]{Proposition}
\newtheorem{theorem}[definition]{Theorem}
\newtheorem{lemma}[definition]{Lemma}
\newtheorem{corollary}[definition]{Corollary}
\newtheorem{example}[definition]{Example}
\newtheorem*{main-theorem}{Theorem}
\newtheorem*{fact}{Fact}
\theoremstyle{remark}
\newtheorem*{remark}{Remark}
\newtheorem*{claim}{Claim}

\renewcommand{\restriction}{ {\upharpoonright} }
\newcommand{\Z}{\mathcal{Z}}
\newcommand{\A}{\mathcal{A}}
\newcommand{\B}{\mathcal{B}}
\newcommand{\C}{\mathcal{C}}
\newcommand{\D}{\mathcal{D}}
\newcommand{\E}{\mathcal{E}}
\newcommand{\M}{\mathcal{M}}

\newcommand{\U}{\mathcal{U}}
\newcommand{\F}{\mathcal{F}}
\renewcommand{\S}{\mathcal{S}}

\DeclareMathOperator{\dep}{\mathop{=}}
\newcommand{\vx}{\vec{x}}

\newcommand{\vy}{\vec{y}}
\newcommand{\vz}{\vec{z}}

\DeclareMathOperator{\Pow}{\wp}
\DeclareMathOperator{\dom}{dom}
\DeclareMathOperator{\ran}{ran}

\DeclareMathOperator{\acl}{acl}
\DeclareMathOperator{\aut}{aut}

\newcommand{\acleq}{\acl^{\mathrm{eq}}}

\DeclareMathOperator{\Fv}{Fv}

\newcommand\ivee{\mathrel{\rotatebox[origin=c]{270}{$\geqslant$}}}

\newcommand{\dobigivee}[1]{%
	\vcenter{#1\kern.2ex\hbox{$\raisebox{-0.015ex}{\rotatebox[origin=c]{-2}{\text{\textbf\textbackslash}}}\unskip\hspace*{-1.35ex}\ignorespaces\bigvee$}\kern.2ex}%
}

\newcommand{\existsone}{\exists^1}
\newcommand{\forallone}{\forall^1}

\newcommand{\cneg}{\mathord{\dot{\sim}}}

\newcommand\fol{\mathsf{FO}}
\newcommand\sol{\mathsf{SO}}
\newcommand\fot{\mathsf{FOT}}

\newcommand\eso{\mathsf{ESO}}

\newcommand\foil{\mathsf{FO}(\perp_c)}

\DeclareMathOperator{\rel}{rel}

\newcommand\nmodels{\mathbin{\cancel{\models}}}

\newcommand{\Hom}{\mathrm{Hom}}
\newcommand{\K}{\mathcal{K}}
\newcommand{\R}{\mathcal{R}}
\newcommand{\Y}{\mathcal{Y}}
\newcommand{\X}{\mathcal{X}}
\DeclareMathOperator{\team}{team}
\DeclareMathOperator{\ar}{ar}

\newcommand{\V}{\mathbf{V}}
\DeclareMathOperator{\LS}{LS}

\def\Ind{\setbox0=\hbox{$x$}\kern\wd0\hbox to 0pt{\hss$\mid$\hss}
	\lower.9\ht0\hbox to 0pt{\hss$\smile$\hss}\kern\wd0}
\def\Notind{\setbox0=\hbox{$x$}\kern\wd0\hbox to 0pt{\mathchardef
		\nn=12854\hss$\nn$\kern1.4\wd0\hss}\hbox to
	0pt{\hss$\mid$\hss}\lower.9\ht0 \hbox to 0pt{\hss$\smile$\hss}\kern\wd0}

\def\ind{\mathop{\mathpalette\Ind{}}}

\usepackage{tikz}
\usepackage{wrapfig}
\usepackage{tikz-cd}
\usepackage[backend=biber, isbn=false,  doi=false, url=false, style=numeric, sortcites=true, citestyle=numeric, giveninits=true]{biblatex}
\bibliography{revision.bib}

\theoremstyle{definition}

\renewcommand{\A}{\mathfrak{A}}
\renewcommand{\B}{\mathfrak{B}}
\renewcommand{\C}{\mathfrak{C}}
\renewcommand{\D}{\mathfrak{D}}
\renewcommand{\E}{\mathfrak{E}}
\renewcommand{\F}{\mathfrak{F}}
\renewcommand{\M}{\mathfrak{M}}
\newcommand{\bM}{\mathbf{M}}

\DeclareMathOperator{\Th}{Th}
\newcommand{\ESOTh}{\Th_{\eso}}
\DeclareMathOperator{\Str}{Str}
\DeclareMathOperator{\Sig}{Sig}

\newcommand{\Fml}{\mathrm{Fml}_L}
\newcommand{\Sat}{\mathrm{Sat}_L}

\newcommand{\wcimp}{\rightarrowtriangle}
\newcommand{\wcequiv}{\leftrightarrowtriangle}

\DeclareMathOperator{\gtype}{tp_g}
\DeclareMathOperator{\foltype}{tp_\fol}
\DeclareMathOperator{\fottype}{tp_\fot}

\DeclareMathOperator{\stp}{stp}
\DeclareMathOperator{\FE}{FE}

\DeclareMathOperator{\Mod}{Mod}
\DeclareMathOperator{\Sub}{SubMod}
\DeclareMathOperator{\id}{id}
\DeclareMathOperator{\Aut}{Aut}

\DeclareMathOperator{\Diag}{Diag}
\DeclareMathOperator{\eq}{eq}

\newcommand{\cF}{\mathcal{F}}
\newcommand{\cG}{\mathcal{G}}

\newcommand{\f}{\mathcal{f}}
\newcommand{\g}{\mathcal{g}}
\newcommand{\h}{\mathcal{h}}
\renewcommand{\k}{\mathcal{k}}

\usepackage[cal=boondoxo]{mathalfa}
\renewcommand{\Pow}{\mathcal{P}}

\DeclareMathOperator{\upset}{\uparrow}

\DeclareMathOperator{\GMod}{GMod}
\DeclareMathOperator{\GSubMod}{GSubMod}

\DeclareMathOperator{\cl}{cl}

\title{On the Model Theory of Second-Order Objects}
\author[T. Hyttinen]{Tapani Hyttinen}
\author[J. Puljujärvi]{Joni Puljujärvi}
\author[D. E. Quadrellaro]{Davide Emilio Quadrellaro}

\address{(Tapani Hyttinen) Department of Mathematics and Statistics, University of Helsinki, P.O. Box 68 (Pietari Kalmin katu 5), 00014 Helsinki, Finland.}
\email{tapani.hyttinen@helsinki.fi}

\address{(Joni Puljujärvi) Department of Computer Science, University College London, 66–72 Gower Street, London WC1E 6EA, United Kingdom.}
\email{joni.puljujarvi@ucl.ac.uk}

\address{(Davide Emilio Quadrellaro) Department of Mathematics “Giuseppe Peano”, University of Torino, Via Carlo Alberto 10, 10123 Torino, Italy.}
\email{davideemilio.quadrellaro@unito.it}
\address{Istituto Nazionale di Alta Matematica ``Francesco Severi'',
	Piazzale Aldo Moro 5
	00185 Roma, Italy.}
\email{quadrellaro@altamatematica.it}

\thanks{This project has received funding from the European Research Council (ERC) under the
European Union’s Horizon 2020 research and innovation programme (grant agreement No
101020762). J.~Puljujärvi was partially supported by the Academy of Finland, grant 322795. D. E. Quadrellaro was supported by an INdAM Postdoctoral Grant.}

\keywords{abstract elementary classes, existential second-order logic, team semantics}
\subjclass[2020]{03C45, 03C48, 03C85}

\begin{document}

\begin{abstract}
    Motivated by team semantics and existential second-order logic, we develop a model-theoretic framework for studying second-order objects such as sets and relations. We introduce the notion of an abstract elementary team category that generalizes the standard notion of an abstract elementary class, and show that it is an example of an accessible category. We apply our framework to show that the logic $\fot$ introduced by Kontinen and Yang \cite{MR4594292} satisfies a version of Lindström's Theorem. Finally, we consider the problem of transferring categoricity between different cardinalities for complete theories in existential second-order logic (or independence logic) and prove both a downwards and an upwards categoricity transfer result.
\end{abstract}

\maketitle

\setcounter{tocdepth}{1}
\tableofcontents

\section{Introduction}

Model theory is the study of general properties of mathematical structures. In traditional model theory, the center of attention---we could say the smallest unit of interest---are elements, or tuples of elements. One is often interested in whether certain types of elements or tuples exist, or how many different types there might be, or how elements get mapped when one moves from one structure to another via some mapping that preserves certain properties of structures.

On the other hand, often in mathematics one is interested in sets of elements, or sets of tuples (i.e., relations). While elements are first-order objects, sets and relations are second-order objects. Not much model theory has been done in a setting where the basic building block is, instead of an element, a set. The reason may be that second-order logic---the natural counterpart of first-order logic in this setting---is so strong that it lacks many properties that make elementary model theory interesting. For example, the second-order theory of many interesting structures---such as the ordered field of real numbers, the field of complex numbers, or the semiring of natural numbers---are categorical, i.e. they have a unique model up to isomorphism. However, it may not be completely hopeless to study relations using simply first-order logic, or some other logic that lies between first and second order.

Our interest in the model theory of relations comes from the field of \emph{team semantics}. Team semantics is an extension of the usual Tarski semantics of first-order logic that allows formulas to be evaluated over sets of assignments, called \emph{teams}, rather than single assignments. Such a framework was originally introduced by \textcite{Hodges} in order to provide a compositional semantics to Hintikka and Sandu's IF-logic. In particular, \textcite{VAANANEN2010817} and \textcite{Vaananen2007-VNNDLA} used team semantics to extend first-order logic by means of the so-called \emph{dependence atom}
\[
    \dep(x_0,\dots,x_{n-1};y),
\]
which we read as saying ``the value of $x_1,\dots,x_n$ completely determines the value of $y$''. The resulting logic, namely \emph{dependence logic}, dramatically increases the expressive power of first-order logic and turns out to be equivalent both to IF-logic and to the existential fragment of second-order logic $\eso$ with the usual Tarski semantics (cf.~\cref{Translation between independence logic and ESO}).

Additionally, team semantics proved soon to be a very powerful and flexible framework, and it was noticed by Väänänen and others that it allows to consider several atoms characterizing various notions of dependence, e.g. the \emph{independence atom} $x\perp y$ \cite{gradel2013dependence} or the \emph{inclusion atom} $x\subseteq y$ \cite{galliani2012inclusion}. Importantly, each of the logics that one obtains by extending the syntax of first-order logic by these new atomic formulas corresponds to a different fragment of existential second order logic $\eso$ (with usual semantics) via a translation that goes both ways. More recently, however, \textcite{MR4594292} have also introduced a logic over team semantics that corresponds to first-order logic in this way, i.e. it captures exactly the elementary properties of teams. From the perspective of the present work, the results from \cite{MR4594292} represent a significant breakthrough. In our view, Kontinen and Yang's work shows that the most essential aspect of team semantics lies in this shift of attention from first-order to second-order objects, rather than the increase in expressive power.

In fact, since teams are essentially relations, it is very natural to adopt team semantics as a framework to study the model theory of sets and relations. Interestingly enough, there has not been much study of the model theory of logics in team semantics, with the exception of the seminal work of \textcite{Vaananen2007-VNNDLA} and the study of team ultraproducts by \textcite{luck2020team}. In \cite{puljujärvi2022compactness}, the study of model theory for team semantics was initated, in particular by proving a general version of the compactness theorem for independence logic and several of its fragments: in this version, instead of theories---that is, sets of sentences---one looks at sets of possibly open formulas. In \cite{MR4594292}, a version of the compactness theorem for open formulas was proved where the number of free variables allowed to occur in the sets of formulas was limited to be countable. This limitation was then removed in \cite{puljujärvi2022compactness}.

In this article we continue the work from \cite{puljujärvi2022compactness} and develop a model-theoretic framework for team semantics.  We further believe this framework provides the natural environment for studying the elementary and second-order properties of sets and relations. In particular, we apply the model-theoretic toolkit developed in this article to study the problem of transferring categoricity for complete theories in existential second-order logic. ``Categoricity transfer'' refers to the phenomenon where the categoricity of a theory in one cardinality follows from its categoricity in another cardinality; for instance, by Morley's theorem, any first-order theory that is categorical in some uncountable cardinality is, in fact, categorical in every uncountable cardinality. Our categoricity transfer theorems provide an important example of how team semantics can be used to deliver results on the model theory of existential second-order logic.

Let us now summarize the structure and the major contributions of this work. In \cref{Section_Preliminaries} we review some basic notions from team semantics and existential second-order logic. In \cref{Section_maps} we introduce the notion of team maps between structures and prove some of their basic properties. In \cref{Section_AEC}, we discus problems arising from trying to fit a class of structures together with elementary team embeddings into the framework of abstract elementary classes (AECs). We resolve these problems by defining a generalization of AECs that makes more sense in our setting. The objects of these so-called abstract elementary team categories (AETCs) are reminiscent of general models (or Henkin models) of second-order logic. The main result of the section is \cref{theorem:aetc.accessible}, which states that AETCs, when treated as categories with certain team maps as morphisms, are examples of accessible categories~\cite{adamek1994locally}, which are known to generalize AECs. In \cref{Section_Monster} we use the general framework of abstract elementary team categories to show that one can build a suitable version of the monster model construction for theories in independence logic and $\eso$. 

We then consider an application of this model-theoretic toolkit in \cref{Lindstrom}, where we use the machinery introduced in the previous sections to show  that the elementary team logic $\fot$ introduced by Kontinen and Yang in~\cite{MR4594292} satisfies a version of Lindström's theorem. The main tool utilized in the proofs is the direct limit construction from \cref{colimits.subsection}. Finally, we conclude the article with a more model-theoretically oriented addendum in \cref{section:complete_theories}, where we focus on the problem of classifying models of complete existential second-order theories. We prove, in particular, two versions of categoricity transfer: one from uncountable to countable cardinals (\cref{ESO categoricity transfer down to countable}) and one from countable to uncountable cardinals (\cref{ESO categoricity transfer up to uncountable}). These two results show how the spectrum function of an existential second-order theory depends in many cases on the stability properties of its first-order reduct, and they also display that the present work has many potential connections to the standard elementary model theory. \cref{section:complete_theories} is fairly independent from \cref{Section_AEC}.

\subsection*{Acknowledgements}
We thank Åsa Hirvonen, Jonathan Kirby, Nicolás Nájar, Tapio Saarinen, Jouko Väänänen, Andrés Villaveces and Fan Yang for useful comments and discussions regarding this manuscript. We are also thankful to two anonymous reviewers for reading and commenting a prior version of this article.

\section{Preliminaries}\label{Section_Preliminaries}

\subsection{Notational Conventions}

We denote sets and relations with uppercase Roman letters such as $A$, $B$ and $C$ and elements by lowercase Roman letters such as $a$, $b$ and $c$. We usually denote collections of sets $A$ by calligraphic letters such as $\mathcal{A}$. By Fraktur letters such as $\A$, $\B$ and $\C$ we denote first-order structures of a given signature.

If $n$ is a natural number, an $n$-tuple is technically a function with domain $n$. We usually denote tuples by $\vec a$, $\vec b$ etc. and often write $\vec a = (a_0,\dots,a_{n-1})$, where $a_i = \vec a(i)$ for $i<n$. We call $n$ the length of $\vec a$ and may denote it by $|\vec a|$. The empty set is the unique $0$-tuple. An $n$-ary relation is a set of $n$-tuples. If $A$ is an $n$-ary relation, we denote by $\ar(A)$ the arity of $A$, i.e. the number $n$. The empty set is considered a relation of every arity. $\{\emptyset\}$ is the unique nonempty $0$-ary relation. If $f$ is a function from $A^n$ to $A$ for some $n<\omega$, then we say that it has arity $n$ and denote this number by $\ar(f)$.

If $\vec a = (a_0,\dots,a_{n-1})$ is an $n$-tuple and $\vec b = (b_0,\dots,b_{m-1})$ is an $m$-tuple, the concatenation of $\vec a$ and $\vec b$ is the $n+m$-tuple $(a_0,\dots,a_{n-1},b_0,\dots,b_{m-1})$, which we usually denote just by $\vec a\vec b$. If there is a risk of confusion, we may also denote the concatenation by $\vec a^\frown\vec b$. Often for our convenience, we also identify a pair $(\vec a,\vec b)$ with the tuple $\vec a\vec b$. Note that $\emptyset\vec a = \vec a\emptyset = \vec a$ for any tuple $\vec a$. If $A$ is an $n$-ary and $B$ an $m$-ary relation, then we identify the Cartesian product $A\times B$ with the $(n+m)$-ary relation
\[
    \{ \vec a\vec b \mid \text{$\vec a\in A$ and $\vec b\in B$} \}.
\]
Note that by this convention, $\{\emptyset\}\times A = A\times\{\emptyset\} = A$. Also, $\emptyset\times A = A\times\emptyset = \emptyset$.

If $A$ is an $(n+k)$-ary and $B$ a $(k+m)$-ary relation, then the natural $k$-join $A\bowtie_k B$ of $A$ and $B$ is the $(n+k+m)$-ary relation
\[
    \{ \vec a\vec b\vec c \mid \text{$|\vec a|=n$, $|\vec b|=k$, $|\vec c|=m$ and $\vec a\vec b\in A$ and $\vec b\vec c\in B$} \}.
\]
Note that $A\times B = A\bowtie_0 B$. If $A$ and $B$ are $n$-ary relations, then $A\cap B = A\bowtie_n B$.

Let $A$ be an $n$-ary relation, let $i_0,\dots,i_{m-1}<n$ and denote $\vec\imath = (i_0,\dots,i_{m-1})$. We then denote by $\Pr_{\vec\imath}(A)$ the $m$-ary relation
\[
    \{ (a_{i_0},\dots,a_{i_{m-1}}) \mid \text{there is $(b_0,\dots,b_{n-1})\in A$ with $b_{i_j} = a_{i_j}$ for $j<m$} \}.
\]
This is a kind of projection, with the possibility of permuting and repeating components. If $\pi$ is a permutation of $n$, then
\begin{align*}
    \Pr_{(\pi(0),\dots,\pi(n-1))}(A) &= \{ (a_{\pi(0)},\dots,a_{\pi(n-1)}) \mid (a_0,\dots,a_{n-1})\in A \} \\
    &= \{ (a_0,\dots,a_{n-1}) \mid (a_{\pi^{-1}(0)},\dots,a_{\pi^{-1}(n-1)})\in A \}.
\end{align*}

If $\alpha$ is a possibly infinite ordinal and $a$ is a function with domain $\alpha$, we call $a$ an $\alpha$-sequence and write $a = (a_i)_{i<\alpha}$, where $a(i) = a_i$, similar to tuples. 

If $\A$ is a structure, we do not distinguish between $\A$ and the domain of $\A$ when there is no risk of confusion, i.e. we write e.g. $a\in\A$ when $a$ is an element of the domain of $\A$ and we write $|\A|$ for the cardinality of the domain. We also denote the domain of $\A$ by $\dom(\A)$ when necessary for clarity.

If $\tau$ is a signature and $\A$ and $\B$ are elementarily equivalent $\tau$-structures, i.e. satisfy the same $\tau$-sentences of first-order logic, then we write $\A\equiv\B$. If $L$ is another logic and $\A$ and $\B$ satisfy the same $\tau$-sentences of $L$, we write $\A\equiv_L\B$.

We denote the powerset of a set $I$ by $\Pow(I)$. By $\Pow^+(I)$ we mean $\Pow(I)\setminus\{\emptyset\}$. By $\R(I)$ we mean the collection of all relations on $I$, i.e. the set $\bigcup_{n<\omega}\Pow(I^n)$.

\subsection{Existential Second-Order Logic}
We assume the reader is familiar with the syntax and semantics of second-order logic ($\sol$).
We refer the reader to the last chapter of \cite{enderton2001logic} and to \cite{vanbenthem2001higher} and \cite{sep-logic-higher-order} for an overview of basic definitions and results.

If $\phi\in\sol$ does not contain any second-order quantifiers, we say that $\phi$ is first-order, even if it has free second-order variables. If no first-order variable occurs free in $\phi$ and $\phi$ is first-order, we say that $\phi$ is a first-order sentence; with this terminology, a first-order sentence (of $\sol$) may contain free second-order variables. If $\phi$ does not contain free (first- or second-order) variables, we say that $\phi$ is a (second-order) sentence. As is customary, we write $\phi(x_0,\dots,x_{n-1},X_0,\dots,X_{m-1})$ when the free variables of $\phi$ are among $x_0,\dots,x_{n-1},X_0,\dots,X_{m-1}$.

\begin{fact}
    Every formula of $\sol$ is equivalent to a formula of the form
    \[
        Q_0 X_0\dots Q_{m-1} X_{m-1}\phi,
    \]
    where $Q_i\in\{\exists,\forall\}$ for $i<m$, each $X_i$ is either a relation or function variable and $\phi$ is first-order.
\end{fact}

We denote by $\eso$ the existential fragment of $\sol$, i.e., the set of formulas of $\sol$ that are equivalent to a formula of the form
\[
    \exists X_0\dots\exists X_{m-1}\phi
\]
for first-order $\phi$.

The following properties of $\eso$ are well known.

\begin{lemma}\label{Coding many predicates with one}\quad
    \begin{enumerate}
        \item Let $X_0,\dots,X_{n-1}$ be second-order variables. Then there is a relation variable $X$ such that for any first-order $\tau$-sentence $\phi(X_0,\dots,X_{n-1})$, there is a first-order $\tau$-sentence $\phi^*(X)$ such that
        \[
            \exists X_0\dots\exists X_{n-1}\phi \equiv \exists X\phi^*.
        \]
        \item $\eso$ is closed under conjunction, disjunction and first-order quantifiers.
    \end{enumerate}
\end{lemma}

It is easy to show that $\eso$ satisfies the compactness theorem and the following version of the Löwenheim--Skolem theorem: if $T$ is an existential second-order theory with infinite models, then for any infinite $\kappa\geq|\tau|$, $T$ has a model of cardinality $\kappa$. To see this, let $T$ be an $\eso$-theory with infinite models, and let $\{ \exists X_i\alpha_i(X_i) \mid i<|\tau|+\aleph_0\}$, for $\alpha_i$ first-order, axiomatize $T$. Fix an infinite cardinal $\kappa\geq|\tau|$, and let $\A$ be an infinite model of $T$. Then there are relations $R_i\subseteq\A^{\ar(X_i)}$ such that $\A\models\alpha_i(R_i)$ for all $i<|\tau|+\aleph_0$. Let $\hat\A$ be an expansion of $\A$ by fresh predicate symbols $S_i$ with $\ar(S_i)=\ar(R_i)$ such that $S_i^{\hat\A}=R_i$. Then $\Th_\fol(\hat\A)$ has, by the Löwenheim--Skolem theorem of first-order logic, a model $\hat\B$ of cardinality $\kappa$. Then, letting $\B=\hat\B\restriction\tau$, we have $\B\models\exists X_i\alpha_i(X_i)$ for every $i<|\tau|+\aleph_0$, as witnessed by $S_i^{\hat\B}$. We stress that this does not contradict Lindström's theorem, as $\eso$ is not closed under negation, which is a requirement imposed on an abstract logic in the statement of the theorem.

We wish to define the concept of completeness for theories in $\eso$. Usually, one would define a complete theory to be one that semantically (or syntactically if there is a proof system available) entails, for every sentence $\phi$, either $\phi$ or $\neg\phi$. Since $\eso$ is a positive logic, this option is not viable. We can, however, give a definition that does not use negation and is equivalent for logics that are closed under negation.

\begin{definition}\label{definition: complete eso theory}
    Let $T$ be an existential second-order theory, i.e. a set of sentences of $\eso$. We say that $T$ is \emph{complete} if all of its models are $\eso$-equivalent, i.e. for all $\A,\B\models T$ and existential second-order sentences $\phi$, we have
    \[
        \A\models\phi \iff \B\models\phi.
    \]
\end{definition}

It may well be that the set of $\eso$-sentences true in a structure $\A$, which we denote by $\ESOTh(\A)$, is not complete in the above sense. For instance, if $\A$ is a $\{P\}$-structure, where $P$ is a unary predicate, such that $|P^\A| = \aleph_1$ and $|\A\setminus P^\A| = \aleph_0$, then $\ESOTh(\A)$ does not contain the sentence expressing that there is a bijection between $P^\A$ and its complement, but by the Löwenheim--Skolem theorem $\ESOTh(\A)$ has a countable model, and such a model will satisfy said sentence. It turns out that a complete first-order theory has a unique $\eso$-completion, consisting of every $\eso$-sentence that is consistent with the first-order theory. We will discuss this in \cref{Section: Complete Theories}. We  point out the following corollary on complete $\eso$-theories, which immediately follows by the previous Löwenheim--Skolem theorem.

\begin{corollary}
    Let $T$ be a complete $\tau$-theory in $\eso$, with infinite models. Then $T$ has infinite models in all cardinalities $\geq|\tau|$.
\end{corollary}

We conclude this section by making the following observation, akin to the Łos--Vaught test of first-order logic.

\begin{proposition}
    Let $T$ be an existential second-order theory. If $T$ does not have models of size $<|\tau|+\aleph_0$ and is categorical in every cardinality in which it has a model, then $T$ is complete.
\end{proposition}
\begin{proof}
    Let $\A,\B\models T$ and let $\A\models\phi$. Since $|\A|,|\B|\geq|\tau|+\aleph_0$, by Löwenheim--Skolem, there is $\A'$ with $|\A'|=|\B|$ and $\A'\models T\cup\{\phi\}$. As $T$ is $|\B|$-categorical, $\A'\cong\B$, and hence $\B\models\phi$.
\end{proof}

\subsection{Team Semantics} Next we recall the concept of a team and the team-semantic interpretations of first-order connectives and quantifiers. We refer the reader to \textcite{Vaananen2007-VNNDLA,galliani2012inclusion,gradel2013dependence} for an introduction to team semantics and proofs of the main facts recalled here.

Let $\tau$ be a vocabulary, $\A$ a $\tau$-structure and $D\subseteq\{v_i \mid i<\omega\}$ a set of (first-order) variables. An \emph{assingment} of $\A$ with domain $D$ is a function $D\to\A$. A team of $\A$ with domain $D$ is a set of assignments of $\A$ with domain $D$, i.e. a subset of $\A^D$. We usually take $D$ to be finite, but sometimes teams with infinite domains may be of interest. When given a team $X$, we denote its domain by $\dom(X)$, similarly to domains of functions.

If $s$ is an assignment of $\A$ with $x_0,\dots,x_{n-1}\in\dom(s)$ and $t(x_0,\dots,x_{n-1})$ is a $\tau$-term, we write $s(t)$ as a shorthand for $t^\A(s(x_0),\dots,s(x_{n-1}))$. If $\vx=(x_0,\dots,x_{n-1})$ is a tuple of variables, we write $s(\vx)$ for $(s(x_0),\dots,s(x_{n-1}))$. If $X$ is a team and $x_0,\dots,x_{n-1}\in\dom(X)$, then we denote by $X[\vx]$ the relation
\[
    \{ s(\vx) \mid s\in X \}.
\]
Conversely, if $R$ is an $n$-ary relation, we denote by $\team(R)$ the team
\[
    \{ \{ (v_i, a_i) \mid i<n \} \mid (a_0,\dots,a_{n-1})\in R \}
\]
with domain $\{v_i \mid i<n\}$.

If $s$ is an assignment of $\A$, $a\in\A$ and $x$ is a variable, we denote by $s(a/x)$ the assignment $s'$ with domain $\dom(s)\cup\{x\}$ such that
\[
    s'(y) =
    \begin{cases}
        a & \text{if $y = x$}, \\
        s(y) & \text{otherwise}.
    \end{cases}
\]
If $X$ is a team of $\A$, then by $X(a/x)$ we denote the team $Y$ of $\A$ with domain $D\coloneqq \dom(X)\cup\{x\}$ such that
\[
    Y = \{ s(a/x) \mid s\in X \}.
\]
If $F$ is a function $X\to\Pow^+(\A)$, we denote by $X(F/x)$ the team
\[
    \{ s(a/x) \mid \text{$s\in X$ and $a\in F(s)$} \}.
\]
We call $F$ a supplement function and $X(F/x)$ the \emph{supplementation} of $X$ by $F$. If $F(s) = \A$ for all $s\in X$, then we denote the supplemented team by $X(\A/x)$ and call it the \emph{duplication} of $X$. If $D\subseteq\dom(X)$, we denote by $X\restriction D$ the team
\[
    \{s\restriction D \mid s\in X\}.
\]

Given a vocabulary $\tau$, we consider the syntax of $\tau$-formulas of first-order logic $\fol$ to be given by the grammar
\[
	\phi \Coloneqq t=t'  \mid \neg t=t' \mid R(t_0,\dots,t_{n-1}) \mid \neg R(t_0,\dots,t_{n-1}) \mid  (\phi \land \phi) \mid (\phi\lor\phi)  \mid \exists x \phi \mid \forall x \phi,	
\]
where $t$, $t'$ and $t_i$ are $\tau$-terms and $R\in\tau$ is an $n$-ary relation symbol. Note that we readily assume a negation normal form for our first-order formulas. The team semantics of the operators above is the following.

Let $\A$ be a $\tau$-structure and $X$ a team of $\A$ with $x_0,\dots,x_{n-1}\in \dom(X)$. For a first-order formula $\phi(x_0,\dots,x_{n-1})$, $\A\models_X\phi$ is defined as follows.
\begin{enumerate}
    \item $\A\models_X t=t'$ if $s(t)=s(t')$ for all $s\in X$.
    \item $\A\models_X \neg t=t'$ if $s(t)\neq s(t')$ for all $s\in X$.
    \item $\A\models_X R(t_0,\dots,t_{n-1})$ if $(s(t_0),\dots,s(t_{n-1})\in R^{\A}$ for all $s\in X$.
    \item $\A\models_X \neg R(t_0,\dots,t_{n-1})$ if $(s(t_0),\dots,s(t_{n-1})\notin R^{\A}$ for all $s\in X$.
    \item $\A\models_X \psi \land \chi$ if $\A\models_X \psi$ and $\A\models_X \chi$.
    \item $\A\models_X \psi \vee \chi$ if there are $Y,Z\subseteq X$ such that $ Y\cup Z = X$,  $\A\models_Y \psi$ and $\A\models_Z \chi$.
    \item $\A\models_X \exists x \psi$ if there is a function $F\colon X\to \Pow^+(\A)$ such that $ \A\models_{X(F/x)} \psi$.
    \item $\A\models_X\forall x \psi $ if $\A\models_{X(\A/x)} \psi$. 
\end{enumerate}

We stress that the team-semantic definition of the classical connectives and quantifiers does not provide us with much new when it comes to just first-order logic, in the sense that formulas of first-order logic are \emph{flat}.

\begin{fact}[Flatness]
    For any $\A$, $X$ and $\phi(x_0,\dots,x_{n-1})\in\fol$,
    \[
        \A\models_X\phi \iff \text{$\A\models\phi(s(x_0),\dots,s(x_{n-1}))$ for all $s\in X$}.
    \]
\end{fact}

\noindent In particular, any sentence is satisfied in a structure $\A$ in the sense of Tarski semantics if and only if it is satisfied in $\A$ in the sense of team semantics.

\subsection{Dependence Logic} 

When working with team semantics, we are generally interested in expanding the syntax of first-order logic by means of new atomic formulas encoding different relations of (in)dependence between variables. We recall here the semantics of the \emph{dependence atom} $\dep(\vx; \vy)$, the \emph{independence atom} $\vx\perp_{\vz}\vy$, the \emph{inclusion atom} $\vx\subseteq\vy$ and the \emph{exclusion atom} $\vx\mid\vy$.

\begin{enumerate}[resume]
    \item $\A\models_X \dep(\vx; \vy)$ if for all $s,s'\in X$, if $s(\vx)=s'(\vx)$  then $s(\vy)=s'(\vy)$, i.e. there is a function $f\colon\A^{|\vx|}\to\A^{|\vy|}$ such that $f(s(\vx)) = s(\vy)$ for all $s\in X$.
    \item $\A\models_X \vx\subseteq\vy$, where $\vx$ and $\vy$ are tuples of the same length, if for all $s\in X$ there is some $s'\in X$ such that $s(\vx)=s'(\vy)$, i.e. $X[\vx]\subseteq X[\vy]$.
    \item $\A\models_X \vx\mid\vy$, where $\vx$ and $\vy$ are tuples of the same length, if for all $s,s'\in X$ we have $s(\vx)\neq s'(\vy)$, i.e. $X[\vx]\cap X[\vy] = \emptyset$.
    \item $\A\models_X \vx\perp_{\vz}\vy$ if for all $s,s'\in X$ such that $ s(\vz)=s'(\vz)$, there is $s''\in X$ such that $s''(\vx\vz) = s(\vx\vz)$ and $s''(\vy)=s'(\vy)$, i.e. $X[\vx\vz\vy] = X[\vx\vz]\bowtie_{|\vz|}X[\vz\vy]$.
\end{enumerate}
As a special case of the dependence atom, we have the \emph{constancy atom} $\dep(\emptyset;\vx)$ that is true in a team $X$ if for all $s,s'\in X$, $s(\vx) = s'(\vx)$, i.e. $|X[\vx]|\leq 1$. We denote the constancy atom simply by $\dep(\vx)$.

Note that adding any of the aforementioned atoms to the syntax of first-order logic results in a non-flat logic. If $C\subseteq\{\dep(\dots), \perp_c, \subseteq, \mid \}$, we denote by $\fol(C)$ the logic resulting from adding the atoms in $C$ to the syntax of $\fol$. Below we list some well-known properties of these logics.

\begin{fact}
    Let $\A$ be a $\tau$-structure and $X$ and $Y$ teams of $\A$ with the same domain.
	\begin{enumerate}
	    \item \emph{Locality}: For any formula $\phi$ of $\fol( \dep(\dots), \perp_c, \subseteq, | )$, we have $\A\models_{X} \phi$ if and only if $\A\models_{X\restriction \Fv(\phi)} \phi$.
		\item \emph{The empty team property}: For any formula $\phi$ of $\fol( \dep(\dots), \perp_c, \subseteq, | )$ we have $\A\models_{\emptyset} \phi$.
		\item \emph{Downwards closure}: For any formula $\phi$ of $\fol( \dep(\dots), | )$, if $\A\models_X \phi$ and $Y\subseteq X$, then $\A\models_Y \phi$.
		\item \emph{Union-closure}: For any formula $\phi$ of $\fol(\subseteq )$, if $\A\models_{X} \phi$ and $\A\models_{Y} \phi$, then $\A\models_{X\cup Y} \phi$.
	\end{enumerate}    
\end{fact}

By locality, whenever $\phi$ is a sentence of $\fol( \dep(\dots), \perp_c, \subseteq, | )$, we have
\[
    \A\models_X\phi \iff \A\models_{X\restriction\emptyset}\phi \iff \A\models_{\{\emptyset\}}\phi
\]
for any $\A$ and nonempty $X$. We write $\A\models\phi$ for $\A\models_{\{\emptyset\}}\phi$ whenever $\phi$ is a sentence.

It is well known that in $\foil$ one can define all the other aforementioned atoms, i.e. the logics $\foil$ and $\fol(\dep(\dots), \perp_c, \subseteq, \mid)$ are equiexpressive. In fact, any property of teams definable in $\eso$ (modulo the empty team property) can be defined in $\foil$ and, conversely, any formula of $\foil$ can be defined in $\eso$, in the following sense.

\begin{lemma}[Galliani, Grädel--Väänänen~\cite{gradel2013dependence, galliani2012inclusion}]\label{Translation between independence logic and ESO}
    Fix a finite set $D = \{x_0,\dots,x_{n-1}\}$ of variables.
    \begin{enumerate}
        \item For every $\tau$-formula of $\phi(x_0,\dots,x_{n-1})$ of $\foil$, there is a $\tau$-formula $\chi(R)$ of $\eso$ with no free first-order variables and only one free $n$-ary second-order variable $R$, such that for every $\tau$-structure $\A$ and team $X$ of $\A$ with $D\subseteq\dom(X)$,
        \[
            \A\models_X\phi \iff \A\models\chi(X[x_0,\dots,x_{n-1}]).
        \]
        We call $\chi(R)$ a \emph{translation} of $\phi$ to $\eso$.

        \item For every $\tau$-formula $\chi(R)$ of $\eso$ with no free first-order variables and only one free $n$-ary second-order variable $R$, there is a $\tau$-formula $\phi(x_0,\dots,x_{n-1})$ of $\foil$ such that for every $\tau$-structure $\A$ and team $X$ of $\A$ with $D\subseteq\dom(X)$,
        \[
            \A\models_X\phi \iff \A\models\chi^+(X[x_0,\dots,x_{n-1}]),
        \]
        where $\chi^+(R) = \exists x_0\dots\exists x_{n-1} R(x_0,\dots,x_{n-1})\to\chi(R)$. We call $\phi$ a \emph{translation} of $\chi(R)$ to $\foil$.
    \end{enumerate}
\end{lemma}

Finally, we define what it means for a theory in $\foil$ to be complete.

\begin{definition}\label{definition: complete foil-theory and first-order-completeness}
    Let $T$ be an $\foil$-theory.
    \begin{enumerate}
        \item We say that $T$ is \emph{first-order-complete} if for all $\A,\B\models T$, we have $\A\equiv\B$.
        \item We say that $T$ is \emph{complete} if its $\eso$-translation is, i.e. if for all $\A,\B\models T$ and sentences $\phi$ of $\foil$,
        \[
            \A\models\phi \iff \B\models\phi.
        \]
    \end{enumerate}
\end{definition}

\subsection{The Logic $\fot$}

First-order Team Logic, $\fot$, was introduced by \textcite{MR4594292} to capture exactly the team properties definable in first-order logic (modulo the empty team property) the same way $\foil$ captures existential second-order team properties. In other words, $\fot$ with team semantics corresponds to $\fol$ with Tarski semantics exactly the same way in which $\foil$ with team semantics corresponds to $\eso$ with Tarski semantics. We can view $\fot$ here as a small extension of first order logic whose expressive power lies between $\fol$ and $\foil$. In particular, since $\fot$ captures all elementary team properties, it is powerful enough to express the most common dependence atoms, for instance those defined in the previous section.

Given a vocabulary $\tau$, the syntax of $\tau$-formulas of $\fot$ is given by the grammar
\[
    \phi \Coloneqq \lambda \mid \vx\subseteq\vy \mid \dep(\vx) \mid \cneg\phi \mid (\phi\land\phi) \mid (\phi\ivee\phi) \mid \existsone x\phi \mid \forallone x\phi,
\]

\noindent where $\lambda$ is a first-order atomic formula and the \emph{weak classical negation} $\cneg$, the \emph{weak disjunction} $\ivee$, and the \emph{weak quantifiers} $\existsone$ and $\forallone$ are interpreted as follows:

\begin{enumerate}[resume]
    \item $\A\models_X \phi \ivee \psi$ if $\A\models_X \phi$ or $\A\models_X \psi$.
    \item $\A\models_X \cneg \phi$ if $X=\emptyset$ or $\A\nmodels_X \phi$.
    \item $\A\models_X \existsone x \psi$ if there is an element $a\in\A$ such that $\A\models_{X(a/x)} \psi$.
    \item $\A\models_X \forallone x \psi$ if for every element $a\in\A$ we have $\A\models_{X(a/x)} \psi$.
\end{enumerate}

In \cite{MR4594292}, the syntax of $\fot$ excludes the constancy atom $\dep(\vx)$. However, our definition results in an equiexpressive logic, as the constancy atom can be defined via the equivalence
\[
    \dep(x_0,\dots,x_{n-1}) \equiv \existsone y_0\dots\existsone y_{n-1}\bigwedge_{i<n}x_i = y_i.
\]
Note that as quantifiers seem to be required here, at the level of quantifier-free formulas there may be a difference in expressive power between the two versions of $\fot$.

We denote by $\phi\wcimp\psi$ the formula $\cneg\phi\ivee\psi$ and call it the \emph{weak classical implication} and by $\phi\wcequiv\psi$ the formula $(\phi\wcimp\psi)\land(\psi\wcimp\phi)$ and call it the \emph{weak classical equivalence}. Clearly, for a nonempty team $X$, $\A\models_X \phi\wcimp\psi$ if and only if
\[
    \A\models_X \phi \implies \A\models_X \psi.
\]
Note that the empty constancy atom $\dep(\emptyset)$ is true in any team, and hence its weak classical negation $\cneg\dep(\emptyset)$ is true only in the empty team. We denote these two atomic sentences by $\top$ and $\bot$.

The following is the $\fot$-counterpart of \cref{Translation between independence logic and ESO}.

\begin{lemma}[Kontinen--Yang~\cite{MR4594292}]\label{Translation between FOT and FO}
    Fix a finite set $D = \{x_0,\dots,x_{n-1}\}$ of variables.
    \begin{enumerate}
        \item For every $\tau$-formula $\phi(x_0,\dots,x_{n-1})$ of $\fot$, there is a first-order $\tau$-sentence $\chi(R)$ with a single free $n$-ary second order variable $R$ such that for every $\tau$-structure $\A$ and team $X$ of $\A$ with $D\subseteq\dom(X)$,
        \[
            \A\models_X\phi \iff \A\models\chi(X[x_0,\dots,x_{n-1}]).
        \]
        We call $\chi(R)$ a \emph{translation} of $\phi$ to $\fol$.

        \item For every first-order $\tau$-sentence $\chi(R)$ with a single free $n$-ary second order variable $R$, there is a $\tau$-formula $\phi(x_0,\dots,x_{n-1})$ of $\fot$ such that for every $\tau$-structure $\A$ and team $X$ of $\A$ with $D\subseteq\dom(X)$,
        \[
            \A\models_X\phi \iff \A\models\chi^+(X[x_0,\dots,x_{n-1}]),
        \]
        where $\chi^+(R) = \exists x_0\dots\exists x_{n-1} R(x_0,\dots,x_{n-1})\to\chi(R)$. We call $\phi$ a translation of $\chi(R)$ to $\fot$.
    \end{enumerate}
\end{lemma}

We can conclude that the relationships of the logics mentioned so far when it comes to expressive power is the following:
\[
    \fol \lneq \fot \lneq \foil.
\]

Next we prove some basic but useful facts about $\fot$.

\begin{lemma}\label{teams_elementary_equivalent_as_predicates}
    Let $X$ be a team of $\A$ and $Y$ a team of $\B$, both with domain $\{v_0,\dots,v_{n-1}\}$. Then
    \[
        \A\models_X \phi \iff \B\models_Y \phi
    \]
    for all formulas $\phi(v_0,\dots,v_{n-1})$ of $\fot$ if and only if
    \[
        (\A, X[v_0,\dots,v_{n-1}]) \equiv (\B, Y[v_0,\dots,v_{n-1}]).
    \]
\end{lemma}
\begin{proof}
    Denote $\vx = (v_0,\dots,v_{n-1})$. If $X=Y=\emptyset$, then clearly the claim holds. If $X=\emptyset$ and $Y\neq\emptyset$ (or vice versa), then we have that $\A\models_X\bot$ but $\B\nmodels_Y\bot$. On the other hand, clearly $(\A, X[\vx]) = (\A,\emptyset)\not\equiv(\B, Y[\vx])$, proving our claim. So we may assume that $X\neq\emptyset\neq Y$.

    ``$\Longrightarrow$'': Suppose $(\A, X[\vx])\not\equiv(\B, Y[\vx])$. Then there is a first-order $\tau\cup\{R\}$-sentence $\chi$ such that $(\A, X[\vx])\models\chi$ and $(\B, Y[\vx])\nmodels\chi$. Since $X\neq\emptyset\neq Y$, we have $(\A,X[\vx])\models\chi^+$ and $(\B,Y[\vx])\nmodels\chi^+$. Considering $R$ a second-order variable and letting $\phi(\vx)$ be a translation of $\chi(R)$ to $\fot$, we have $\A\models_X\phi$ but $\B\nmodels_Y\phi$.

    ``$\Longleftarrow$'': If $(\A, X[\vx]) \equiv (\B, Y[\vx])$ and $\phi(\vx)$ is a formula of $\fot$, let $\chi(R)$ be the translation of $\phi$ to $\fol$. Consider $R$ as a relation symbol so that $(\A,X[\vx])$ and $(\B,Y[\vx])$ are $\tau\cup\{R\}$-structures. Then as $X$ and $Y$ are nonempty, we have
    \begin{align*}
        (\A, X[\vx])\models\chi &\iff (\B, Y[\vx])\models\chi,
    \end{align*}
    whence by \cref{Translation between FOT and FO},
    \[
        \A\models_X\phi \iff \B\models_Y\phi. \qedhere
    \]
\end{proof}

\begin{lemma}\label{FOT-formula that characterises being a predicate}
    For any $n$-ary relation symbol $R$ and an $n$-tuple $\vx$ of variables, there is an $\{R\}$-formula $\theta(R,\vx)$ of $\fot$ such that for any $\{R\}$-structure $\A$ and a \emph{nonempty} team $X$ of $\A$,
    \[
        \A\models_X\theta \iff X[\vx] = R^\A.
    \]
\end{lemma}
\begin{proof}
    Clearly $\theta = R(\vx) \land \forallone y_0\dots\forallone y_{n-1} ( R(\vy) \wcimp \vy\subseteq\vx)$ suffices.
\end{proof}

\subsection{Ultraproducts}

Recall that given an index set $I$, an ultrafilter $\U$ on $I$ and nonempty sets $A_i$, $i\in I$, the ultraproduct $\prod_{i\in I}A_i/\U$ of the sets $A_i$ is the quotient set $\{ f/\U \mid f\in\prod_{i\in I}A_i \}$ of $\prod_{i\in I}A_i$ by the equivalence relation
\[
    f\equiv_\U g \iff \{ i\in I \mid f(i) = g(i) \} \in \U,
\]
where $f/\U$ denotes the equivalence class of $f$. If even one of the sets $A_i$ is empty, then the Cartesian product will also be empty and hence so will the ultraproduct. However, when one defines the ultraproduct $\prod_{i\in I}\A_i/\U$ of $\tau$-structures $\A_i$, $i\in I$, the interpretation of a predicate symbol $R\in\tau$ will be the set
\[
    \{(f_0/\U,\dots,f_{n-1}/\U) \mid \{i\in I \mid (f_0(i),\dots,f_{n-1}(i))\in R^{\A_i}\}\in\U\}.
\]
This set, which we naturally think of as the ultraproduct of the sets $R^{\A_i}$, is empty only when $R^{\A_i}$ is empty for $\U$-many indices $i$. Given a context of an ultraproduct structure $\A = \prod_{i\in I}\A_i/\U$, we generalize the notion of an ultraproduct of sets to correspond to the interpretation of relation symbols, i.e. if $A_i\subseteq\A_i^n$ for all $i\in I$, then we denote by $\prod_{i\in I}A_i/\U$ the set
\[
    \{ (f_0/\U,\dots,f_{n-1}/\U) \in \A^n \mid \{i\in I \mid (f_0(i),\dots,f_{n-1}(i))\in A_i\}\in\U \}.
\]
This leads to the definition of the ultraproduct of teams.

\begin{definition}[\textcite{luck2020team}]
    Let $\A$ be the ultraproduct of the structures $\A_i$, $i\in I$, and let $D$ be a set of first-order variables.
    \begin{enumerate}
        \item Given assignments $s_i\colon D\to\A_i$, $i\in I$, we denote by $(s_i)_{i\in I}$ the assignment $s\colon D\to\prod_{i\in I}\A_i$ such that $s(x) = (s_i(x))_{i\in I}$ for all $x\in D$.
        \item Given an assignment $s\colon D\to\prod_{i\in I}\A_i$, we denote by $s/\U$ the assignment $t\colon D\to\A$ such that $t(x)=s(x)/\U$ for all $x\in D$.
        \item Given a team $X_i$ of $\A_i$ with domain $D$ for each $i\in I$, we define their \emph{team ultraproduct} $\prod_{i\in I}X_i/\U$ as the set of all assignments $s\colon D\to\A$ such that there are $s_i\colon D\to\A_i$, $i\in I$, with $s = (s_i)_{i\in I}/\U$ and $\{i\in I \mid s_i\in X_i\}\in\U$.
    \end{enumerate}
\end{definition}

Note that with these definitions, $\left(\prod_{i\in I}X_i/\U\right)[\vx] = \prod_{i\in I}X_i[\vx]/\U$ for any variable tuple $\vx$.

Recall that the classical fundamental theorem of ultraproducts, also known as Łoś' theorem, is the following: if $f_0,\dots,f_{n-1}\in\prod_{i\in I}\A_i$ and $\phi(x_0,\dots,x_{n-1})$ is a first-order formula, then
\[
    \{i\in I \mid \A_i\models\phi(f_0(i),\dots,f_{n-1}(i)) \}\in\U \iff \prod_{i\in I}\A_i/\U \models\phi(f_0/\U,\dots,f_{n-1}/\U).
\]
It was proved by \textcite{luck2020team} that first-order definable properties of teams are preserved in ultraproducts, which together with \cref{Translation between FOT and FO} gives the following theorem. 

\begin{theorem}[Łoś' Theorem of $\fot$]\label{Los theorem of FOT}
    If $X_i$ is a team of $\A_i$ with domain $D$ for all $i\in I$ and $\phi$ is a formula of $\fot$ with $\Fv(\phi)\subseteq D$, then
    \[
        \{ i\in I \mid \A_i\models_{X_i}\phi \}\in\U \iff \prod_{i\in I}\A_i/\U\models_{\prod_{i\in I}X_i/\U}\phi.
    \]
\end{theorem}

\noindent A version for $\foil$ was proved by \textcite{puljujärvi2022compactness}:

\begin{theorem}[Łoś' Theorem of $\foil$]\label{Los theorem of FOIL}
    If $X_i$ is a team of $\A_i$ with domain $D$ for all $i\in I$ and $\phi$ is a formula of $\foil$ with $\Fv(\phi)\subseteq D$, then
    \[
        \{ i\in I \mid \A_i\models_{X_i}\phi \}\in\U \implies \prod_{i\in I}\A_i/\U\models_{\prod_{i\in I}X_i/\U}\phi.
    \]
\end{theorem}

\noindent Note that in independence logic one does not obtain equivalence between a formula being satisfied $\U$-often and it being satisfied in the ultraproduct; this reflects the positive nature of $\eso$. In fact, a universal second-order sentence that is preserved under ultraproducts is actually definable in first-order logic~\cite{keisler-thesis}. Furthermore, a logic that satisfies the stronger form of Łoś' theorem is equiexpressive with first-order logic~\cite{sgro1977maximal}. Fortunately, the missing direction of the equivalence is unnecessary for proving the compactness theorem, which we obtain as a corollary. We say that a set $\Gamma$ of formulas of $\foil$ is satisfiable if there is a structure $\A$, and a \emph{nonempty} team $X$ of $\A$ whose domain contains all the free variables of $\Gamma$, such that $\A\models_X\Gamma$. Notice that, in general, satisfaction for open formulas of $\foil$ does not boil down to satisfaction for sentences (as noted in~\cite{puljujärvi2022compactness}).

\begin{corollary}
    If $\Gamma$ is a set of formulas of $\foil$ and each finite $\Gamma'\subseteq\Gamma$ is satisfiable, then $\Gamma$ is satisfiable.
\end{corollary}

In the sequel, the following classic result, usually dubbed the Keisler--Shelah theorem, will be of much use.

\begin{theorem}[\textcite{keisler-thesis},~\textcite{MR0297554}]
    Let $\A$ and $\B$ be $\tau$-structures. If $\A\equiv\B$, then there is a cardinal $\kappa$ and an ultrafilter $\U$ on $\kappa$ such that $\A^\kappa/\U \cong \B^\kappa/\U$.
\end{theorem}

We conclude this section by reviewing the notion of a limit ultrapower, which was introduced by \textcite{MR0148547} as a kind of generalization of ultrapowers.

\begin{definition}
    Let $I$ be a set, $\cF$ an ultrafilter on $I$ and $\cG$ a filter on $I^2$.
    \begin{enumerate}
        \item For $f\in\A^I$, we denote by $\eq(f)$ the set of all pairs $(i,j)\in I^2$ such that $f(i) = f(j)$.
        \item Let $A$ be a set and denote $B = A^I/\cF$. By $B|\cG$ we denote the set
        \[
            \{ f/\cF \in B \mid \text{$\eq(g)\in\cG$ for some $g \equiv_\cF f$} \}.
        \]
        \item Let $\A$ be a $\tau$-structure and let $\B = \A^I/\cF$. Then we denote by $\B|\cG$ the substructure of $\B$ generated by the set $\dom(\B)|\cG$ and call it a \emph{limit ultrapower} of $\A$.
    \end{enumerate}
\end{definition}

\noindent The limit ultrapower $\B|\cG$ is a substructure of the ultrapower $\B = \A^I/\cF$ that consists of equivalence classes of those sequences $f\in\A^I$ that are ``almost constant'' according to the filter $\cG$. In the definition we do not require $\cG$ to be proper; if $\cG = \Pow(I^2)$, we have $\B|\cG = \B$, so in this sense limit ultrapowers generalize ultrapowers. It is proved in \cite{MR0148547} that $\B|\cG$ is always an elementary substructure of $\B$, and that the ultrapower embedding $a\mapsto (a)_{i\in I}/\cF$ is an elementary embedding $\A\to\B|\cG$.

Limit ultrapowers are connected to complete embeddings, whose definition we now recall.

\begin{definition}\label{Definition: Complete embedding}
    Let $\A$ and $\B$ be $\tau$-structures. An embedding $f\colon\A\to\B$ is \emph{complete} if for every $\tau'\supseteq\tau$ and $\tau'$-expansion $\hat\A$ of $\A$ there is a $\tau'$-expansion $\hat\B$ of $\B$ such that $f$ is an elementary embedding $\hat\A\to\hat\B$.
\end{definition}

\noindent Clearly the ultrapower embedding $\A\to\A^I/\cF$ is complete. It turns out that this is a special case of the more general fact that the ultrapower embedding to a limit ultrapower is complete. In fact, every complete embedding is essentially a limit ultrapower embedding.

\begin{theorem}[\textcite{MR0148547}]\label{Complete embeddings are limit ultrapower embeddings}
    The following are equivalent for $f\colon\A\to\B$.
    \begin{enumerate}
        \item $f$ is a complete embedding.
        \item There are $I$, $\cF$ and $\cG$ and an isomorphism $\pi\colon(\A^I/\cF)|\cG\to\B$ such that $f = \pi\circ\iota$, where $\iota$ is the ultrapower embedding $\A\to(\A^I/\cF)|\cG$.
    \end{enumerate}
\end{theorem}

\section{Team Maps}\label{Section_maps}

In traditional model theory, structures are usually not studied in isolation, but in connection with several variants of the substructure relation as well as different kinds of mappings between structures. In particular, a key concept which is fundamental in the setting of first-order model theory is the notion of an \emph{elementary substructure} (and that of an \emph{elementary embedding}). Quite surprisingly, no obvious counterpart has been introduced in the framework of logics in team semantics. In this section we try to fill this gap by studying several different notions of team maps.

In particular, our goal will be (as we detail more explicitly in Section~\ref{Section_AEC} and Section~\ref{Section_Monster})  to understand how much of the properties of embeddings and of elementary embeddings can be replicated in the setting of $\foil$ and $\eso$. As a matter of fact, if, in the definition of an elementary embedding (resp. embedding), one removes the requirement that it is a total function, then one obtains the notion of a partial elementary map (resp. partial isomorphism). It turns out that, in the setting of $\foil$ and $\eso$, it is likewise important to work with partial maps and not simply with total ones.  However, in a context where elements are replaced by relations as the smallest unit of interest, the notion of a ``partial elementary map'', or even a ``partial isomorphism'', becomes a bit more complicated and one has to be very careful when dealing with ``substructures''. Most of the time, it is not sufficient to map relations into their pointwise images under a map that has been defined on elements but, rather, to consider more complicated mappings. We deal with these difficulties in this section: in \cref{definition.team.maps}, we identify several notions of partial and total maps defined on relations and preserving some properties definable in $\foil$ and $\eso$ and, in particular, define team-semantic versions of partial isomorphisms and partial elementary maps, which play a crucial role in the following sections.

For the rest of the paper (except for \cref{Lindstrom}), we identify a relation $X$ with the corresponding team $\team(X)$ with domain $\{v_0,\dots,v_{\ar(X)-1}\}$. If $\phi(v_0,\dots,v_{n-1})$ is a formula of $\fot$ or $\foil$ and $X\subseteq\A^n$ for a structure $\A$, then by $\A\models_X\phi$ we mean $\A\models_{\team(X)}\phi$. If $X\subseteq\A^n$ and $a\in\A$, then by $X(a/n)$ we denote the $(n+1)$-ary relation $X\times\{a\}$. Note that $\team(X(a/n)) = \team(X)(a/v_n)$. Also note that $\emptyset(a/0) = \emptyset$ and $\{\emptyset\}(a/0) = \{a\}$. Given a structure $\A$, we denote by $\Delta_\A$ the diagonal set $\{(a,a) \mid a\in\A\}$, i.e. the identity relation ($a = b$ if and only if $(a,b)\in\Delta_\A$).

\begin{definition}\label{definition.team.maps}
    Let $\A$ and $\B$ be $\tau$-structures.
    \begin{enumerate}
        \item Given a set $\X\subseteq\R(\A)$, we denote by $\cl(\X;\A)$ the $\subseteq$-least set $\Y\subseteq\R(\A)$ such that $\X\subseteq\Y$ and:
        \begin{enumerate}[label = (C\arabic*)]
            \item\label{Domain of partial isomorphism contains bottom and top}
            $\emptyset\in\Y$ and $\A^n\in\Y$ for all $n<\omega$,
            \item\label{Domain of partial isomorphism closed under meets}
            if $X,Y\in\Y$, then $X\cap Y\in\Y$,
            \item\label{Domain of partial isomorphism closed under products}
            if $X,Y\in\Y$, then $X\times Y\in\Y$,
            \item\label{Domain of partial isomorphism closed under projections}
            if $X\in\Y$ is $n$-ary and $\vec\imath\in n^{<\omega}$, then $\Pr_{\vec\imath}(X)\in\Y$, and
            \item\label{Domain of partial isomorphism contains interpretations of vocabulary}
            $\Delta_\A,R^\A,F^\A,\{c^\A\}\in\Y$ for all relation symbols $R\in\tau$, function symbols $F\in\tau$ and constant symbols $c\in\tau$.
        \end{enumerate}
        If the structure $\A$ is clear from the context (which it almost always is), we simply write $\cl(\X)$ for $\cl(\X;\A)$.
        \item We say that $\f$ is a \emph{partial team map} $\A\to\B$ if it is a partial function $\R(\A)\to\R(\B)$ such that
        \begin{itemize}
            \item $\f$ preserves arities, i.e. if $X\subseteq\A^n$, then $\f(X)\subseteq\B^n$, and
            \item the domain and range of $\f$ satisfy the closure properties \ref{Domain of partial isomorphism contains bottom and top}--\ref{Domain of partial isomorphism contains interpretations of vocabulary}, i.e. $\dom(\f) = \cl(\dom(\f);\A)$ and $\ran(\f) = \cl(\ran(\f);\B)$.
        \end{itemize}
        If $\f$ maps all singletons from its domain to singletons and $\{\vec a\}\in\dom(\f)$, we write $\f(\vec a)$ for the unique element inhabiting the singleton $\f(\{\vec a\})$. We write $\f\restriction_n$ for the function $\f\restriction(\dom(\f)\cap\Pow(\A^n))$.

        \item We say that a partial team map $\f\colon\A\to\B$ is a \emph{partial team isomorphism} if the following conditions are satisfied:
        \begin{enumerate}[label=(PI\arabic*)]
            \item\label{Partial isomorphisms preserve bottom and top}
            $X = \emptyset$ if and only if $\f(X) = \emptyset$, and $X=\A^n$ if and only if $\f(X) = \B^n$ for all $n<\omega$,
            \item\label{Partial isomorphisms preserve atoms}
            for all $X\in\dom(\f)$, $X$ is a singleton if and only if $\f(X)$ is,
            \item\label{Partial isomorphisms preserve products}
            for all $X,Y\in\dom(\f)$, $\f(X\times Y) = \f(X)\times\f(Y)$,
            \item\label{Partial isomorphisms preserve projections}
            for all $n$-ary $X\in\dom(\f)$ and $\vec\imath\in n^{<\omega}$, $\f(\Pr_{\vec\imath}(X)) = \Pr_{\vec\imath}(\f(X))$,
            \item\label{Partial isomorphisms preserve order}
            for all $X,Y\in\dom(\f)$, $X\subseteq Y$ if and only if $\f(X)\subseteq\f(Y)$, and
            \item\label{Partial isomorphisms preserve structure}
            $\f(\Delta_\A) = \Delta_\B$, $\f(R^\A) = R^\B$ for any relation or function symbol $R\in\tau$ and $\f(c^\A) = c^\B$ for any constant symbol $c\in\tau$.
        \end{enumerate}
        If $\f$ happens to be total, then we call $\f$ a \emph{team embedding}. If $\f$ is such that for any $a\in\A$, $\{a\}\in\dom(\f)$, then we call $\f$ \emph{element-total}, and if for any $b\in\B$, $\{b\}\in\ran(\f)$, then we call $\f$ \emph{element-surjective}.

        \item We say that $\f$ is a \emph{team isomorphism} if it is an element-surjective team embedding. If $\pi\colon\A\to\B$ is an ordinary isomorphism, we denote by $\hat\pi$ the team isomorphism $\A\to\B$ defined by
        \[
            \hat\pi(X) = \{ \pi(\vec a) \mid \vec a\in X \}
        \]
        for all relations $X$ of $\A$.

        \item We say that a partial team map $\f\colon\A\to\B$ is a partial \emph{elementary team map} if it satisfies \ref{Partial isomorphisms preserve products}, and for all $\tau$-formulas $\phi(v_0,\dots,v_{n-1})$ of $\fot$ and $n$-ary $X\in\dom(\f)$, we have
        \[
            \A\models_X\phi \iff \B\models_{\f(X)}\phi.
        \]
        If $\f$ happens to be total, then we call $\f$ an \emph{elementary team embedding}.

        \item We say that a partial elementary team map $\f\colon\A\to\B$ is a partial \emph{independence team map} if for all $\tau$-formulas $\phi(v_0,\dots,v_{n-1})$ of $\foil$ and $X\subseteq\A^n$, we have
        \[
            \A\models_X\phi \implies \B\models_{\f(X)}\phi.
        \]
        If $\f$ happens to be total, then we call $\f$ an \emph{independence team embedding}.
    \end{enumerate}
\end{definition}

Note that $\f(\emptyset) = \emptyset$ vacuously holds for any team map $\f$: the empty set is the only relation of every arity, and $\f$ preserves arity of relations. If $\f$ is a partial team isomorphism and $\{a_0\},\dots,\{a_{n-1}\}\in\dom(\f)$, then $\{\vec a\} = \{(a_0,\dots,a_{n-1})\}\in\dom(\f)$ and $\f(\vec a) = (\f(a_0),\dots,\f(a_{n-1}))$.

It is not completely clear from the definition that e.g.\ team embeddings are injective. However, justification for the nomenclature will be found in the coming results (e.g. \cref{lemma: partial team isomorphism properties}).

\begin{example}
    Let $\A$ be a $\tau$-structure, $\kappa$ an infinite cardinal and $\U$ an ultrafilter on $\kappa$. Let $\B = \A^\kappa/\U$. Now $\iota\colon\A\to\B$ defined by $\iota(X) = X^\kappa/\U$ is an elementary team embedding by \cref{Los theorem of FOT} and furthermore an independence team embedding by \cref{Los theorem of FOIL}.
\end{example}

\begin{lemma}\label{lemma: partial team isomorphism properties}
    Let $\f$ be a partial team isomorphism. Then
    \begin{enumerate}
        \item $\f$ is injective,
        \item $\f(X(a/n)) = \f(X)(\f(a)/n)$ for an $n$-ary $X\in\dom(\f)$ and $\{a\}\in\dom(\f)$.
        \item $\f(X\cap Y) = \f(X)\cap\f(Y)$ for all $X,Y\in\dom(\f)$.
    \end{enumerate}
\end{lemma}
\begin{proof}
    We show the third claim.
    Let $X,Y\in\dom(\f)$. If they are of different arity, the intersection is empty and hence mapped to the empty set, so we may assume $\ar(X)=\ar(Y)$. By \ref{Partial isomorphisms preserve order}, as $X\cap Y\subseteq X,Y$, we have $\f(X\cap Y)\subseteq\f(X),\f(Y)$ and hence $\f(X\cap Y)\subseteq\f(X)\cap\f(Y)$. For the converse, note that as $\f(X)\cap\f(Y)\in\ran(\f)$, there is some $Z\in\dom(\f)$ such that $\f(Z)=\f(X)\cap\f(Y)$. Since $\f(Z)\subseteq\f(X)\cap\f(Y)$, we have $\f(Z)\subseteq\f(X),\f(Y)$, so by \ref{Partial isomorphisms preserve order} we have $Z\subseteq X,Y$. But then $Z\subseteq X\cap Y$. Thus $\f(Z)\subseteq\f(X\cap Y)$. But as $\f(Z)=\f(X)\cap\f(Y)$, we have $\f(X)\cap\f(Y)\subseteq\f(X\cap Y)$.
\end{proof}

\begin{lemma}\quad
    \begin{enumerate}
        \item If $\f$ is a partial team isomorphism $\A\to\B$, then $\f^{-1}$ is a partial team isomorphism $\B\to\A$. If $\f$ is a partial elementary team map $\A\to\B$, then $\f^{-1}$ is a partial elementary team map $\B\to\A$.
        \item If $\f\colon\A\to\B$ and $\g\colon\B\to\C$ are partial team isomorphisms with $\ran(\f) = \dom(\g)$, then $\g\circ\f$ is a partial team isomorphism $\A\to\C$. If $\f$ and $\g$ are elementary, then so is $\g\circ\f$. If $\f$ and $\g$ are independence team maps, then so is $\g\circ\f$.
    \end{enumerate}
\end{lemma}
\begin{proof}
    Clear. \qedhere
\end{proof}

\begin{lemma}
    Let $\f\colon\A\to\B$ be a partial team map. Then for any $n,m<\omega$ and $n$-ary and $m$-ary relations $X$ and $Y$ of $\A$, respectively, the following holds: $X\in\dom(\f)$ and $Y\in\dom(\f)$ if and only if for all $k\leq n,m$, $X\bowtie_k Y\in\dom(\f)$. Furthermore, if $\f$ is a partial team isomorphism, then $\f(X\bowtie_k Y) = \f(X)\bowtie_k\f(Y)$ for all $k$.
\end{lemma}
\begin{proof}
    Let $\C$ be an arbitrary structure. We first show that for any $n<\omega$, the $2n$-ary relation
    \[
        \Delta^n_\C \coloneqq \{ \vec a\vec b \mid |\vec a| = |\vec b| = n, \vec a = \vec b \}
    \]
    is expressible in terms of products, projections and $\Delta_\C$. Note that $\Delta^0_\C = \{\emptyset\} = \Pr_{\emptyset}(\Delta_\C)$ and $\Delta^1_\C = \Delta_\C$. Suppose we have shown this for $\Delta^n_\C$. Then
    \begin{align*}
        \Delta^{n+1}_\C &= \{ \vec a a\vec b b \mid |\vec a| = |\vec b| = n, \vec a = \vec b, a=b \} \\
        &= \Pr_{\vec\imath}(\{ \vec a\vec b ab \mid |\vec a| = |\vec b| = n, \vec a = \vec b, a=b \}) \\
        &= \Pr_{\vec\imath}(\Delta^n_\C\times\Delta_\C),
    \end{align*}
    where $\vec\imath = (0,\dots,n-1,2n,n,\dots,2n-1,2n+1)$.
    
    Now for any $(n+k)$-ary and $(k+m)$-ary relations $X$ and $Y$ of $\C$, we have
    \begin{align*}
        X\bowtie_k Y &= \{\vec a\vec b\vec c \mid \vec a\vec b\in X, \vec b\vec c\in Y\} \\
        &= \Pr_{\vec\jmath}((X\times Y)\cap(\C^n\times\Delta^k_\C\times\C^m)),
    \end{align*}
    where $\vec\jmath = (0,\dots,n+k-1,n+2k,\dots,n+2k+m-1)$.

    Now it is clear that $\Delta^k_\A\in\dom(\f)$. It follows that $X\bowtie_k Y\in\dom(\f)$ whenever $X,Y\in\dom(\f)$. On the other hand, if $X\bowtie_0 Y = X\times Y\in\dom(\f)$, then $X$ and $Y$ are easily obtained as projections. If $\f$ is a partial team isomorphism, then clearly $\f(\Delta^k_\A) = \Delta^k_\B$, whence also $\f(X\bowtie_k Y) = \f(X)\bowtie_k\f(Y)$.
\end{proof}

\subsection{Characterizations of Partial Isomorphisms and Elementary Maps}

Although we have defined partial team isomorphisms simply by requiring some specific commutativity conditions, the next lemma and the following proposition show that these are exactly those team maps which preserve the quantifier-free fragment of $\fot$.

\begin{lemma}\label{Partial isomorphisms preserve first-order atomic formulas}
    Let $\f\colon\A\to\B$ be a partial team isomorphism. Then $\f$ preserves atomic first-order formulas, i.e. for any $X\in\dom(\f)$ and atomic first-order formula $\phi(v_0,\dots,v_{n-1})$,
    \[
        \A\models_X\phi \iff \B\models_{\f(X)}\phi.
    \]
\end{lemma}
\begin{proof}

    We begin by showing that for any $\tau$-term $t(v_0,\dots,v_{n-1})$, we have $G_n(t;\A)\in\dom(\f)$ and $\f(G_n(t;\A)) = G_n(t;\B)$, where
    \[
        G_n(t;\C) = \{ (c_0,\dots,c_n)\in\C^{n+1} \mid t^\C(c_0,\dots,c_{n-1}) = c_n \}
    \]
    is the graph of the function $\vec c\mapsto t^\C(\vec c)$. We proceed by induction on $t$.
    \begin{enumerate}
        \item If $t = v_i$, then the function $\vec c\mapsto t^\C(\vec c)$ is the $i^{\text{th}}$ projection. Let $\vec\imath = (0,\dots,n-1,i)$. Then
        \[
            G_n(t;\C) = \{ (c_0,\dots,c_n)\in\C^{n+1} \mid c_i = c_n \} = \Pr_{\vec\imath}(\C^{n+1}).
        \]
        Hence
        \[
            \f(G_n(t;\A)) = \f(\Pr_{\vec\imath}(\A^{n+1})) = \Pr_{\vec\imath}(\f(\A^{n+1})) = \Pr_{\vec\imath}(\B^{n+1}) = G_n(t;\B).
        \]

        \item If $t = c$ for a constant symbol $c\in\tau$, then the function $\vec c\mapsto t^\C(\vec c)$ is the constant map $\vec c \mapsto c$. Then
        \[
            G_n(t;\C) = \{ (c_0,\dots,c_n)\in\C^{n+1} \mid c_n = c^\C \} = \C^n\times\{c^\C\}.
        \]
        Hence
        \begin{align*}
            \f(G_n(t;\A)) &= \f(\A^n\times\{c^\A\}) = \f(\A^n)\times\f(\{c^\A\}) \\
            &= \B^n\times\{c^\B\} = G_n(t;\B).
        \end{align*}

        \item Finally suppose that $t = F(t_0(v_0,\dots,v_{n-1}),\dots,t_{m-1}(v_0,\dots,v_{n-1}))$ for $\tau$-terms $t_i$ and an $m$-ary function symbol $F\in\tau$. By the induction hypothesis, $\f(G_n(t_i;\A)) = G_n(t_i;\B)$ for all $i<m$. First consider the relation
        \[
            X_0(\C) = \{ (t_0^\C(c_0,\dots,c_{n-1}),c_0,\dots,c_{n-1}) \mid c_0,\dots,c_{n-1}\in\C \}.
        \]
        Clearly $X_0(\C) = \Pr_{\vec\imath_0}(G_n(t_0;\C))$ for $\vec\imath_0 = (n,0,\dots,n-1)$, and hence
        \begin{align*}
            \f(X_0(\A)) &= \f(\Pr_{\vec\imath_0}(G_n(t_0;\A))) = \Pr_{\vec\imath_0}(\f(G_n(t_0;\A))) \\
            &= \Pr_{\vec\imath_0}(G_n(t_0;\B)) = X_0(\B).
        \end{align*}
        Then consider the relation
        \[
            X_1(\C) = \{ (t_0^\C(\vec c),t_1^\C(\vec c),c_0,\dots,c_{n-1}) \mid \vec c = (c_0,\dots,c_{n-1})\in\C^n \}.
        \]
        Denoting $\vec\imath_1 = (0,n+1,1,\dots,n)$, we have
        \begin{align*}
            X_1(\C) &= \Pr_{\vec\imath_1}( \{ (t_0^\C(\vec c),c_0,\dots,c_{n-1},t_1^\C(\vec c)) \mid \vec c = (c_0,\dots,c_{n-1})\in\C^n \} ) \\
            &= \Pr_{\vec\imath_1}( X_0(\C) \bowtie_n G_n(t_1;\C) ).
        \end{align*}
        Then
        \begin{align*}
            \f(X_1(\A)) &= \f(\Pr_{\vec\imath_1}( X_0(\A) \bowtie_n G_n(t_1;\A) )) \\
            &= \Pr_{\vec\imath_1}( \f(X_0(\A)) \bowtie_n \f(G_n(t_1;\A)) ) \\
            &= \Pr_{\vec\imath_1}( X_0(\B) \bowtie_n G_n(t_1;\B) ) \\
            &= X_1(\B).
        \end{align*}
        Continuing this way we obtain a description of the relation
        \[
            \hspace{3em} X_{m-1}(\C) = \{ (t_0^\C(\vec c),\dots,t_{m-1}^\C(\vec c),c_0,\dots,c_{n-1}) \mid \vec c = (c_0,\dots,c_{n-1})\in\C^n \}
        \]
        as an element of the closure of $G_n(t_0;\C),\dots,G_n(t_{m-1};\C)$ under projections and join, whence $\f(X_{m-1}(\A)) = X_{m-1}(\B)$. Then, denoting $\vec\jmath_0 = (0,\dots,n-1,n+m)$ and $\vec\jmath_1 = (m,\dots,m+n-1,0,\dots,m-1)$, we have
        \begin{align*}
            \hspace{3em} G_n(t;\C) &= \Pr_{\vec\jmath_0}\left( \{ (\vec c,t_0^\C(\vec c),\dots,t_{m-1}^\C(\vec c),F^\C(t_0^\C(\vec c),\dots,t_{m-1}^\C(\vec c))) \mid \vec c\in\C^n \} \right) \\
            &= \Pr_{\vec\jmath_0}\left(\Pr_{\vec\jmath_1}\big( X_{m-1}(\C)  \big)\bowtie_m F^\C \right),
        \end{align*}
        whence
        \begin{align*}
            \f(G_n(t;\A)) &= \f\left( \Pr_{\vec\jmath_0}\left(\Pr_{\vec\jmath_1}\left( X_{m-1}(\A) \bowtie_m F^\A \right)\right) \right) \\
            &= \Pr_{\vec\jmath_0}\left(\Pr_{\vec\jmath_1}\left( \f(X_{m-1}(\A)) \bowtie_m \f(F^\A) \right)\right) \\
            &= \Pr_{\vec\jmath_0}\left(\Pr_{\vec\jmath_1}\left( X_{m-1}(\B) \bowtie_m F^\B \right)\right) \\
            &= G_n(t;\B).
        \end{align*}
    \end{enumerate}

    Now we show that $\f$ preserves atomic formulas $\phi(v_0,\dots,v_{n-1})$ of first-order logic. Suppose $\phi = R(t_0,\dots,t_{m-1})$ for some $m$-ary relation symbol $R\in\tau$ and $\tau$-terms $t_i(v_0,\dots,v_{n-1})$, $i<m$. Now
    \begin{align*}
        \phi(\C) &\coloneqq \{ (c_0,\dots,c_{n-1})\in\C^n \mid \C\models\phi(c_0,\dots,c_{n-1}) \} \\
        &\,= \{ (c_0,\dots,c_{n-1}) \mid (t_0^\C(c_0,\dots,c_{n-1}),\dots,t_{m-1}^\C(c_0,\dots,c_{n-1}))\in R^\C \} \\
        &\,= \Pr_{(0,\dots,n-1)}( Y^m_n(\C) \bowtie_m R^\C ),
    \end{align*}
    where
    \begin{align*}
        Y^m_n(\C) &\coloneqq \{ (c_0,\dots,c_{n-1},t_0^\C(\vec c),\dots,t_{m-1}^\C(\vec c)) \mid \vec c = (c_0,\dots,c_{n-1})\in\C^n \}.
    \end{align*}
    If we can show that $Y^m_n(\A)\in\dom(\f)$ and $\f(Y^m_n(\A)) = Y^m_n(\B)$, then
    \begin{align*}
        \f(\phi(\A)) &= \f\left( \Pr_{(0,\dots,n-1)}(Y^m_n(\A))\bowtie_m R^\A \right) = \Pr_{(0,\dots,n-1)}( \f(Y^m_n(\A))\bowtie_m\f(R^\A) ) \\
        &= \Pr_{(0,\dots,n-1)}(Y^m_n(\B)\bowtie_m R^\B) \\
        &= \phi(\B).
    \end{align*}
    Then for any $n$-ary $X\in\dom(\f)$,
    \begin{align*}
        \A\models_X\phi &\iff X\subseteq\phi(\A) \iff \f(X) \subseteq \f(\phi(\A)) \\
        &\iff \f(X) \subseteq\phi(\B) \iff \B\models_{\f(X)}\phi,
    \end{align*}
    as desired. So it suffices show that $Y^m_n(\C)$ can be presented in terms of products, projections, intersections and the sets $G(t_i;\C)$. We show this by induction on $m>0$. If $m=1$, then $Y^m_n(\C) = G_n(t_0;\C)$, so we have already proved this. If we have already handled $Y^m_n(\C)$, then
    \begin{align*}
        Y^{m+1}_n(\C) &= \Pr_{\vec\imath}(\{(c_0,\dots,c_{n-1},c_0,\dots,c_{n-1},t_0^\C(\vec c),\dots,t_m^\C(\vec c) \mid \vec c\in\C^n\}) \\
        &= \Pr_{\vec\imath}(\Pr_{\vec\jmath}(Y^m_n(\C)\times G_n(t_m;\C))\cap\Delta^n_\C\times\C^{m+1}),
    \end{align*}
    where $\vec\imath = (0,\dots,n-1,2n,\dots,2n+m)$ (skips the second batch of $c_0,\dots,c_{n-1}$) and $\vec\jmath = (0,\dots,n-1,n+m,\dots,2n+m-1,n,\dots,n+m-1,2n+m)$ (permutes the second batch of $c_0,\dots,c_{n-1}$ to be after the first).

    Then, finally, suppose that $\phi$ is the equation $t_0(v_0,\dots,v_{n-1}) = t_1(v_0,\dots,v_{n-1})$. Now
    \begin{align*}
        \phi(\C) &= \{ (c_0,\dots,c_{n-1})\in\C^n \mid t_0^\C(c_0,\dots,c_{n-1}) = t_1^\C(c_0,\dots,c_{n-1}) \} \\
        &= \Pr_{(0,\dots,n-1)}(G_n(t_0;\C)\cap G_n(t_1;\C)).
    \end{align*}
    Thus $\f(\phi(\A)) = \phi(\B)$. Then we obtain
    \[
        \A\models_X\phi \iff \B\models_{\f(X)}\phi,
    \]
    which concludes the proof. \qedhere
\end{proof}

\begin{proposition}\label{Partial isomorphisms preserve quantifier-free formulas}
    Let $\f\colon\A\to\B$ be a team map. Then the following are equivalent.
    \begin{enumerate}
        \item $\f$ is a partial team isomorphism.

        \item $\f$ satisfies \ref{Partial isomorphisms preserve products} and for all quantifier-free formulas $\phi(v_0,\dots,v_{n-1})$ of $\fot$ and $n$-ary $X\in\dom(f)$,
        \[
            \A\models_X\phi \iff \B\models_{\f(X)}\phi.
        \]
    \end{enumerate}
    In particular, a partial elementary team map is a partial team isomorphism.
\end{proposition}
\begin{proof}
    We first show that if $\f$ satisfies \ref{Partial isomorphisms preserve products} and preserves quantifier-free formulas of $\fot$, then it is a partial team isomorphism.
    We show \ref{Partial isomorphisms preserve order}, the rest are similar.
    Let $X,Y\in\dom(\f)$ be $n$-ary. If $X = \emptyset$, then $X\subseteq Y$ and $\f(X)=\emptyset\subseteq\f(Y)$. If $X\neq\emptyset = Y$, then $X\nsubseteq Y$ and $\emptyset\neq \f(X)\nsubseteq \f(Y) = \emptyset$.
    Otherwise denote $\vx = (v_0,\dots,v_{n-1})$ and $\vy = (v_n,\dots,v_{2n-1})$. Then
    \begin{align*}
        X\subseteq Y &\iff \A\models_{X\times Y} \vx\subseteq\vy \\
        &\iff \B\models_{\f(X\times Y)} \vx\subseteq\vy \\
        &\stackrel{\text{\ref{Partial isomorphisms preserve products}}}\iff \B\models_{\f(X)\times\f(Y)} \vx\subseteq\vy \\
        &\iff \f(X)\subseteq\f(Y).
    \end{align*}

    Next we show that if $\f$ is a partial isomorphism, then for all quantifier-free formulas $\phi(v_0,\dots,v_{n-1})$ of $\fot$ and $n$-ary $X\in\dom(\f)$,
    \[
        \A\models_X\phi \iff \B\models_{\f(X)}\phi.
    \]
    We proceed by induction on $\phi$.
    The case for first-order atomic formulas follows from \cref{Partial isomorphisms preserve first-order atomic formulas}, and the cases for connectives are immediate consequences of the induction hypothesis. We show the case for constancy atoms.
    Suppose that $\phi = \dep(\vx)$ for some $\vx\in\{v_0,\dots,v_{n-1}\}^{<\omega}$, and let $\vec\imath$ be the corresponding index tuple. Then
    \begin{align*}
        \A\models_X\phi &\iff \text{$X$ is empty or $\Pr_{\vec\imath}(X)$ is a singleton} \\
        &\iff \text{$\f(X)$ is empty or $\f(\Pr_{\vec\imath}(X))$ is a singleton} \\
        &\iff \text{$\f(X)$ is empty or $\Pr_{\vec\imath}(\f(X))$ is a singleton} \\
        &\iff \B\models_{\f(X)}\phi. \qedhere
    \end{align*}
\end{proof}

We conclude this section by pointing out some equivalent characterisations of elementary team maps. The usual Tarski--Vaught test applies immediately also in this context. 

\begin{proposition}[The Tarski--Vaught Test for $\fot$]\label{Tarski--Vaught of FOT}
    Let $\f\colon\A\to\B$ be a team map. The following are equivalent.
    \begin{enumerate}
        \item\label{Tarski-Vaught: elementary embedding}
        $\f$ is an elementary team embedding.

        \item\label{Tarski-Vaught: T-V condition}
        $\f$ is a team embedding and satisfies the following condition:
        \begin{quote}
            For any formula $\phi(v_0,\dots,v_n)$ of $\fot$ and $X\subseteq\A^n$, \\ if $\B\models_{\f(X)}\existsone v_n\phi$, there is $a\in\A$ such that $\B\models_{\f(X)(\f(a)/n)}\phi$.
        \end{quote}
    \end{enumerate}
\end{proposition}
\begin{proof}
    After the observation that $\f(X(a/n)) = \f(X)(\f(a)/n)$ for all $a\in\A$, the proof is identical to its first-order counterpart.
\end{proof}

\begin{proposition}\label{Partial elementary maps: FOT vs FO}
    Let $\f\colon\A\to\B$ be a team map. The following are equivalent.
    \begin{enumerate}
        \item\label{PEM FOT vs FO: partial elementary map}
        $\f$ is a partial elementary team map.

        \item\label{PEM FOT vs FO: preserves first-order logic}
        For any first-order sentence $\phi(R_0,\dots,R_{n-1})$ with free second-order variables $R_i$, and for any nonempty $\ar(R_i)$-ary $X_i\in\dom(\f)$, $i<n$,
        \[
            \A\models\phi(X_0,\dots,X_{n-1}) \iff \B\models\phi(\f(X_0),\dots,\f(X_{n-1})).
        \]

        \item\label{PEM FOT vs FO: first-order elementary embedding in large vocabulary}
        For any $n<\omega$ and $n$-ary $X\in\dom(\f)$, let $\underline X$ be a new $n$-ary relation symbol and denote $\tau' = \tau\cup\{\underline X \mid X\in\dom(\f)\}$. Then the (partial) function $f\colon\A\to\B$, $f(a) = \f(\{a\})$ for all $\{a\}\in\dom(\f)$, is a partial elementary map between the $\tau'$-structures $\hat\A$ and $\hat\B$, where $\hat\A\restriction\tau = \A$, $\hat\B\restriction\tau = \B$, $\underline X^{\hat\A} = X$ and $\underline X^{\hat\B} = \f(X)$.
    \end{enumerate}
\end{proposition}
\begin{proof}
    The equivalence of \ref{PEM FOT vs FO: preserves first-order logic} and \ref{PEM FOT vs FO: first-order elementary embedding in large vocabulary} follows immediately by the following observation. For any first-order $\phi(x_0,\dots,x_{n-1},\vec R)$ and elements $a_0,\dots,a_{n-1}$ such that $\{a_0\},\dots,\{a_{n-1}\}\in\dom(\f)$,
    \[
        \A\models\phi(a_0,\dots,a_{n-1},\vec X) \iff \A\models\phi^*(\{a_0\},\dots,\{a_{n-1}\},\vec X),
    \]
    where $\phi^*(S_0,\dots,S_{n-1},\vec R)$ is the formula
    \[
        \exists x_0\dots\exists x_{n-1}\left(\phi \land \bigwedge_{i<n}(S_i(x_i) \land \forall y (S_i(y)\to y = x_i)) \right)
    \]
    and $S_i$ are fresh unary second-order variables. We next show the equivalence of \ref{PEM FOT vs FO: preserves first-order logic} and \ref{PEM FOT vs FO: partial elementary map}.

    \ref{PEM FOT vs FO: preserves first-order logic}$\implies$\ref{PEM FOT vs FO: partial elementary map}: The fact that $\f$ preserves formulas of $\fot$ follows directly from \cref{Translation between FOT and FO}. We show that $\f$ satisfies \ref{Partial isomorphisms preserve products}. Let $X,X'\in\dom(\f)$ be an $n$-ary and $m$-ary nonempty relation, respectively, and let $R$ and $R'$ be an $n$-ary and an $m$-ary and $S$ an $(n+m)$-ary second-order variable. We let $\phi(S,R,R')$ be the sentence
    \[
        \forall x_0\dots\forall x_{n-1}\forall y_0\dots\forall y_{m-1} (S(\vx\vy) \leftrightarrow (R(\vx)\land R'(\vy)),
    \]
    expressing that ``$S$ is the cartesian product of $R$ and $R'$''. Denote $Y=X\times X'$. Then $\A\models\phi(Y,X,X')$, whence $\B\models\phi(\f(Y),\f(X),\f(X'))$. Now $\f(Y) = \f(X)\times \f(X')$ as desired.

    \ref{PEM FOT vs FO: partial elementary map}$\implies$\ref{PEM FOT vs FO: preserves first-order logic}: Let $\phi(R_0,\dots,R_{n-1})$ be a first-order sentence with free second-order variables $R_i$ of arity $n_i$. The first observation is that whenever $X_i$ is empty, so is $\f(X_i)$. Let $\phi^*(R_0,\dots,R_{i-1},R_{i+1},\dots,R_{n-1})$ be the formula obtained by replacing in $\phi$ each occurrence of formulas of the form $X_i(\vec t\,)$ by $\bot$. Then, whenever $X_i=\emptyset$,
    \[
        \A\models\phi(X_0,\dots,X_{n-1}) \iff \A\models\phi^*(X_0,\dots,X_{i-1},X_{i+1},\dots,X_{n-1}),
    \]
    and the same holds for $\B$ and $\f(X_j)$.
    Hence we only need to consider nonempty $X_i$.
    
    Now let $\psi(R)$ be the sentence one obtains by replacing in $\phi$ formulas of the form $R_i(t_0,\dots,t_{n_i-1})$ by
    \[
        \exists\vx\exists\vy R(\vx,t_0,\dots,t_{n_i-1},\vy),
    \]
    where $\vx$ and $\vy$ are variable tuples of length $\sum_{j<i}n_j$ and $\sum_{i<j<n}n_j$, respectively, and $R$ is a $\sum_{i<n}n_i$-ary second order variable. Now for any structure $\C$ and $X_i\subseteq\C^{n_i}$,
    \[
        \C\models\phi(X_0,\dots,X_{n-1}) \iff \C\models\psi\left(\prod_{i<n}X_i\right).
    \]
    Denote $N = \sum_{i<n}n_i$. By \cref{Translation between FOT and FO}, there is a formula $\chi(v_0,\dots,v_{N-1})$ of $\fot$ such that for any structure $\C$ and $X\subseteq\C^N$,
    \[
        \C\models\exists\vx R(\vx)\to \psi(X) \iff \C\models_X\chi.
    \]
    Now, let $X_i\in\dom(\f)$ be nonempty and denote $X=\prod_{i<n}X_i$. As $\f$ is an elementary team embedding, we have $\f(X) = \prod_{i<n}\f(X_i)$. Then
    \begin{align*}
        \A\models\phi(X_0,\dots,X_{n-1}) &\iff \A\models\psi\left(X\right) \iff \A\models_X\chi \\
        &\iff \B\models_{\f\left(X\right)}\chi \iff \B\models\psi\left(\f\left(X\right)\right) \\
        &\iff \B\models\psi\left(\prod_{i<n}\f(X_i)\right) \\
        &\iff \B\models\phi(\f(X_0),\dots,\f(X_{n-1})). \qedhere
    \end{align*}
\end{proof}

\begin{corollary}
    If $\f\colon\A\to\B$ is a partial elementary team map, then $\A\equiv\B$.
\end{corollary}

\subsection{Elementary Team Embeddings and Isomorphisms}

Earlier we noticed that the ultrapower embedding $\iota\colon\A\to\A^\kappa/\U$, $\iota(X)=X^\kappa/\U$, is not just an elementary team embedding but an independence team embedding as well. This is actually a general phenomenon: in fact, all elementary team embeddings preserve $\foil$ in the direction of the map, as demonstrated by the next result. The true difference between elementary and independence team maps shows up at the level of partial maps.

\begin{proposition}\label{elementary.team.embeddings.are.independence}
    If $\f\colon\A\to\B$ is an elementary team embedding, then it is an independence team embedding.
\end{proposition}
\begin{proof}
    Let $\phi(v_0,\dots,v_{n-1})$ be a formula of $\foil$, let $X\subseteq\A^n$ and suppose that $\A\models_X\phi$. If $X=\emptyset$, then also $\f(X)=\emptyset$ (and vice versa) and hence trivially $\B\models_{\f(X)}\phi$. So we may assume that $X\neq\emptyset\neq \f(X)$.

    By the translation of $\foil$ to $\eso$ and \cref{Coding many predicates with one}, there exists a first-order $\tau$-sentence $\alpha(R,R')$ such that for any structure $\C$ and $Y\subseteq\C^n$,
    \[
        \C\models_Y\phi \iff \C\models\exists R'\alpha(Y,R'),
    \]
    where $R'$ is some $m$-ary relation symbol. Hence $\A\models\exists R'\alpha(X,R')$. Now, there is $X'\subseteq\A^m$ such that $\A\models\alpha(X,X')$.
    Now either $X'=\emptyset$ or $X'\neq\emptyset$. We look at the two cases separately.
    \begin{enumerate}
        \item Suppose that $X'=\emptyset$. Let $\alpha^*(R)$ be the formula one obtains from $\alpha$ by replacing every occurence of $R'$ by $\bot$. Then for any $\C$ and $Y\subseteq\C^n$,
        \[
            \C\models\alpha(Y,\emptyset) \iff \C\models\alpha^*(Y).
        \]
        Now $\A\models\alpha^*(X)$. By the translation of $\fol$ to $\fot$, there is a $\tau$-formula $\psi$ of $\fot$ such that for any structure $\C$ and $Y\subseteq\C^n$,
        \[
            \C\models\exists\vx Y(\vx)\to\alpha^*(Y) \iff \C\models_Y\psi.
        \]
        Since $\A\models\exists\vx X(\vx)\to\alpha^*(X)$, we have $\A\models_X\psi$. Since $\f$ is an elementary team embedding, we obtain $\B\models_{\f(X)}\psi$. Then $\B\models\exists\vx \f(X)(\vx)\to\alpha^*(\f(X))$, and as $\f(X)$ is nonempty, we have $\B\models\alpha^*(\f(X))$. Hence $\B\models\alpha(\f(X),\emptyset)$, and so $\B\models\exists R'\alpha(\f(X),R')$. Thus $\B\models_{\f(X)}\phi$.

        \item If $X'\neq\emptyset$, then let $X'' = X\times X'$. Let $\alpha^*$ be the formula one obtains from $\alpha$ by replacing each occurence of $R(t_0,\dots,t_{n-1})$ by
        \[
            \exists x_0\dots\exists x_{m-1}R''(t_0,\dots,t_{n-1},x_0,\dots,x_{m-1})
        \]
        and each $R'(t_0,\dots,t_{m-1})$ by
        \[
            \exists x_0\dots\exists x_{n-1}R''(x_0,\dots,x_{n-1},t_0,\dots,t_{m-1}),
        \]
        where $x_i$ are fresh variables. By the translation of $\fol$ to $\fot$, there is a $\tau$-formula $\psi$ of $\fot$ such that for all $\C$ and $Y\subseteq\C^{n+m}$,
        \[
            \C\models\exists\vx R''(\vx)\to\alpha^*(Y) \iff \C\models_Y\psi.
        \]
        Now  as $\A\models_{X''}\psi$, and as $\f$ is an elementary team embedding, we have $\B\models_{\f(X'')}\psi$. Hence $\B\models\exists\vx R''(\vx)\to\alpha^*(\f(X''))$. As $X$ and $X'$ are both nonempty, so is $X''$ and thus so is $\f(X'')$. Therefore $\B\models\alpha^*(\f(X''))$. As $\f$ is an elementary team embedding, we have $\f(X'') = \f(X)\times\f(X')$. As $\B\models\alpha^*(\f(X)\times\f(X'))$, we have $\B\models\alpha(\f(X),\f(X'))$ and thus $\B\models\exists R'\alpha(\f(X), R')$. But then $\B\models_{\f(X)}\phi$.
    \end{enumerate}
    Thus $\B\models_{\f(X)}\phi$, which finishes the proof.
\end{proof}

Next, we show that all element-total, element-surjective partial team isomorphisms are simply usual isomorphism lifted to the level of relations.

\begin{proposition}\label{extension.of.isomorphisms}
    Suppose that $\f\colon\A\to\B$ is an element-total and element-surjective partial team isomorphism. Then there is an (ordinary) isomorphism $\pi\colon\A\to\B$ such that $\f \subseteq \hat\pi$.
\end{proposition}
\begin{proof}
    Define $\pi\colon\A\to\B$ by setting $\pi(a) = \f(a)$ (i.e. $\pi(a)$ is the unique element of $\f(\{a\})$) for all $a\in\A$. By \cref{Partial isomorphisms preserve quantifier-free formulas}, $\pi$ is an embedding, and as $\pi$ is surjective by assumption, it is an isomorphism.
    
    Let $X\in\dom(\f)$ be $n$-ary. Then for any $\vec b\in \f(X)$ we have $\{\vec b\}\in\ran(\f)$. By \ref{Partial isomorphisms preserve atoms}, there is $\vec a\in\A^n$ such that $\pi(\vec a) = \f(\vec a) = \vec b$. Hence $\f(X)\subseteq\{\pi(\vec a) \mid \vec a\in X\}$. Then note that given any $\vec a\in\A^n$, by \ref{Partial isomorphisms preserve order} we have
    \begin{align*}
        \vec a\in X &\implies \{\vec a\} \subseteq X \implies \f(\{\vec a\})\subseteq \f(X) \implies \{\f(\vec a)\}\subseteq \f(X) \\
        &\implies \f(\vec a)\in \f(X) \implies \pi(\vec a)\in \f(X).
    \end{align*}
    Hence $\f(X) \supseteq \{\pi(\vec a) \mid \vec a\in X \}$. Thus $\f(X) = \hat\pi(X)$. It follows that $\f\subseteq\hat\pi$.
\end{proof}

\begin{corollary}
    Suppose that $\f\colon\A\to\B$ is a team map. Then $\f$ is a team isomorphism if and only if there is an (ordinary) isomorphism $\pi\colon\A\to\B$ such that $\f = \hat\pi$.
\end{corollary}

\begin{proposition}\label{From team embedding to a substructure}
    Let $\f\colon\A\to\C$ be a team embedding, and let $\B$ be a structure such that $\dom(\B) = \{\f(a) \mid a\in\A\}$, if $R\in\tau$ is a relation or function symbol, then $R^\B = \f(R^\A)\cap\B^{\ar(R)}$, and $c^\B = \f(c^\A)$ for constant symbols $c\in\tau$. Then
    \begin{enumerate}
        \item $\B$ is a substructure of $\C$,
        \item $\g\colon\A\to\B$, $\g(X) = \f(X)\cap\B^n$ for $X\subseteq\A^n$, is a team isomorphism, and
        \item if $\f$ is elementary, then also $\f\circ\g^{-1}$ is elementary and $\B$ is an elementary substructure of $\C$.
    \end{enumerate}
\end{proposition}
\begin{proof}
    We show that $\B$ is a substructure of $\C$. The other claims are proved in a similar fashion.

    First of all, note that $\g(a) = \f(a)$ for all elements $a\in\A$. Additionally,
    \[
        \f(a_0,\dots,a_{n-1}) = (\f(a_0),\dots,\f(a_{n-1})) = (\g(a_0),\dots,\g(a_{n-1})) = \g(a_0,\dots,a_{n-1}),
    \]
    whence $\g(\vec a) = \f(\vec a)$ for all $\vec a\in\A^n$.
    If $R\in\tau$ is an $n$-ary relation symbol, then $R^\B = \f(R^\A)\cap\B^n = R^\C\cap\B^n$. If $F\in\tau$ is a function symbol, the same holds for $F$, i.e. $F^\B = F^\C\cap\B^n$. We show that $\B$ is closed under $F^\C$. Let $\vec b \in\B^n$. Then $\vec b = \f(\vec a)$ for some $\vec a\in\A^n$. Let $a = F^\A(\vec a)$. Then $\vec aa\in F^\A$, so $\f(\vec aa)\in \f(F^\A) = F^\C$, so $\f(a) = F^\C(\f(\vec a)) = F^\C(\vec b)$. But by the definition of $\B$, $\f(a)\in\B$. Hence $\B$ is closed under $F^\C$, and thus $F^\B$ is well defined. Finally, $\B$ contains $c^\C$ for all constant symbols $c\in\tau$, as $c^\C = \f(c^\A)$. Hence $\B$ is a substructure of $\C$.
\end{proof}

The next proposition highlights a connection between elementary team maps and complete atomic Boolean algebras. It is more of a side remark and shall not be needed in later proofs.

\begin{proposition}
    Let $\f$ be a partial elementary team map. Then
    \begin{enumerate}
        \item for all $X,Y\in\dom(\f)$ of the same arity,
        \[
            \f(X\cup Y) = \f(X)\cup\f(Y),
        \]
        whenever $X\cup Y\in\dom(\f)$, and
        \item for all $X\in\dom(\f)$,
        \[
            \f(X^c) = \f(X)^c,
        \]
        whenever $X^c\in\dom(\f)$.
    \end{enumerate}
    Furthermore, if $\f$ is an elementary team embedding $\A\to\B$, then for any $n<\omega$, $\f\restriction_n$ is an (ordinary) elementary embedding between $(\Pow(\A^n),\cup,\cap,{}^c,\emptyset,\A^n)$ and $(\Pow(\B^n),\cup,\cap,{}^c,\emptyset,\B^n)$ in the language $\{\lor,\land,\neg,\bot,\top\}$ of Boolean algebras.
\end{proposition}
\begin{proof}\quad
    \begin{enumerate}
        \item Denote $\vx = (v_0,\dots,v_{n-1})$, $\vy = (v_n,\dots,v_{2n-1})$, $\vz = (v_{2n},\dots,v_{3n-1})$ and $\vec w = (v_{3n},\dots,v_{4n-1})$. Then it is straightforward to verify that for any structure $\C$ and nonempty $X,Y,Z\subseteq\C^n$, $Z = X\cup Y$ if and only if
        \[
            \C\models_{X\times Y\times Z} \vx\subseteq\vz \land \vy\subseteq\vz \land \forallone\vec w (\vec w\subseteq\vz \wcimp (\vec w\subseteq\vx \ivee \vec w\subseteq\vy)),
        \]
        whence the claim follows.

        \item We show that for all $X$ and $Y$, $X = Y^c$ if and only if $\f(X) = \f(Y)^c$. Since $\f(\emptyset) = \emptyset$ and $\f(\A^n) = \B^n$, the claim holds whenever either of the sets $X$ and $Y$ is empty, so suppose that is not the case. Denote $\vx = (v_0,\dots,v_{n-1})$ and $\vy = (v_n,\dots,v_{2n-1})$. Now let $\phi$ be the formula
        \[
            \forallone z_0\dots\forallone z_{n-1} ( \vz\subseteq\vx \ivee\vz\subseteq\vy ) \land \vx\mid\vy.
        \]
        Then
        \begin{align*}
            X = Y^c &\iff \A\models_{X\times Y}\phi \iff \B\models_{\f(X\times Y)}\phi \\
            &\iff \B\models_{\f(X)\times\f(Y)}\phi \iff \f(X) = \f(Y)^c.
        \end{align*}
    \end{enumerate}
    For the ``furthermore'' part, first note that the Boolean algebras in question are complete and atomic. If $\A$ and $\B$ are finite structures, then $\f$ must arise from an isomorphism between the structures and hence each $\f\restriction_n$ is an isomorphism between the respective Boolean algebras. Otherwise the Boolean algebras are infinite. The atoms of $(\Pow(\A^n),\cup,\cap,{}^c,\emptyset,\A^n)$ are singletons, and $\f\restriction_n$ maps singletons (and nothing else) to singletons, i.e. preserves the atoms of the Boolean algebra in both directions. Hence it is enough to show that whenever $\C$ and $\D$ are infinite atomic Boolean algebras and $g\colon\C\to\D$ is an embedding that preserves atoms both ways, then $g$ is an elementary embedding.

    \begin{claim}
        Suppose that $\C$ and $\D$ are infinite atomic Boolean algebras. Then $g\colon\C\to\D$ is an elementary embedding if and only if it is an embedding that preserves atoms.
    \end{claim}
    \noindent\emph{Proof of claim.}
    If $g$ is elementary, then it preserves atoms, as being an atom is a first-order definable property.
    For the converse, let $\tau = \{\lor,\land,\neg,\bot,\top\}\cup\{P_n \mid n<\omega\}$, i.e. the signature of Boolean algebras augmented with infinitely many fresh unary predicates. Let $T$ be the $\tau$-theory consisting of the first-order theory of infinite atomic Boolean algebras and the additional sentences stating that $P_n(a)$ holds exactly when $a$ is above at least $n$ atoms in the Boolean algebra order. By \cite[Theorem~6.20]{poizat2012course} we have that $T$ admits quantifier elimination and so $T$ is model complete. Now, if $g$ is an atom-preserving embedding $\C\to\D$, then it is easy to show that $\hat\C\models P_n(a)$ if and only if $\hat\D\models P_n(f(a))$ for all $a\in\C$, where $\hat\C$ and $\hat\D$ are the unique $\tau$-expansions of $\C$ and $\D$, respectively, that satisfy $T$. This means that $g$ is a $\tau$-embedding $\hat\C\to\hat\D$, and as $T$ is model complete, $g$ is elementary. Hence $g$ is also an elementary embedding between the reducts $\C$ and $\D$.
\end{proof}

\subsection{Diagrams}

In this section, we give a simple condition to check if there is an elementary team embedding from a structure $\A$ to a structure $\B$. Such condition is the obvious team-semantic extension of the method of elementary diagrams, used to provide a ``template'' for elementary embeddings in first-order logic. In turn, this provides us with a connection between elementary team embeddings and complete (first-order) embeddings.

\begin{definition}
    Let $\A$ be a $\tau$-structure, $A\subseteq\A$ and $\X\subseteq\R(\A)$.
    \begin{enumerate}
        \item By $\tau(A)$ we denote the expansion of $\tau$ by names for all elements of $A$, i.e. the signature $\tau\cup\{\underline a \mid a\in A\}$, where each $\underline a$ is a fresh constant symbol. By $\tau(\X)$ we mean the expansion of $\tau$ by names for all \emph{nonempty} elements of $\X$, i.e. the signature $\tau\cup\bigcup_{n<\omega}\{\underline X \mid \emptyset\neq X\in\X\}$, where $\underline X$ is a fresh $\ar(X)$-ary relation symbol for any $X\in\X$.
        \item By $\A_A$ we mean the $\tau(A)$-expansion $\hat\A$ of $\A$ such that $\underline a^{\hat\A} = a$ for all $a\in A$. By $\A_\X$ we mean the $\tau(\X)$-expansion $\hat\A$ of $\A$ with $\underline X^{\hat\A} = X$ for all $X\in\X$.
        \item The \emph{(elementary) team diagram of $(\A,\X)$}, denoted by $\Diag(\A,\X)$, is the complete $\fot$-theory of $\A_\X$.
    \end{enumerate}
\end{definition}

\begin{lemma}\label{diagram.closure.lemma}
    Let $\A$ be a $\tau$-structure and $\X\subseteq\R(\A)$. Then every $\tau(\X)$-structure $\B$ with $\B\models\Diag(\A,\X)$ has a unique $\tau(\cl(\X))$-expansion $\hat\B$ such that $\hat\B\models\Diag(\A,\cl(\X))$.
\end{lemma}
\begin{proof}
    Clearly $\A^n$, $\Delta_\A$ and $s^\A$ for all symbols $s\in\tau$ are each definable by a first-order $\tau$-sentence. Likewise, given relations $X,Y\in\R(\A)$, the relations $X\cap Y$, $X\times Y$ and $\Pr_{\vec\imath}(X)$ for any $\vec\imath\in\ar(X)^{<\omega}$ are each definable by a first-order sentence with only second-order parameters $X$ and $Y$. Hence the relations $\underline{\A^n}$, $\underline{\Delta_\A}$, $\underline{s^\A}$ for $s\in\tau$, $\underline{X\cap Y}$, $\underline{X\times Y}$ and $\underline{\Pr_{\vec\imath}(X)}$ are each definable by a $\tau(\X)$-sentence of first-order logic. Hence each of them is also definable by a $\tau(\X)$-sentence of $\fot$. The claim clearly follows.
\end{proof}

\begin{lemma}\label{diagram.team.embeddings}
    For $\tau$-structures $\A$ and $\B$, the following are equivalent.
    \begin{enumerate}
        \item There is a $\tau(\X)$-expansion $\hat\B$ of $\B$ with $\hat\B\models\Diag(\A,\X)$.
        \item There is a partial elementary team map $\f\colon\A\to\B$ with $\dom(\f)=\cl(\X)$.
    \end{enumerate}
\end{lemma}
\begin{proof}
    Suppose first that $\f\colon\A\to\B$ is a partial elementary team map with $\dom(\f)=\cl(\X)$. Now, by \cref{Partial elementary maps: FOT vs FO}, the map $a\mapsto\f(a)$ is a partial elementary map $\A_{\cl(\X)}\to\hat\B$, where $\hat\B$ is the $\tau(\cl(\X))$-expansion of $\B$ satisfying $\underline X^{\hat\B} = \f(X)$. Then $\hat\B\models\Diag(\A,\cl(\X))$. Thus, $\hat\B\restriction\tau(\X)$ is the desired expansion of $\B$.

    For the converse, suppose there is some $\tau(\X)$-expansion $\hat\B$ of $\B$ such that $\hat\B\models\Diag(\A)$. By \cref{diagram.closure.lemma}, there is a unique $\tau(\cl(\X))$-expansion $\hat\B^*$ of $\hat\B$ such that $\hat\B^*\models\Diag(\A,\cl(\X))$. Let $\Gamma$ be the set of $\fol$-translations of sentences of $\Diag(\A,\cl(\X))$. Now $\Gamma$ is the complete first-order $\tau(\cl(\X))$-theory of $\A_{\cl(\X)}$ and $\hat\B^*\models\Gamma$. Define a team map $\f\colon\A\to\B$ by setting $\f(X) = \underline X^{\hat\B^*}$ for all $X\in\cl(\X)$. Let $\phi(R_0,\dots,R_{n-1})$ be a first-order $\tau$-sentence, let $X_0,\dots,X_{n-1}\in\cl(\X)$ with $\ar(X_i)=\ar(R_i)$ and suppose that $\A\models\phi(X_0,\dots,X_{n-1})$. Then $\A_{\cl(\X)}\models\phi(\underline X_0,\dots,\underline X_{n-1})$, whence $\phi(\underline X_0,\dots,\underline X_{n-1})\in\Gamma$ and so $\hat\B^*\models\phi(\underline X_0,\dots,\underline X_{n-1})$. Then, since $\underline X_i^{\hat\B^*} = \f(X_i)$, we have $ \B\models\phi(\f(X_0),\dots,\f(X_{n-1})$. By \cref{Partial elementary maps: FOT vs FO} it follows that $\f$ is an elementary team embedding.
\end{proof}

As a corollary of \cref{Partial elementary maps: FOT vs FO}, we obtain a connection between complete embeddings in first-order logic (Definition~\ref{Definition: Complete embedding}) and elementary team embeddings in our setting: every elementary team embedding is a complete embedding on the level of elements, and conversely every complete embedding gives rise to an elementary team embedding.

\begin{corollary}\label{elementary.team.maps.vs.complete.embeddings}\quad
    \begin{enumerate}
        \item Let $\f\colon\A\to\B$ be an elementary team embedding. Then the function $a\mapsto \f(a)$ is a complete embedding $\A\to\B$.
        \item Let $f\colon\A\to\B$ be a complete embedding. Then there is an elementary team embedding $\f\colon\A\to\B$ such that $\f(a) = f(a)$ for all $a\in\A$.
    \end{enumerate}
\end{corollary}
\begin{proof}\quad
    \begin{enumerate}
        \item Follows immediately from \cref{Partial elementary maps: FOT vs FO}.

        \item As $f$ is complete, there is a $\tau(\R(\A))$-expansion $\hat\B$ of $\B$ such that $f$ is elementary between $\A_{\R(\A)}$ and $\hat\B$. For all $X\in\R(\A)$, define $\f(X) = \underline X^{\hat\B}$. Then by \cref{Partial elementary maps: FOT vs FO}, $\f$ is an elementary team embedding. Now, for every $a\in\A$, we have $a\in\{a\} = \underline{\{a\}}^{\A_{\R(\A)}}$, whence, as $f$ is elementary, we obtain $f(a)\in\underline{\{a\}}^{\hat\B} = \f(\{a\})$, yielding $f(a) = \f(a)$. \qedhere
    \end{enumerate}
\end{proof}

As every complete embedding is essentially a limit ultrapower embedding, the same holds for elementary team embeddings. Recall that given an isomorphism $\pi$, $\hat\pi$ is the team isomorphism $X \mapsto \{\pi(\vec a) \mid \vec a\in X\}$.

\begin{proposition}
    Suppose that $\f\colon\A\to\B$ is an elementary team embedding. Then there is a cardinal $\kappa$, an ultrafilter $\cF$ on $\kappa$ and a filter $\cG$ on $\kappa^2$ such that, denoting the structure $(\A^\kappa/\cF)|\cG$ by $\C$, there is an isomorphism $\pi\colon\B\to\C$ with
    \[
        (\hat\pi\circ\f)(X) = \g(X)\cap\C^n
    \]
    for all $X\subseteq\A^n$, where $\g\colon\A\to\A^\kappa/\cF$ is the ultrapower team embedding ${X\mapsto X^\kappa/\cF}$.
\end{proposition}
\begin{proof}
    Straightforwardly follows from \cref{Complete embeddings are limit ultrapower embeddings} and \cref{elementary.team.maps.vs.complete.embeddings}.
\end{proof}

\section{Abstract Elementary Classes for Team Semantics}\label{Section_AEC}

Abstract elementary classes (AECs) were introduced by Shelah in  \cite{shelah2006classification} and provide an important tool to adapt several techniques of elementary model theory beyond the scope of first-order logic. It is, thus, very natural to consider whether we can introduce a suitable AEC to study (complete) theories in independence logic or, equivalently, in existential second-order logic. After recalling the definition of an AEC, we consider, in this section, whether the maps considered in the previous section induce a natural strong substructure relation for model classes of theories in $\eso$ and $\foil$ that would make them AECs. In particular, we shall examine in \cref{subsec:coherence} the coherence axiom of AECs, in \cref{colimits.subsection} the unions of chains axiom (in the form of the existence of direct limits), and in \cref{subsec:LoSk} the Löwenheim--Skolem axiom. In order to motivate our subsequent choices, we proceed in these sections by ``trial and error'', and we seek to explain what goes right and what goes wrong when we try to define AECs for $\eso$ and $\foil$. Afterwards, we suggest in \cref{subsec:accessible.team.category} a way to solve the issues identified in  Sections~\ref{subsec:coherence}--\ref{subsec:LoSk}. Most specifically, we move from the setting of AECs to the setting \emph{abstract elementary team categories}, a specific subclass accessible categories
(cf.~\cite{adamek1994locally}),
and show that they provide the right framework to analyse classes of structures with team maps. We start by recalling the definition of an abstract elementary class.

\begin{definition}
    Let $\K$ be a class of $\tau$-structures and $\preceq$ a partial order on $\K$ that is a subrelation of the substructure relation $\subseteq$. We say that $(\K,\preceq)$ is an abstract elementary class if the following hold.
    \begin{enumerate}
        \item \emph{Closure under isomorphism}: If $\A\in\K$ and $\pi\colon\A\to\B$ is an isomorphism, then $\B\in\K$, and if $\C\preceq\A$, then $\pi(\C)\preceq\B$.
        \item \emph{Coherence}: If $\A,\B\preceq\C$ and $\A\subseteq\B$, then $\A\preceq\B$.
        \item \emph{Unions of chains}: If $\alpha$ is an ordinal and $(\A_i)_{i<\alpha}$ is a $\preceq$-increasing continuous sequence of elements of $\K$, i.e. $\A_i\preceq\A_j$ for $i<j$ and $\A_\beta = \bigcup_{i<\beta}\A_i$ for $\beta<\alpha$ a limit, then
        \begin{enumerate}
            \item $\bigcup_{i<\alpha}\A_i\in\K$,
            \item for all $j<\alpha$, $\A_j\preceq\bigcup_{i<\alpha}\A_i$, and
            \item \emph{Smoothness}: if $\B\in\K$ is such that $\A_i\preceq\B$ for all $i<\alpha$, then $\bigcup_{i<\alpha}\A_i\preceq\B$.
        \end{enumerate}
        \item \emph{Löwenheim--Skolem property}: There is a cardinal $\kappa\geq|\tau|+\aleph_0$ such that for all $\A\in\K$ and $X\subseteq\A$ with $|X|\leq \kappa$, there is $\B\in\K$ such that $X\subseteq\B\preceq\A$ and $|\B|\leq\kappa$.
    \end{enumerate}
\end{definition}

\noindent We call the least cardinal $\kappa$ that satisfies the Löwenheim--Skolem property for $\K$ the Löwenheim--Skolem number of $\K$ and denote it by $\LS(\K)$.

The relation $\preceq$ in the above definition is often called ``strong substructure relation'', and a map $f\colon\A\to\B$ such that $f(\A)\preceq\B$ is called a ``strong embedding''. So far, we have seen two possible candidates for a notion of strong embedding in the context of team semantics:
\begin{enumerate}
    \item a team embedding, and
    \item an elementary team embedding (which is the same thing as an independence team embedding by Proposition~\ref{elementary.team.embeddings.are.independence}).
\end{enumerate}

As mentioned above, in this section we will study all the properties of abstract elementary classes in the context of team semantics. In particular, we shall see that one has to modify the definition of an AEC to fit the context of team maps. In fact, it is more natural to consider just embeddings and strong embeddings, rather than to bother with substructures and strong substructures. Our conclusion is that, essentially, all key properties of abstract elementary classes hold in our setting, though a major issue is that we have to drop the requirement that strong maps ``extend'' the substructure relation. We will consider this issue at length in the following sections. We begin by studying the properties of coherence, closure under direct limits and Löwenheim--Skolem property. Then, in \cref{subsec:accessible.team.category}, we conclude our discussion describing a more category-theoretic ``quasi-AEC'' for team semantics, namely an accessible category of models and elementary team maps which we call an abstract elementary team category. The reader solely interested in the end-result of our approach can simply refer to this \cref{subsec:accessible.team.category}, and the referenced propositions.


\subsection{Coherence}\label{subsec:coherence} The following example shows why the traditional substructure relation is not a good basis for our study of structures in $\eso$ and $\foil$.

\begin{example}
    Let $\tau = \{P\}$, where $P$ is a unary predicate, and let $\preceq$ be the relation
    \begin{quote}
        $\A\preceq\B$ if and only if $\A\subseteq\B$ and there is an elementary team embedding $\f\colon\A\to\B$ with $\f(a) = a$ for all $a\in\A$.
    \end{quote}
    We show that coherence fails for this relation. For this, let $\B$ be a $\tau$-structure with $|\B| = |P^\B| = \aleph_1$ and $|\B\setminus P^\B| = \aleph_0$. Now $\B$ has a countable elementary substructure $\A$. As $\A\equiv\B$, by the Keisler--Shelah theorem there is $\kappa$ and an ultrafilter $\U$ on $\kappa$ such that $\A^\kappa/\U \cong B^\kappa/\U$. Let $\pi\colon\A^\kappa/\U\to\B^\kappa/\U$ be an isomorphism, let $\iota_\A\colon\A\to\A^\kappa/\U$ and $\iota_\B\colon\B\to\B^\kappa/\U$ be the ultrapower embeddings, and denote $\C = \A^\kappa/\U$. Now $\f = \iota_\A$ and $\g = \hat\pi^{-1}\circ\iota_\B$ are elementary team embeddings $\A\to\C$ and $\B\to\C$, respectively. Let $\phi$ be a sentence of $\foil$ expressing that $P$ and the complement of $P$ have the same cardinality. Now, suppose that there is an elementary team embedding $\h\colon\A\to\B$ (such that $\h(a) = a$ for all $a\in\A$). Now, as $\A$ is countable, we have $|P^\A| = |\A\setminus P^\A| = \aleph_0$, whence $\A\models\phi$. By \cref{elementary.team.embeddings.are.independence}, $\h$ is an independence team embedding and hence $\B\models\phi$. But this is impossible, since $|P^\B|>|\B\setminus P^\B|$. Hence no such $\h$ can exist.
\end{example}

The fault in the above example seems to lie in the fact that if $\f\colon\A\to\B$ is an independence team embedding, $\B$ may satisfy more sentences of $\foil$ than $\A$, due to the positive nature of $\eso$. But even in the case that there exist elementary team embeddings $\A\to\C$ and $\B\to\C$ and a team embedding $\A\to\B$ that do not move elements, there is no guarantee that these maps commute in a nice way. We take this as an argument in favour of a more category-theoretic approach. In particular, the following form of coherence works.

\begin{proposition}\label{prop:coherence}
    If $\f\colon\A\to\C$ and $\g\colon\B\to\C$ are elementary team embeddings and $\h\colon\A\to\B$ is a team embedding such that $\f = \g\circ\h$, then $\h$ is elementary.
\end{proposition}
\begin{proof}
    Let $\phi(v_0,\dots,v_{n-1})$ be a formula of $\fot$ and $X\subseteq\A^n$. Then
    \[
        \A\models_X\phi \iff \C\models_{\f(X)}\phi \iff \C\models_{\g(\h(X))}\phi \iff \B\models_{\h(X)}\phi.
    \]
    Hence $\h$ is elementary. \qedhere
\end{proof}

\subsection{Direct Limits}\label{colimits.subsection}

Since we are considering (elementary) team embeddings and are not so interested in substructures, we want to also consider the natural generalization of unions of chains, that is, direct limits.

If $I$ is a set and $\leq$ is a binary relation on $I$, we say that $(I,\leq)$ is a \emph{$\kappa$-directed set} if $\leq$ is a preorder (i.e. reflexive and transitive) and every $J\subseteq I$ of power $<\kappa$ has a $\leq$-upperbound. $(I,\leq)$ is a \emph{directed set} if it is an $\aleph_0$-directed set. A set $J\subseteq I$ is an \emph{upset} if it is upwards closed, i.e. if $i\in I$, $j\in J$ and $i\geq j$, then $i\in J$. For an element $i\in I$, the \emph{upward closure} $\upset i$ of $i$ is the upset $\{j\in I \mid j\geq i\}$. For a subset $J\subseteq I$, the upward closure $\upset J$ is the upset $\bigcup_{j\in J}\upset j$. A subset $J\subseteq I$ is \emph{cofinal} if for every $i\in I$ there is $j\in J$ with $i\leq j$, i.e. it meets all upsets.

\begin{definition}\label{colimit}\quad
    \begin{enumerate}
        \item Let $(I,\leq)$ be a directed set. A directed system of $\tau$-structures and team maps based on $(I,\leq)$ is a sequence $(\A_i,\f_{i,j})_{i,j\in I, i\leq j}$ such that
        \begin{enumerate}
            \item each $\A_i$ is a $\tau$-structure and $\f_{i,j}\colon\A_i\to\A_j$ a team map,
            \item for all $i,j\in I$ with $j\geq i$, we have $\dom(\f_{i,j}) = \dom(\f_{i,i})$ and $\f_{i,i} = \id\restriction\dom(\f_{i,i})$, and
            \item for all $i,j,k\in I$ with $i\leq j\leq k$, we have $\ran(\f_{i,j})\subseteq\dom(\f_{j,k})$ and $\f_{i,k} = \f_{j,k}\circ\f_{i,j}$.
        \end{enumerate}
        We say that a directed system is $\kappa$-directed if $(I,\leq)$ is a $\kappa$-directed set.
        \item Let $\mathcal A = (\A_i,\f_{i,j})_{i,j\in I, i\leq j}$ be a directed system of $\tau$-structures and element-total partial team isomorphisms. We define a structure $\B\coloneqq\lim\mathcal A$, called the \emph{direct limit} of the system $\mathcal A$, as follows. The domain of $\B$ is the set of functions $\eta$ such that
        \begin{enumerate}
            \item $\dom(\eta)$ is a nonempty upset of $(I,\leq)$,
            \item $\eta(i)\in\A_i$ for all $i\in\dom(\eta)$,
            \item for all $i,j\in\dom(\eta)$, if $i\leq j$, then $\f_{i,j}(\eta(i)) = \eta(j)$, and
            \item if $i\in I$, $j\in\dom(\eta)$, $i\leq j$ and there is $a\in\A_i$ with $\f_{i,j}(a) = \eta(j)$, then $i\in\dom(\eta)$.
        \end{enumerate}
        For all $i\in I$, we define team maps $\g_i\colon\A_i\to\B$ as follows: given $X\in\dom(f_{i,i})$, we let $\g_i(X)$ be the set of all $(\eta_0,\dots,\eta_{n-1})\in\B^n$ such that there is some $j\in\bigcap_{k<n}\dom(\eta_k)$ with $j\geq i$ such that $(\eta_0(j),\dots,\eta_{n-1}(j))\in\f_{i,j}(X)$. Then the interpretations of relation and function symbols $R\in L$ are defined by setting $R^\B = \g_i(R^{\A_i})$, and for constant symbols $c\in\tau$, $c^\B = \g_i(c^{\A_i})$, where $i\in I$ is arbitrary.

        When there is no risk of confusion, we write simply $\lim_{i\in I}\A_i$ instead of $\lim\mathcal A$. We call the maps $\g_i$ the direct limit maps of the system $\mathcal A$.
    \end{enumerate}
\end{definition}

We immediately make the following observations. In particular, this shows that the direct limit is well defined.

\begin{lemma}\label{Limit elements are the same iff they are the same somewhere}\quad
    \begin{enumerate}
        \item\label{Direct limit lemma first poperty}
        If $\eta,\xi\in\lim_i\A_i$ and $\eta(i)=\xi(i)$ for some $i\in\dom(\eta)\cap\dom(\xi)$, then $\eta = \xi$. Hence for any $i\in I$ and $a\in\A_i$, there is a unique $\eta\in\lim_i\A_i$ with $\eta(i)=a$.

        \item For all $i\leq j$, $\g_i = \g_j\circ\f_{i,j}$.

        \item The interpretations of symbols $R\in\tau$ in $\lim_{i\in I}\A_i$ are well defined.

        \item Each $\g_i$ is a well-defined team map.

        \item For any $i\in I$, $n$-ary $X\in\dom(\g_i)$ and $\vec a\in X$, there is a unique $\vec\eta\in\g_i(X)$ with $\eta_k(i) = a_i$ for all $k<n$. Furthermore, $\vec\eta = \g_i(\vec a)$.

        \item If each $\f_{i,j}$ is (element-)total, then so is each $\g_i$.

        \item If $J\subseteq I$ is a cofinal subset such that $(J,\leq)$ is also directed, then there is an isomorphism $\pi\colon\lim_{i\in I}\A_i\to\lim_{j\in J}\A_j$ such that if $\g^J_j$ and $\g^I_i$ are the direct limit maps in the systems based on $J$ and $I$, respectively, then for all $j\in J$, $\g^J_j = \hat\pi\circ\g^I_j$.
    \end{enumerate}
\end{lemma}
\begin{proof}
    We prove the last claim and leave the rest to the reader.

    Define $\pi$ by setting $\pi(\eta) = \eta\restriction(\dom(\eta)\cap J)$ for all $\eta\in\lim_{i\in I}\A_i$. We show that this is an isomorphism.
    For injectivity, let $\eta,\xi\in\lim_{i\in I}\A_i$ and suppose that $\pi(\eta) = \pi(\xi)$. Let $j\in\dom(\pi(\eta))=\dom(\pi(\xi))$. Now $\eta(j) = \pi(\eta)(j) = \pi(\xi)(j) = \xi(j)$, whence  by item~\ref{Direct limit lemma first poperty} we have that  $\eta=\xi$.
    
    For surjectivity, let $\eta\in\lim_{j\in J}\A_j$. Pick some $j\in\dom(\eta)$. Now there is a unique $\xi\in\lim_{i\in I}\A_i$ with $\xi(j)=\eta(j)$. Let $i\in\dom(\eta)$. Letting $k\in J$ be an upper bound of $\{i,j\}$, we have $\xi(k) = \f_{j,k}(\xi(j)) = \f_{j,k}(\eta(j)) = \eta(k) = \f_{i,k}(\eta(i))$, whence $i\in\dom(\xi)$. Hence $\dom(\eta)\subseteq\dom(\xi)$, and by uniqueness of all the values, this means that $\eta\subseteq\xi$. Then let $i\in\dom(\xi)\cap J$. Let again $k\in J$ be an upper bound of $\{i,j\}$. Now $\f_{i,k}(\xi(i)) = \xi(k) = \eta(k)$, so we have $i\in\dom(\eta)$. Hence $\dom(\eta) = \dom(\xi)\cap J$ and for all $i\in\dom(\eta)$, $\xi(i) = \eta(i)$. Thus $\eta = \xi\restriction(\dom(\xi)\cap J) = \pi(\xi)$.

    Since the interpretations of symbols of $\tau$ are defined via $\g^I_i$ and $\g^J_j$, the isomorphism condition of $\pi$ will follow from $\g^J_j = \hat\pi\circ\g^I_j$. So we show this final fact. Fix $j\in J$ and let $X\in\dom(\g^I_j) = \dom(\g^J_j)$, and let $\vec\eta\in\hat\pi(\g^I_j(X)) = \pi[\g^I_j(X)]$. Then $\vec\eta = \pi(\vec\xi\,)$ for some $\vec\xi\in\g^I_j(X)$. Let $i\in \upset j\cap J\cap\bigcap_{k<n}\dom(\xi_k)$ be such that $(\xi_0(i),\dots,\xi_{n-1}(i))\in\f_{j,i}(X)$. Now $\eta_k(i)=\xi_k(i)$ for all $k<n$, whence $\vec\eta\in\g^J_j(X)$. Hence $\hat\pi(\g^I_j(X))\subseteq\g^J_j(X)$.

    Then let $\vec\eta\in\g^J_j(X)$ and denote $\vec \xi = \pi^{-1}(\vec\eta)$. Now there is $i\in\upset j\cap J\cap\bigcap_{k<n}\dom(\eta_k)$ such that $(\eta_0(i),\dots,\eta_{n-1}(i))\in\f_{j,i}(X)$. As $i\in\dom(\xi_k)$ and $\eta_k(i) = \xi_k(i)$ for all $k<n$, we have $\vec\xi\in\g^I_j(X)$ and so $\vec\eta=\pi(\vec\xi\,)\in\pi[\g^I_j(X)] = \hat\pi(\g^I_j(X))$. Hence $\g^J_j(X)\subseteq\hat\pi(\g^I_j(X))$. This finishes the proof.
\end{proof}

\begin{lemma}\label{Limit maps are embeddings}
    Each $\g_i$ is a partial team isomorphism $\A_i\to\lim_{i\in I}\A_i$.
\end{lemma}
\begin{proof}
    Denote $\B=\lim_i\A_i$.
    We prove \ref{Partial isomorphisms preserve atoms}, the rest are similar although tedious.

    Let $X\in\dom(\g_i)$. We show that $X$ is a singleton if and only if $\g_i(X)$ is. We consider the case $n=1$, the rest follows from \ref{Partial isomorphisms preserve products}. Suppose $X = \{a\}$. Let $\eta,\xi\in\g_i(X)$. Now there are $j\in\dom(\eta)$ and $k\in\dom(\xi)$ such that $j,k\geq i$, $\eta(j)\in\f_{i,j}(X)$ and $\xi(k)\in\f_{i,k}(X)$. Let $l$ be an upper bound of $\{j,k\}$. Now $\eta(l) = \f_{j,l}(\eta(j))\in \f_{j,l}(\f_{i,j}(X)) = \f_{i,l}(X)$, and similarly $\xi(l)\in\f_{i,l}(X)$. As $\f_{i,l}$ is a team embedding, $\f_{i,l}(X)$ is a singleton, whence $\eta(l)=\xi(l)$. But then $\eta = \xi$. Hence $\g_i(X)$ is a singleton.

    On the other hand, if $\g_i(X) = \{\eta\}$, let $j\in\dom(\eta)$ be such that $j\geq i$ and $\eta(j)\in\f_{i,j}(X)$. Now, if $a\in\f_{i,j}(X)$ is such that $a\neq\eta(j)$, then the unique $\xi\in\B$ such that $\xi(j) = a$ is different from $\eta$, but this is not possible, since then $\xi$ would be an element of $\g_i(X)$. Hence $\f_{i,j}(X)$ is a singleton, and as $\f_{i,j}$ is a partial team isomorphism, so is $X$.
\end{proof}

\begin{proposition}\label{prop:direct.limit.maps.elementary}
    If each $\f_{i,j}$ is elementary, then so is each $\g_i$.
\end{proposition}
\begin{proof}
    Denote $\B = \lim_{i\in I}\A_i$, and let $\phi$ be a formula of $\fot$. We show by induction on $\phi$ that for any $i\in I$ and $X\in\dom(\g_i)$ of arity $|\Fv(\phi)|$,
    \[
        \A_i\models_X\phi \iff \B\models_{\g_i(X)}\phi.
    \]
    If $\phi$ is atomic, this follows from the fact that $\g_i$ is a partial team isomorphism, and the connective cases follow directly from the induction hypothesis. So suppose that $\phi = \existsone v_n\psi(v_0,\dots,v_n)$. If $\A_i\models_X\phi$, then there is $a\in\A_i$ with $\A_i\models_{X(a/n)}\psi$. Element-totality of $\f_{i,i}$ and hence also of $\g_i$ implies that $X(a/n)\in\dom(\g_i)$, whence the induction hypothesis gives $\B\models_{\g_i(X(a/n))}\psi$. Since $\g_i$ is a partial team isomorphism, we have $\g_i(X(a/n)) = \g_i(X)(\g_i(a)/n)$, whence $\B\models_{\g_i(X)}\phi$.
    
    Conversely, suppose that $\B\models_{\g_i(X)}\phi$. Then there is some $\eta\in\B$ such that $\B\models_{\g_i(X)(\eta/n)}\psi$. Let $j\in\dom(\eta)$ be such that $j\geq i$. Then, as $\g_j$ is element-total, we have $Y(\eta(j)/n)\in\dom(\g_j)$ for any $Y\in\dom(\g_j)$, in particular when $Y = \f_{i,j}(X)$. Since $\g_j$ is a partial team isomorphism, we have
    \[
        \g_i(X)(\eta/n) = \g_j(\f_{i,j}(X))(\g_j(\eta(j))/n) = \g_j(\f_{i,j}(X)(\eta(j)/n)).
    \]
    Now by the induction hypothesis, $\A_j\models_{\f_{i,j}(X)(\eta(j)/n)}\psi$, whence $\A_j\models_{\f_{i,j}(X)}\phi$. Since $\f_{i,j}$ is elementary, this means that $\A_i\models_X\phi$.

    The case $\phi = \forallone v_n\psi$ is dual to the existential quantifier case.
\end{proof}

The previous results show that direct limits are well defined in the setting of team semantics. We conclude this section by pointing out how a direct limit comes naturally equipped with a specific set of relations, which we shall call admissible.

\begin{definition}\label{bounded/admissible.maps0}
    Let $\mathcal A = (\A_i, \f_{i,j})_{i,j\in I, i\leq j}$ be a directed system of $\tau$-structures and element-total partial team isomorphisms, and let $\B = \lim\mathcal A$.
    We say that a relation $R\subseteq\B^n$ is \emph{admissible} (in the system $\mathcal A$) if there is $i\in I$ and $S\subseteq\A_i^n$ such that $R = \g_i(S)$.
\end{definition}

Now, suppose each $\f_{i,j}$ (and hence also $\g_i$) is a team embedding and denote $\B=\lim\A_i$. If $\C$ is another structure with team embeddings $\h_i\colon\A_i\to\C$, $i\in I$, such that $\h_i = \h_j\circ\f_{i,j}$, then we may define a partial map $\mathcal k\colon\B\to\C$ such that for any $i\in I$ and $X\subseteq\A_i^n$, $\mathcal k(\g_i(X)) = \h_i(X)$. However, there is no obvious way to extend the map $\mathcal k$ to non-admissible sets. Hence it is unclear whether it is possible for our direct limits to satisfy the smoothness property of the unions of chains axiom of AECs (which corresponds to the universal mapping property of direct limits in category theory). However, we shall solve this problem in \cref{subsec:accessible.team.category} by transitioning to work with something reminiscent of general models of second-order logic.

\subsection{The Löwenheim--Skolem Property}\label{subsec:LoSk}

In the definition of an AEC, the Löwen\-heim--Skolem number $\LS(\K)$ of an AEC $\K$ is the least $\kappa\geq|\tau|+\aleph_0$ such that for any $\A\in\K$, whenever $X\subseteq\A$ has power $\leq\kappa$, there is $\B\in\K$ with power $\leq\kappa$ and $X\subseteq\B\preceq\A$. This is equivalent to the seemingly stronger claim that for every $\A\in\K$, whenever $X\subseteq\A$, then there is $\B\in\K$ with $X\subseteq\B\preceq\A$ and $|\B|\leq|X|+\kappa$. However, this seems to be too much to ask in our context, as even the weaker property where the set $X$ of objects is required to be strictly smaller than $\kappa$ is a very strong set-theoretic assumption: we observe that the existence of this weaker notion of a Löwenheim--Skolem number, even in the model class of the empty theory, is equivalent to the existence of \emph{supercompact cardinals}, which we show using a theorem of \textcite{MR0295904}.

Supercompact cardinals are found very high in the large cardinal hierarchy. They have strong reflection properties---for instance, if the generalized continuum hypothesis holds below a supercompact cardinal, then it holds everywhere~\cite{CODY2014620}. The first supercompact is also the Löwenheim--Skolem--Tarski number of second-order logic~\cite{magidor2011lowenheim}, i.e., the smallest $\kappa$ such that any structure $\A$ has a substructure $\B$ of size strictly smaller than $\kappa$, such that for all $b_0,\dots,b_{n-1}\in\B$ and second-order formulas $\phi(x_0,\dots,x_{n-1})$,
\[
    \B\models\phi(b_0,\dots,b_{n-1}) \iff \A\models\phi(b_0,\dots,b_{n-1}).
\]
This is very similar to what we have here, although not all of second-order logic is utilized. It follows from the observation that already universal second-order logic requires the supercompact in order to get reflected down from the superstructure to the substructure. Thus, for us to be able to preserve the truth of existential second-order logic upwards from the substructure to the superstructure, we need the enormous consistency strength of the supercompact cardinal.

In our view, having to go beyond the axioms of ZFC in order to have the Löwenheim--Skolem property for our AEC, while existential second-order logic is rather well-behaved, is a sign of a problem like the ones demonstrated in the previous two sections. This problem, like the other two, disappears when we transition to work inside the framework of \cref{subsec:accessible.team.category}.

\newcommand{\WLS}{\LS_{\mathrm{w}}}
\begin{definition}
    For a class $\K$ of $\tau$-structures, denote by $\WLS(\K)$ the least infinite cardinal $\kappa>|\tau|$ such that the following holds and call it the weak Löwenheim--Skolem number of $\K$.
    For any $\A\in\K$ and $\X\subseteq\R(\A)$ with $|\X|<\kappa$, there is $\B\in\K$ and an elementary team embedding $\iota\colon\B\to\A$ such that $|\B|<\kappa$ and $\X\subseteq\ran(\iota)$.
\end{definition}

We fix some notation. We denote by $\V$ the universe of sets. We denote other models of set theory by boldface letters such as $\bM$. If $j\colon\V\to\bM$ is an elementary embedding in the language of set theory, the critical point of $j$ is the least ordinal $\alpha$ such that $j(\alpha)>\alpha$.

\begin{definition}
    A cardinal $\kappa$ is supercompact if for every $\lambda\geq\kappa$ there is a transitive model $\bM$ of set theory with $\bM^\lambda\subseteq\bM$ (i.e., $\bM$ is closed under $\lambda$-sequences of its elements), and an elementary embedding $j\colon\V\to\bM$ with critical point $\kappa$, such that $j(\kappa)>\lambda$.
\end{definition}

\begin{theorem}[\textcite{MR0295904}]\label{First supercompact vs Pi^1_1}
    The first supercompact cardinal is the least $\kappa$ such that the following holds: for any finite $\tau$, a $\tau$-structure $\A$ with $|\A|\geq\kappa$ and a $\tau$-sentence $\phi$ of universal second-order logic with $\A\models\phi$, there exists $\B\subseteq\A$ with $|\B|<|\A|$ and $\B\models\phi$.
\end{theorem}

\begin{theorem}
    Suppose that for any finite signature $\tau$, if $\K_\tau$ is the class of all $\tau$-structures, then $\WLS(\K_\tau)$ exists. Then there exists a supercompact cardinal.
\end{theorem}
\begin{proof}
    Denote $\kappa = \sup_{|\tau|<\aleph_0}\WLS(\K_\tau)$. Here note that if $\tau$ and $\tau'$ are finite signatures and $f\colon\tau\to\tau'$ is a bijection such that $\ar(R)=\ar(f(R))$ for all relation or function symbols $R\in\tau$, then $\WLS(\K_\tau) = \WLS(\K_{\tau'})$, so $\kappa$ is well defined (the class of finite signatures is set-sized modulo renaming of symbols). By the definition of the weak LS-number, if $\tau$ is finite and $\A\in\K_\tau$ has power $\geq\kappa$, then, as $\A$ has power $\geq\WLS(\K_\tau)$, there exists $\B\in\K_\tau$ of power $<\WLS(\K_\tau)\leq\kappa$ and an elementary team embedding $\iota\colon\B\to\A$. Let $\C = \iota(\B)$. Now $\C$ is a substructure of $\A$ that is isomorphic with $\B$, and by \cref{elementary.team.embeddings.are.independence}, for all formulas $\phi$ of $\foil$ and relations $X$ of $\B$,
    \[
        \C\models_{\iota(X)}\phi \implies \B\models_{X}\phi \implies \A\models_{\iota(X)}\phi.
    \]
    In particular, for every sentence $\phi$ of $\eso$, we have
    \[
        \C\models\phi \implies \A\models\phi.
    \]
    Taking the contrapositive, this means that for every sentence $\phi$ of $\eso$,
    \[
        \A\models\neg\phi \implies \C\models\neg\phi,
    \]
    i.e. for every sentence $\phi$ of universal second-order logic,
    \[
        \A\models\phi \implies \C\models\phi.
    \]
    But then $\kappa$ is such that for every finite $\tau$, a $\tau$-structure $\A$ with $|\A|\geq\kappa$ and a universal second-order $\tau$-sentence $\phi$ with $\A\models\phi$, there exists $\C\subseteq\A$ with $|\C|<|\A|$ and $\C\models\phi$. Hence the class of cardinals with this property is nonempty. By Theorem~\ref{First supercompact vs Pi^1_1}, the first cardinal of this class is supercompact.
\end{proof}

The first supercompact cardinal also suffices as our weak Löwenheim--Skolem number. The proof for said fact
also would suggest that the stronger form of the LS-property with non-strict inequalities would be too much to ask.

\begin{theorem}\label{Lowenheim--Skolem-theorem}
    Suppose $\kappa>|\tau|$ is supercompact and $\K = \Mod(T)$ for some first-order-complete $\tau$-theory in $\foil$. Then $\WLS(\K)\leq\kappa$.
\end{theorem}
\begin{proof}
    Let $\A$ be a $\tau$-structure and $\X\subseteq\bigcup_{n<\omega}\Pow(\A^n)$, $|\X|<\kappa$. If $|\A|<\kappa$, we are done, so assume that $|\A|\geq\kappa$. We augment $\tau$ into a signature $\tau'$ by adding fresh relation symbols for elements of $\X$. Then denote $\mu = |\A|$ and let $\lambda = 2^{\mu}$. Note that $|\tau'|<\kappa\leq\mu$. Let $j$ be an elementary embedding of $\V$ into a transitive model $\bM$ with critical point $\kappa$ such that $\bM^\lambda\subseteq\bM$ and $j(\kappa)>\lambda$. Note that as $\tau'$ has power $<\kappa$, we may assume that $\tau'\in\V_\kappa$ and hence $j(\tau') = \tau'$. Let $\B$ be the $\tau'$-structure with domain $j[\dom(\A)]\coloneqq\{j(a) \mid a\in\A\}$ and with the interpretation of symbols of $\tau'$ induced by $j$.
    We make the following claims.

    \begin{claim}\quad
        \begin{enumerate}
            \item $\B\in\bM$ and $\bM\models|\B|\leq\lambda$.
            \item For any $R\subseteq\B^n$, we have $R\in\bM$ and $R=j[S] \coloneqq \{j(\vec a) \mid \vec a \in S\}$ for some $S\subseteq\A^n$.
            \item The map $R\mapsto j(j^{-1}[R])$, $R\in\R(\B)$, is in $\bM$.
        \end{enumerate}
    \end{claim}
    \begin{proof}
        \begin{enumerate}
            \item First notice that in $\V$ we have $|\B|=\mu\leq\lambda$, so there exists a surjection $f\colon\lambda\to\dom(\B)$. As $\dom(\B)\subseteq \bM$  and since $\bM$ is closed under $\lambda$-sequences, we have  that $f\in\bM$. Thus also $\dom(\B)\in\bM$, and a similar argument works for $R^\B$ for any $R\in\tau'$. As $|\tau'|\leq\lambda$ we obtain, again by closure under $\lambda$-sequences, that $\B\in\bM$. As being a surjection is absolute, we have $\bM\models\text{``$f$ is a surjection $\lambda\to\dom(\B)$''}$, so $\bM\models|\B|\leq\lambda$.
            
            \item  Since $\B^n = j[\A^n]$, we have for any $R\subseteq\B^n$ that $R = j[S]$ for some $S\subseteq\A^n$. Now $R\in\bM$ for the same reason that $\B\in\bM$

            \item For the third claim, first notice that
            \[
                |\R(\B)| = \left|\bigcup_{n<\omega}\Pow(\B^n)\right| = \aleph_0\cdot 2^{|\B|} = 2^\mu\leq\lambda.
            \]
            Hence the map $R\mapsto j(j^{-1}[R])$ also has cardinality $\leq\lambda$, and as it is a set of elements of $\bM$, it is in itself in $\bM$. \qedhere{ Claim}
        \end{enumerate}
    \end{proof}
    
    Now, suppose for a moment that for any $\tau$-formula $\phi$ of $\fot$ and any $S\subseteq\A^n$ for $n<\omega$, we have managed to prove that
    \begin{equation}
        \bM\models\text{``$\B\models_{j[S]}\phi$''} \iff \A\models_S\phi. \label{Elementary embedding of the universe preserves independence logic}
    \end{equation}
    By claim (ii) above, relations of the form $j[S]$ exhaust all relations of $\B$. Now for any $R\subseteq\B^n$, by elementarity of $j$, we then obtain
    \begin{align*}
        \bM\models\text{``$\B\models_R\phi$''} &\iff \A\models_{j^{-1}[R]}\phi \\
        &\iff \V\models\text{``$\A\models_{j^{-1}[R]}\phi$''} \\
        &\iff \bM\models\text{``$j(\A)\models_{j(j^{-1}[R])}\phi$''},
    \end{align*}
    By claim (iii) above, the map $R\mapsto j(j^{-1}[R])$ is an element of $\bM$.
    Hence,
    \[
        \bM\models\text{``$R\mapsto j(j^{-1}[R])$ is an elementary team map $\B\to j(\A)$''}.
    \]

    Now, note that by claim (i) above, in $\bM$ we have $|\B| \leq \lambda < j(\kappa)$. Thus, $\bM$ satisfies the statement ``there is $\B$ and an elementary team embedding $\B\to j(\A)$ such that $|\B|<j(\kappa)$''. By elementarity of $j$, we have (in $\V$) that there exists $\B$ and an elementary team embedding $\iota\colon\B\to\A$ such that $|\B|<\kappa$. Additionally, each $X\in\X$ is in the range of $\iota$, as $\iota$ is, in particular, a partial team isomorphism in the signature $\tau'$ and $X$ is the interpretation of $\underline X\in\tau'$ in $\A$. This is enough to prove the theorem.

    Therefore, the only thing left to show is that~\eqref{Elementary embedding of the universe preserves independence logic} holds. We prove it by induction on the complexity of $\phi\in\fot$.
    \begin{enumerate}
        \item Suppose first that $\phi = R(\vec t)$ for $R\in\tau'\cup\{=\}$ and $\vec t$ a tuple of $\tau'$-terms. Fix $X\subseteq\A^n$ with the free variables of $\phi$ contained in $\{v_i\mid i\leq n\} $. Then
    \begin{align*}
        \bM\models\text{``$\B\models_{j[X]}\phi$''} &\iff \bM\models\text{``$\forall s\in j[X],\ s(\vec t)\in R^\B$''} \\
        &\iff \forall s\in j[X],\ s(\vec t)\in j[R^{\A}] \text{ (absoluteness)} \\
        &\iff \forall s\in X,\ s(\vec t)\in R^\A \text{ (elementarity)} \\
        &\iff \A\models_X\phi.
    \end{align*}
    Nothing changes if one considers $\phi = \neg R(\vec t)$ instead.
    
    \item Next suppose that $\phi = \vx\subseteq\vy$ and again fix $X$. Now
    \begin{align*}
        \hspace{2em}\bM\models\text{``$\B\models_{j[X]}\vx\subseteq\vy$\,''} &\iff \bM\models\text{``$\forall s\in j[X]\ \exists s'\in j[X]\ s(\vx)=s'(\vy)$''} \\
        &\iff \forall s\in j[X]\ \exists s'\in j[X]\ s(\vx)=s'(\vy) \text{ (absoluteness)} \\
        &\iff \forall s\in X\ \exists s'\in X\ s(\vx)=s'(\vy) \text{ (elementarity)} \\
        &\iff \A\models_X\vx\subseteq\vy.
    \end{align*}

    \item The case of the constancy atom $\dep(\vx)$ follows analogously to the previous one. The cases $\phi=\psi\land\theta$, $\phi = \psi\ivee\theta$ and $\phi = \cneg\psi$ follow immediately from the induction hypothesis.
    
    \item Next suppose that $\phi = \existsone v_n\psi$ and fix $X\subseteq\A^n$. Then
    \begin{align*}
        \bM\models\text{``$\B\models_{j[X]}\phi$''} &\iff \bM\models\text{``$\B\models_{j[X](b/n)}\psi$ for some $b\in\B$''} \\
        &\iff \exists a\in\A,\ \bM\models\text{``$\B\models_{j[X](j(a)/n)}\psi$''} \\
        &\stackrel{(*)}\iff \exists a\in\A,\ \bM\models\text{``$\B\models_{j[X(a/n)]}\psi$''} \\
        &\stackrel{\text{i.h.}}\iff \exists a\in\A,\ \A\models_{X(a/n)}\psi \\
        &\iff \A\models_X\phi.
    \end{align*}
    The step $(*)$ follows by the fact that
    \begin{equation*}
        j[X(a/n)] = j[X](j(a)/n).
    \end{equation*}
    To see this, let $(b_0,\dots,b_n)\in j[X(a/n)]$. Then $b_i = j(a_i)$ for some $a_i$ such that $(a_0,\dots,a_{n-1})\in X$, $b_n = j(a)$ and $(b_0,\dots,b_{n-1}) = j(a_0,\dots,a_{n-1})\in j[X]$, so
    \[
        (b_0,\dots,b_n) = j(a_0,\dots,a_{n-1})^\frown j(a) \in j[X](j(a)/x).
    \]
    Conversely, we have that every $(b_0,\dots,b_n)\in j[X](j(a)/x)$ is of the form $j(a_0,\dots,a_{n-1},a)$ for $(a_0,\dots,a_{n-1})\in X$.
    
    \item The remaining case $\phi = \forallone x\psi$ is similar.
    \end{enumerate}
    \noindent This completes the proof of~\eqref{Elementary embedding of the universe preserves independence logic} and thus of our theorem.
\end{proof}

Even while having a traditional Löwenheim--Skolem property in general seems hopeless, in the framework of \cref{subsec:accessible.team.category}, the following very weak formulation of the Löwenheim--Skolem theorem, following essentially from the LS theorem of first-order logic, seems to suffice.

\begin{proposition}\label{weak-LS}
    Let $\A$ be a $\tau$-structure and $\X\subseteq \R(\A)$. Then there is a $\tau$-structure $\B$ of power $\leq|\X|+|\tau|+\aleph_0$ and an element-total partial elementary team map $\iota\colon\B\to\A$ such that $\B\models\ESOTh(\A)$ and $\X\subseteq \ran(\iota)$.
\end{proposition}
\begin{proof}
    By the Löwenheim--Skolem theorem of $\eso$, there is an elementary substructure $\hat\B$ of $\A_\X$ with $|\hat\B|\leq|\tau(\X)|+\aleph_0 = |\X|+|\tau|+\aleph_0$ and $\hat\B\models\ESOTh(\A)$. Now let $\B=\hat\B\restriction\tau$. Then, by \cref{diagram.team.embeddings}, the conditions $\f(\underline X^{\hat\B}) = X$ for all $X\in\X$ and $\f(b) = b$ for all $b\in\B$ uniquely determine a partial elementary team map $\f\colon\B\to\A$ with $\ran(\f)=\cl(\X\cup\{\{b\} \mid b\in\B\})$.
\end{proof}

\subsection{A Category-theoretic Framework}\label{subsec:accessible.team.category}

We solve all of our AEC problems introduced in the previous sections by moving to work inside a framework where structures come with a predetermined set of relations of interest, much resembling general models, or Henkin models, of second-order logic. General models in the setting of team semantics were investigated in~\cite{MR3028798}, but our definition will be slightly different. We define our framework as a category and show that it is an example of an \emph{accessible category}~\cite{adamek1994locally}, which are known to generalize abstract elementary classes (see e.g.~\cite{beke2012abstract}).

\begin{definition}\quad
    \begin{enumerate}
        \item Given a signature $\tau$, a \emph{general $\tau$-structure} is a pair $(\A,\X)$, where $\A$ is a $\tau$-structure and $\X\subseteq\R(\A)$ is a set of relations such that $\cl(\X) = \X$ and for every $a\in \A$, $\{a\}\in \X$.

        \item If $\mathcal A = ((\A_i,\X_i),\f_{i,j})_{i,j\in I, i\leq j}$ is a directed system of general structures and partial team isomorphisms $\f_{i,j}\colon\A_i\to\A_j$ such that $\dom(\f_{i,j}) = \X_i$ and $\ran(\f_{i,j})\subseteq\X_j$ for all $i\leq j$, then the direct limit $\lim\mathcal A$ of $\mathcal A$ is the general model $(\A,\X)$, where $\A = \lim_{i\in I}\A_i$ and $\X$ is the collection of all admissible sets of $\A$, i.e. all relations $R\in\R(\A)$ such that there is $i\in I$ with $R\in\ran(\g_i)$, where $\g_i\colon\A_i\to\A$, $i\in I$, are the direct limit maps. Like before, when the maps $\f_{i,j}$ are clear from the context, we simply write $\lim_{i\in I}(\A_i,\X_i)$ for the direct limit.

        \item An \emph{abstract elementary team category} (AETC) is a category $\mathcal C$
        \begin{itemize}
            \item whose objects are general $\tau$-structures $(\A,\X)$ for some fixed $\tau$,
            \item morphisms $(\A,\X)\to(\B,\Y)$ are partial team isomorphisms $\f\colon\A\to\B$ such that $\dom(\f) = \X$ and $\ran(\f)\subseteq\Y$, and
            \item composition of morphisms is simply function composition,
        \end{itemize}
        satisfying the following conditions.
        \begin{enumerate}
            \item \emph{Closure under isomorphisms}: If $(\A,\X)\in\mathcal C$ and $\pi\colon\A\to\B$ is a $\tau$-isomorphism, then $(\B,\Y)\in\mathcal C$ for $\Y = \hat\pi[\X]\coloneqq \{\hat\pi(R) \mid R\in\X\}$ and $\hat\pi\restriction\X$ is a morphism $\f\colon(\A,\X)\to(\B,\Y)$.
            \item \emph{Closure under inverses}: If $\f\colon(\A,\X)\to(\B,\Y)$ is a morphism such that $\ran(\f) = \Y$, then $\f^{-1}$ is a morphism $(\B,\Y)\to(\A,\X)$ (and hence $\f$ is an isormorphism in the category $\mathcal C$).
            \item \emph{Coherence}: If $\f\colon(\A,\X)\to(\C,\Z)$ and $\g\colon(\B,\Y)\to(\C,\Z)$ are morphisms and $\h\colon\A\to\B$ is a partial team isomorphism with $\f = \g\circ\h$, then $\h$ is a morphism $(\A,\X)\to(\B,\Y)$.
            \item \emph{Direct limits of directed systems}: $\mathcal C$ is closed under direct limits, i.e. if $((\A_i,\X_i),\f_{i,j})_{i,j\in I,i\leq j}$ is a directed system such that each $(\A_i,\X_i)$ is an object of $\mathcal C$ and each $\f_{i,j}$ is a morphism of $\mathcal C$, then $\lim_{i\in I}(\A_i,\X_i)\in\mathcal C$, and the direct limit maps $\g_i$ are morphisms $(\A_i,\X_i)\to\lim_{j\in I}(\A_j,\X_j)$. Furthermore, this direct limit is a colimit of the category $\mathcal C$ and hence satisfies the universal property of colimits (\emph{smoothess}).
            \item \emph{Löwenheim--Skolem property}: There is a cardinal $\kappa$ such that for any $(\A,\X)\in\mathcal C$ and $\X'\subseteq\X$, there is $(\B,\Y)\in\mathcal C$ with $|\B|\leq|\X'|+\kappa$ and $|\Y|\leq|\X'|+\aleph_0$, and a morphism $\f\colon(\B,\Y)\to(\A,\X)$ such that $\cl(\X')\subseteq\ran(\f)$. The least such $\kappa$ will be denoted by $\LS(\mathcal C)$.
        \end{enumerate}
    \end{enumerate}
\end{definition}

The identity morphism $\id$ of an object $(\A,\X)$ is such that $\dom(\id)=\ran(\id)=\X$. Therefore the requirement that $\cl(\X) = \X$ is necessary.

Observe that in an AETC, an automorphism of an object $(\A,\X)$ is a partial team isomorphism $\f\colon\A\to\A$ such that $\dom(\f)=\ran(\f)=\X$. We denote by $\Aut(\A,\X)$ the set of all automorphisms of $(\A,\X)$. If $\X'\subseteq\X$, we denote by $\Aut((\A,\X) / \X')$ the set of all $\f\in\Aut(\A,\X)$ such that $\f(X) = X$ for all $X\in\X'$.

If $(\A,\X)$ and $(\B,\Y)$ are general structures, we say that $\f$ is a partial team map $(\A,\X)\to(\B,\Y)$ if $\dom(\f)\subseteq\X$ and $\ran(\f)\subseteq\Y$.

Next we recall some preliminary notions from category theory.

\begin{definition}\quad
    \begin{enumerate}
        \item Let $\kappa$ be a regular cardinal. An object $X$ in some category $\mathcal C$ is \emph{$\kappa$-presentable} if the functor $\Hom(X,-)$ preserves all $\kappa$-directed colimits, i.e. for every $\kappa$-directed system $(Y_i,f_{i,j})_{i\in I, i\leq j}$ of objects and morphisms of $\mathcal C$, for every object $X\in\mathcal C$ and morphism $h\colon X\to\mathrm{colim}_{i\in I} Y_i$ there is $i\in I$ and a morphism $h_i\colon X\to Y_i$ such that $h = g_i\circ h_i$, where $g_i\colon Y_i\to\mathrm{colim}_{i\in I}Y_i$ are the colimit morphisms.
        \item Let $\kappa$ be a regular cardinal. A category $\mathcal C$ is \emph{$\kappa$-accessible} if
        \begin{enumerate}
            \item $\mathcal C$ has $\kappa$-directed colimits, and
            \item there is a set $\mathcal A$ of $\kappa$-presentable objects such that every object in $\mathcal C$ is (isomorphic to) a $\kappa$-directed colimit of objects from $\mathcal A$.
        \end{enumerate}
        \item A category is \emph{accessible} if it is $\kappa$ accessible for some regular cardinal $\kappa$.
    \end{enumerate}
\end{definition}

There are many examples of accessible categories which are relevant to model theory. Among the most natural ones are the following \parencite[Def. 2.2]{kamsma2020kim}.
\begin{example}\quad
    \begin{enumerate}
        \item Let $T$ be a complete first-order theory. The category $\Mod(T)$ of all models of $T$ with elementary embeddings is accessible.
        \item Given a complete first-order theory $T$, let $\Sub(T)$ be the category of pairs $(\A,X)$, where $X\subseteq\A\models T$, and morphisms $(\A,X)\to(\B,Y)$ are partial elementary maps $f\colon\A\to\B$ with $\dom(f) = X$ and $\ran(f)\subseteq Y$. Then $\Sub(T)$ is also an accessible category.
    \end{enumerate}
\end{example}

\begin{theorem}\label{theorem:aetc.accessible}
    Abstract elementary team categories are accessible.
\end{theorem}
\begin{proof}
    Let $\mathcal C$ be an abstract elementary team category and $\kappa = \LS(\mathcal C)^+$. We show that $\mathcal C$ is $\kappa$-accessible. Since $\mathcal C$ has colimits,  in particular it has $\kappa$-directed colimits.
    
    Let $\mathcal C_{<\kappa}$ be the full subcategory of $\mathcal C$ whose objects are all $(\A,\X)\in\mathcal C$ with $|\A|,|\X|<\kappa$. We claim that objects of $\mathcal C_{<\kappa}$ are $\kappa$-presentable. Let $(\A,\X)\in\mathcal C_{<\kappa}$ and consider the colimit $(\B,\Y)$ of a $\kappa$-directed system $((\B_i,\Y_i),\f_{i,j})_{i,j\in I,i\leq j}$, where we denote by $\g_i\colon(\B_i,\Y_i)\to(\B,\Y)$ the colimit morphisms. We claim that we can factor any morphism $\h\colon(\A,\X)\to(\B,\Y)$ through some $(\B_i,\Y_i)$. First, notice that since $|\X|< \kappa$ and $\kappa$ is regular, the fact that $I$ is $\kappa$-directed entails that there is some $i\in I$ such that $\h(X)\in\ran(\g_i)$ for all $X\in \X$. Define a partial team isomorphism $\h_i\colon\A\to\B_i$ by letting
    \[
        \h_i(X) = \g_i^{-1}(\h(X))
    \]
    for all $X\in\X$. Notice that $\X\subseteq\dom(\h_i)$ and $\ran(h_i)\subseteq\dom(\g_i) = \ran(\f_{i,i})\subseteq \Y_i$. Thus $\h=\g_i\circ\h_i$, so by coherence, $\h_i$ is a morphism $(\A,\X)\to(\B_i,\Y_i)$. Clearly, such factorization is unique. Now by e.g.\ taking in only pairs $(\A,\X)$, where the domain of $\A$ is a cardinal, we can make $\mathcal C_{<\kappa}$ a set.

    What is left to show is that every object in $\mathcal C$ can be obtained as a $\kappa$-directed colimit of objects of $\mathcal C_{<\kappa}$. Let $(\A,\X)\in\mathcal C$. Let $(\X_i)_{i\in I}$ list all $\Y\subseteq\X$ of power $<\kappa$ such that $\cl(\Y)=\Y$. By letting $i\leq j$ whenever $\X_i\subseteq\X_j$, we make $(I,\leq)$ a directed set. Since $\kappa$ is regular, $(I,\leq)$ is also $\kappa$-directed.
    
    Now, by the Löwenheim--Skolem property, for each $i\in I$, we obtain a general model $(\B_i,\Y_i)$ with $|\B_i|\leq|\X_i|+\LS(\mathcal C) < \kappa$ and a morphism $\iota_i\colon(\B_i,\Y_i)\to(\A,\X)$ such that $\ran(\iota_i)=\X_i$. As both $\B_i$ and $\Y_i$ have power $<\kappa$, we have $(\B_i,\Y_i)\in\mathcal C_{<\kappa}$.
        
    Now, given $i,j\in I$ with $i\leq j$, we define a partial team isomorphism $\f_{i,j}\colon\B_i\to\B_j$ by letting, for every $X\in\Y_i$,
    \[
        \f_{i,j}(X) = \iota_{j}^{-1}(\iota_i (X)),
    \]
    which is well defined since for $i\leq j$ we have $\X_i\subseteq\X_j$. Additionally, we have
    \begin{align*}
       \Y_i=\dom(\f_{i,j}) \quad\text{  and}\quad \ran(\f_{i,j})\subseteq \Y_j.
    \end{align*}
    Now clearly $\iota_i(X) = \iota_j(\f_{i,j}(X))$ for all $X\in\Y_i$, so by coherence $\f_{i,j}$ is a morphism $(\B_i,\Y_i)\to(\B_j,\Y_j)$. It follows that $((\B_i,\X_i),\f_{i,j})_{i,j\in I,i\leq j}$ forms a $\kappa$-directed system, so we can consider its colimit $(\B,\Y)$. Let $\g_i\colon(\B_i,\Y_i)\to(\B,\Y)$ be the colimit maps and notice that, by construction, we have
    \[
        \iota_j\circ \f_{i,j} = \iota_i
    \]
    for every $i\leq j$. Therefore, by the universal property of colimit, there is a unique morphism $\mathcal u\colon(\B,\Y)\to(\A,\X)$ such that, for all $i\in I$,
    \[
        \mathcal u\circ\g_i = \iota_i.
    \]
    Now, for every $X\in\X$, there is by construction some $i\in I$ such that $X\in\X_i$. Then $X = \iota_i(Y)$ for some $Y\in\Y_i$, so $\mathcal u(\g_i(Y))=X$. This shows that $\mathcal u$ is surjective with respect to $\X$ and so it is an isomorphism in the category $\mathcal C$. Finally, this shows $(\A,\X)$ can be obtained as a  colimit of $\kappa$-presentable objects.
\end{proof}

\begin{definition}
    Let $T$ be a first-order-complete theory in $\foil$. We denote by $\GMod(T)$ the category whose objects are general models of $T$ (i.e. general structures $(\A,\X)$ such that $\A\models T$) and morphisms $(\A,\X)\to(\B,\Y)$ are elementary team maps $\f$ with $\dom(\f)=\X$ and $\ran(\f)\subseteq\Y$.
\end{definition}

\begin{corollary}
    $\GMod(T)$ is an abstract elementary team category.
\end{corollary}
\begin{proof}
    Closure under isomorphisms and inverses are clear. Coherence follows from \cref{prop:coherence}. By \cref{prop:direct.limit.maps.elementary}, $\GMod(T)$ has direct limits and since the limit object $(\A,\X)$ of a directed system is such that $\X$ contains only admissible sets, we have the universal property of colimits (note that all singletons are admissible). Finally, by \cref{weak-LS}, $\LS(\GMod(T))\leq|\tau|+\aleph_0$.
\end{proof}

As we remarked before, the setting of accessible categories is known to generalize the notion of abstract elementary class. The above results thus show that our setting of general models and (elementary) team maps fits this categorical framework to study the model theory of independence logic and existential second-order logic. Additionally, we remark a further connection with another categorification of AECs, namely \emph{abstract elementary categories (AECats)}, which were introduced by Kirby and Kamsma \cite{kirby_abstract_2008,kamsma2020kim}. These are pairs of accessible categories with some further structure, the protypical example of which are  $(\Sub(T),\Mod(T))$ and $(\Mod(T),\Mod(T))$.

\begin{definition}[\cite{kamsma2020kim}]
    An abstract elementary category (AECat) is a pair $(\mathcal C,\mathcal M)$ of accessible categories such that
    \begin{enumerate}
        \item $\mathcal M$ is a full subcategory of $\mathcal C$,
        \item all morphisms of $\mathcal C$ are mono, and
        \item $\mathcal M$ has directed colimits, which are preserved by the inclusion functor $\mathcal M\to\mathcal C$.
    \end{enumerate}
\end{definition}
\begin{corollary}
    The pair $(\GMod(T),\GMod(T))$ is an AECat.
\end{corollary}

\noindent We also remark that one could introduce a category $\GSubMod(T)$ analogously to $\Sub(T)$, but we shall not pursue this in the present paper, as we required general models of an AETC to contain all singleton relations.

\section{Galois Types and the Monster Model}\label{Section_Monster}

In this section we continue the study of the model-theoretic properties of AETCs and, in particular, of the model categories of first-order-complete theories in $\foil$ (or $\eso$). For notational convenience, when given a general structure $(\A,\X)$, we often omit the set $\X$ from the notation and denote it by $\rel(\A)$, i.e. if we say that $\A$ is a general structure, we mean the structure $(\A,\rel(\A))$.

\begin{definition}
    Let $\K$ be an AETC.
    \begin{enumerate}
        \item We say that $\K$ has \emph{arbitrarily large models} (ALM) if for every cardinal $\kappa$ there is $\A\in\K$ with $|\A|\geq\kappa$.        
        \item We say that $\K$ has the \emph{joint embedding property} (JEP) if for all $\A,\B\in\K$, there is $\C\in\K$ and morphisms $\f\colon\A\to\C$ and $\g\colon\B\to\C$.
        \item We say that $\K$ has the \emph{amalgamation property} (AP) if for all $\A,\B,\C\in\K$ and morphisms $\f\colon\A\to\B$ and $\g\colon\A\to\C$, there is $\D\in\K$ and morphisms $\h_0\colon\B\to\D$ and $\h_1\colon\C\to\D$ such that $\h_0\circ\f = \h_1\circ\g$.
    \end{enumerate}
\end{definition}

\begin{remark}
    Notice that if $\K$ has AP and contains a prime model $\A$, i.e. for all $\B\in\K$ there is a morphism $\A\to\B$, then $\K$ also has JEP.
\end{remark}

We start by showing some sufficient conditions for an AETC to satisfy the above versions of ALM, JEP and AP.

\begin{definition}\quad
    \begin{enumerate}
        \item Given general structures $(\A_i,\X_i)$, $i\in I$, and an ultrafilter $\U$ on $I$, the ultraproduct $\prod_{i\in I}(\A_i,\X_i)/\U$ of $\A_i$ is the general structure $(\A,\X)$ such that $\A = \prod_{i\in I}\A_i/\U$ and $\X = \bigcup_{n<\omega}\{\prod_{i\in I}X_i/\U \mid \text{$X_i\in\X_i\cap\Pow(\A_i^n)$ for all $i\in I$}\}$.
        \item We say that an AETC $\K$ is \emph{closed under ultrapowers} if for all $\A\in\K$, for all infinite cardinals $\kappa$ and for all ultrafilters $\U$ on $\kappa$, $\A^\kappa/\U\in\K$ and the ultrapower embedding $\iota\colon\A\to\A^\kappa/\U$ is a morphism of $\K$. We say that $\K$ is \emph{closed under ultraproducts} if for all infinite $\kappa$, for all $\A_i\in\K$, $i<\kappa$, and for all ultrafilters $\U$ on $\kappa$, $\prod_{i<\kappa}\A_i/\U\in\K$.
        \item We say that an AETC $\K$ is first-order-complete if for all $\A,\B\in\K$ we have $\A\equiv\B$.
    \end{enumerate}
\end{definition}

\begin{proposition}\label{class.properties}
    Let $\K$ be an AETC.
    \begin{enumerate}
        \item If $\K$ has infinite models and is closed under ultrapowers, then $\K$ has ALM.
        \item If $\K$ is first-order-complete and closed under ultrapowers, then $\K$ has JEP.
        \item If $\K$ is closed under ultrapowers and morphisms of $\K$ are elementary team maps, then $\K$ has AP.
    \end{enumerate}
\end{proposition}
\begin{proof}\quad
    \begin{enumerate}
        \item Let $\A\in\K$ and $\kappa$ be infinite. By a theorem of \textcite{MR0142459}, there is an ultrafilter $\U$ on $\kappa$ such that $|\A^\kappa/\U| \geq 2^\kappa > \kappa$. By closure under ultrapowers, $\A^\kappa/\U\in\K$.

        \item Let $\A,\B\in\K$. As $\K$ is first-order-complete, $\A$ and $\B$ are elementarily equivalent. By the Keisler--Shelah theorem, there is a cardinal $\kappa$ and an ultrafilter $\U$ on $\kappa$ such that $\A^\kappa/\U\cong\B^\kappa/\U$. Let $\pi$ be an isomorphism $\B^\kappa/\U\to\A^\kappa/\U$, and denote $\C = \A^\kappa/\U$. Now $\pi$ lifts to a team isomorphism $\hat\pi\colon\B^\kappa/\U\to\C$ in the obvious way. Since $\K$ is closed under ultrapowers, we have $\A^\kappa/\U,\B^\kappa/\U\in\K$ and the ultrapower embeddings $\iota_\A\colon\A\to\A^\kappa/\U$ and $\iota_\B\colon\B\to\B^\kappa/\U$ are morphisms. Hence $\hat\pi\circ\iota_B$ is also a morphism. Thus we may take $\f = \iota_\A$ and $\g=\hat\pi\circ\iota_\B$ as our desired joint embeddings, which finishes the proof.

        \item Let $\A$, $\B$, $\C$, $\f$ and $\g$ be as in the definition of AP. Let $\X = \rel(\A)$ and let $\hat\B$ and $\hat\C$ be the $\tau(\X)$-expansions of $\B$ and $\C$, respectively, by interpreting $\underline{X}^{\hat\B} = \f(X)$ and $\underline{X}^{\hat\C} = \g(X)$ for all $X\in\X$. We first show that $\hat\B\equiv\hat\C$.

        Suppose that this is not the case. Then there is a first-order $\tau(\X)$-sentence $\phi$ such that $\hat\B\models\phi$ and $\hat\C\models\neg\phi$. Let $R_0,\dots,R_{n-1}$ enumerate all $R\in\X$ such that $\underline R$ occurs in $\phi$. Without loss of generality, each $R_i$ is nonempty (otherwise replace the symbol $\underline R_i$ in $\phi$ by $\bot$ whenever $R_i = \emptyset$). Then, coding the relation symbols $\underline R_0,\dots\underline R_{n-1}$ into one $\sum_{i<n}\ar(R_i)$-ary relation symbol $S$ as in the proof of \cref{Partial elementary maps: FOT vs FO}, we obtain a $\tau\cup\{S\}$-sentence $\psi$ such that $(\B, \prod_{i<n}R_{i}^{\hat\B})\models\psi$ and $(\C,\prod_{i<n}R_{i}^{\hat\C})\models\neg\psi$. Since $\X$ is closed under Cartesian products, we have $\prod_{i<n}R_i\in\X$, and by \ref{Partial isomorphisms preserve products} we additionally obtain
        \[
            \f\left(\prod_{i<n}R_{i}\right) = \prod_{i<n}\f(R_{i}) \quad\text{and}\quad \g\left(\prod_{i<n}R_{i}\right) = \prod_{i<n}\g(R_{i}).
        \]
        Then, let $\chi$ be the $\fot$-translation of $\psi$. Now $\B\models_{\f(\prod_{i<n}R_i)}\chi$ and $\C\nmodels_{\g(\prod_{i<n}R_i)}\chi$. As $\f$ and $\g$ are elementary team maps, it follows that $\A\models_{\prod_{i<n}R_i}\chi$ but $\A\nmodels_{\prod_{i<n}R_i}\chi$, which is a contradiction.

        Next, since $\hat\B\equiv\hat\C$, it follows by Keisler--Shelah that the structures have isomorphic ultrapowers $\hat\B^\kappa/\U$ and $\hat\C^\kappa/\U$. Let $\pi$ be an isomorphism $\hat\B^\kappa/\U\to\hat\C^\kappa/\U$, denote $\hat\D = \hat\C^\kappa/\U$ and let $\D=\hat\D\restriction\tau$. As $\K$ is closed under ultrapowers, $\B^\kappa/\U,\C^\kappa/\U\in\K$ and the ultrapower embeddings $\iota_\B\colon\B\to\B^\kappa/\U$ and $\iota_\C\colon\C\to\D$ are morphisms. Let $\h_0 = \hat\pi\circ\iota_B$ and $\h_1 = \iota_\C$. Since $\h_0$ and $\h_1$ are elementary team maps even in the signature $\tau(\X)$, we have
        \begin{align*}
            \h_0(\f(R)) &= \h_0(\underline R^{\hat\B}) = \underline R^{\hat\D} = \h_1(\underline R^{\hat\C}) = \h_1(\g(R))
        \end{align*}
        for all $R\in\X$. Hence $\D$ with $\h_0\colon\B\to\D$ and $\h_1\colon\C\to\D$ suffices as an amalgam of $\B$ and $\C$ over $\A$.
        \qedhere
    \end{enumerate}
\end{proof}

Note that the amalgamation property holds in a tiny bit stronger form in the category $\GMod(T)$.

\begin{proposition}\label{AP for elementary classes}
    Let $\K=\GMod(T)$ for some first-order-complete theory $T$ of $\foil$. Then for all $\A,\B,\C\in\K$ and partial elementary team maps $\f\colon\A\to\B$ and $\g\colon\A\to\C$, there is $\D\in\K$ and morphisms $\h_0\colon\B\to\D$ and $\h_1\colon\C\to\D$ such that $\h_0(\f(X)) = \h_1(\g(X))$ for all $X\in\dom(\f)\cap\dom(\g)$.
\end{proposition}
\begin{proof}
    The proof is essentially same as that of \cref{class.properties}.
\end{proof}

Amalgamation is a fundamental property in the context of abstract elementary classes as it grants a good behaviour to so-called Galois types. Here we define Galois types of relations and show that also in our context they induce a natural equivalence relation.

\begin{definition}[Galois Types]\label{galois=type0}
Let $\K$ be an AETC.
\begin{enumerate}
    \item We define a relation $\equiv_{\mathrm{g}}^\K$ on the class of all pairs $(\A,\vec X)$, where $\A\in\K$ and $\vec X = (X_i)_{i<\alpha}\in\rel(\A)^\alpha$ for some ordinal $\alpha$, as follows. If $\A,\B\in\K$, $\vec X = (X_i)_{i<\alpha}\in\rel(\A)^\alpha$ and $\vec Y = (Y_i)_{i<\beta}\in\rel(\B)^\beta$, then
    \[
        (\A,\vec X) \equiv_{\mathrm{g}}^\K (\B,\vec Y)
    \]
    if $\alpha = \beta$, $\ar(X_i)=\ar(Y_i)$ for all $i<\alpha$ and there are $\C\in\K$ and morphisms $\h_0\colon\A\to\C$ and $\h_1\colon\B\to\C$ such that $\h_0(X_i) = \h_1(Y_i)$ for all $i<\alpha$. 
    \item Given $(\A,\vec X)$, we denote the class of all $(\B,\vec Y)$ such that $(\A,\vec X) \equiv_{\mathrm{g}}^\K (\B,\vec Y)$ by $\gtype(\vec X /\emptyset; \A)$ and call it the \emph{Galois type} of $\vec X$ (over $\emptyset$, in $\A$).
\end{enumerate}
\end{definition}

The following lemma shows that the notion of Galois type is well defined and, in the presence of AP, having the same type is an equivalence relation.

\begin{lemma}\label{property.gtypes}
    Suppose $\K$ has AP. Then
    \begin{enumerate}
        \item $\equiv_{\mathrm{g}}^\K$ is an equivalence relation, and
        \item for all $\A,\B\in\K$, $\vec X\in\rel(\A)^\alpha$ and morphism $\f\colon\A\to\B$, we have
        \[
            \gtype(\vec X / \emptyset; \A) = \gtype(\f(\vec X) / \emptyset; \B).
        \]
    \end{enumerate}
\end{lemma}
\begin{proof}\quad
    \begin{enumerate}
        \item Reflexivity and symmetry are clear, so we verify that transitivity holds. Let $\A,\B,\C\in\K$, $\vec X = (X_i)_{i<\alpha}\in\rel(\A)^\alpha$, $\vec Y = (Y_i)_{i<\alpha}\in\rel(\B)^\alpha$ and $\vec Z = (\Z_i)_{i<\alpha}\in\rel(\C)^\alpha$ and suppose that
        \[
            (\A,\vec X) \equiv_{\mathrm{g}}^\K (\B,\vec Y) \quad\text{and}\quad (\B,\vec Y) \equiv_{\mathrm{g}}^\K (\C,\vec Z).
        \]
        Then there are $\D,\E\in\K$ and morphisms $\h_0\colon\A\to\D$, $\h_1\colon\B\to\D$, $\h_2\colon\B\to\E$ and $\h_3\colon\C\to\E$ such that $\h_0(X_i)=\h_1(Y_i)$ and $\h_2(Y_i)=\h_3(Z_i)$ for all $i<\alpha$.  Apply AP to the morphisms $h_1\colon\B\to\D$ and $\h_2\colon\B\to\E$ to obtain $\F\in\K$ and morphisms $\k_0\colon\D\to\F$ and $\k_1\colon\E\to\F$ such that $\k_0\circ\h_1 = \k_1\circ\h_2$. It follows that for all $i<\alpha$,
    	\begin{align*}
            (\k_0\circ\h_0)(X_i) &= \k_0(\h_0(X_i)) = \k_0(\h_1(Y_i)) = (\k_1\circ\h_2)(Y_i) \\
            &= \k_1(\h_2(Y_i)) = \k_1(\h_3(Z_i)) = (\k_1\circ\h_3)(Z_i),
    	\end{align*}
        which shows that $(\A,\vec X)\equiv_{\mathrm{g}}^\K(\C,\vec Z)$.

        \item Trivial.
        \qedhere
    \end{enumerate}
\end{proof}

We have defined Galois types for arbitrary ordinal-length sequences. The next proposition shows that, when we restrict to finite tuples, we can show that $\fot$ actually provides an alternative, syntactic way to define types in $\GMod(T)$.

\begin{definition}[Syntactic Types]
    Let $\A$ be a general $\tau$-structure, $\X\subseteq\rel(\A)$ with $\emptyset\notin\X$ and $n>0$ a natural number.
    \begin{enumerate}
        \item An $\fot$-type over $\X$ in variables $\vx^0,\dots,\vx^{n-1}$, where each $\vx^i$ is a finite tuple of variables and $\vx^i$ and $\vx^j$ do not share variables for $i\neq j$, is a nonempty set of $\tau(\X)$-formulas of $\fot$ with free variables among $\vx^0,\dots,\vx^{n-1}$.
        \item A type $p(\vx^0,\dots,\vx^{n-1})$ over $\X$ is consistent if there is a general $\tau$-structure $\B$ and a morphism $\f\colon\A\to\B$ with $\hat\B\models_{\prod_{i<n}Y_i}\phi$ for all $\phi\in p$ for some $Y_0,\dots,Y_{n-1}\in\rel(\B)\setminus\{\emptyset\}$, where $\hat\B$ is the $\tau(\X)$-extension of $\B$ with $\underline X^{\hat\B} = \f(X)$ for all $X\in\X$. We say that the tuple $\vec Y = (Y_0,\dots,Y_{n-1})$ realizes $p$.
        \item A type $p$ is complete if for all $\tau$-formulas $\phi$ with the proper free variables, we have either $\phi\in p$ or $\cneg\phi\in p$.
        \item We denote by $\fottype(\vec Y / \X; \A)$ the unique complete type realized by $\vec Y\in(\rel(\A)\setminus\{\emptyset\})^n$. In particular,
        \[
            \fottype((X_0,\dots,X_{n-1}) / \emptyset; \A) = \{\phi\in\fot \mid \A\models_{\prod_{i<n}X_i}\phi\}.
        \]
    \end{enumerate}
\end{definition}

\begin{proposition}\label{galois-types.fot-types}
    Suppose that $\K$ is closed under ultrapowers and morphisms are elementary. Let $\A,\B\in\K$, $\vec X = (X_i)_{i<n}\in(\rel(\A)\setminus\{\emptyset\})^n$ and $\vec Y = (Y_i)_{i<n}\in(\rel(\B)\setminus\{\emptyset\})^n$. Then the following are equivalent.
    \begin{enumerate}
        \item $\gtype(\vec X / \emptyset; \A) = \gtype(\vec Y / \emptyset; \B)$.
        \item $\fottype(\vec X / \emptyset; \A) = \fottype(\vec Y / \emptyset; \B)$.
    \end{enumerate}
\end{proposition}
\begin{proof}
    First suppose that $\gtype(\vec X / \emptyset; \A) = \gtype(\vec Y / \emptyset; \B)$. Then there is $\C\in\K$ and morphisms $\f\colon\A\to\C$ and $\g\colon\B\to\C$ such that $\f(X_i) = \g(Y_i)$ for all $i<n$. Then $\f(\prod_{i<n} X_i) = \g(\prod_{i<n} Y_i)$, which entails $\fottype(\vec X / \emptyset; \A) = \fottype(\vec Y / \emptyset; \B)$ by elementarity of the morphisms.

    Then suppose that $\fottype(\vec X / \emptyset; \A) = \fottype(\vec Y / \emptyset; \B)$. Then in particular the structures $(\A, \prod_{i<n} X_i)$ and $(\B, \prod_{i<n} Y_i)$ are elementary equivalent by \cref{teams_elementary_equivalent_as_predicates}. By applying Keisler--Shelah exactly as in \cref{class.properties}, we find $\C\in\K$ and morphisms $\f\colon\A\to\C$ and $\g\colon\B\to\C$ such that $\f(X_i)=\g(Y_i)$ for all $i<n$. Hence $\gtype(\vec X / \emptyset; \A) = \gtype(\vec Y / \emptyset; \B)$.
\end{proof}

\begin{remark}
    For the \emph{aficionados} of topology, we allow ourselves a short excursion in this otherwise very concise article.  In fact, we would like to remark that, in the case where $\K=\GMod(T)$, we can provide a natural Stone space topology to the set of types. In fact, since they have a syntactical representation as $\fot$-types by \cref{galois-types.fot-types}, we can use formulas to define a clopen basis. Write $G(\A)$ for the set of all Galois types consistent with $\A$, i.e. all $\gtype( \vec X / \emptyset; \B)$, where  $\vec X\in(\rel(\B)\setminus\{\emptyset\})^{<\omega}$ for some $\B\in\K$ such that there is a morphism $\f\colon\A\to \B$.
    We provide a topology  to $G(\A)$ by taking as basic opens the sets induced by formulas of $\fot$:
    \[
        \llbracket \phi \rrbracket :=    \{ \gtype((X_i)_{i<n}/\emptyset;\B) \mid \B \models_{\prod_{i<n}X_i} \phi \text{ and there is } \f:\A\to \B \}.
    \]
    We verify that the resulting topological space  is a Stone space, namely that it is totally disconnected, Hausdorff and compact.
    \begin{enumerate}
        \item \textit{$G(\A)$ is totally disconnected.} In fact if  a set $Z\subseteq G(\A)$ contains two distinct types $ \gtype(\vec X/\emptyset;\B) $ and $ \gtype(\vec Y/\emptyset;\C) $, then by \cref{galois-types.fot-types} there is a formula $\phi\in \fot$ such that $Z\cap \llbracket \phi \rrbracket$ and   $Z\cap \llbracket \cneg\phi \rrbracket$ split $Z$ into two nonempty opens.
        \item \textit{$G(\A)$ is Hausdorff.} If $ \gtype(\vec X/\emptyset;\B) \neq  \gtype(\vec Y/\emptyset;\C) $ then again by \cref{galois-types.fot-types} we can find a $\fot$-formula $\phi$ such that $\B\models_{\prod_{i<n}X_i} \phi$ but $\C\nmodels_{\prod_{i<n}Y_i}\phi$.
        \item \textit{$G(\A)$ is compact.} Without loss of generality an open cover is of the form $\{\llbracket \phi_i \rrbracket \mid i\in I \}$, thus $ \bigcap_{i\in I}\llbracket \cneg\phi_i \rrbracket =\emptyset$ and in particular $\{  \cneg\phi_i \mid i\in I \}\models \bot$. Since $\fot$ is compact, we obtain that $\{ \cneg \phi_i \mid i\in I_0 \}\models \bot$ for a finite set $I_0\subseteq I$. This is equivalent to say that $ \{ \llbracket \phi_i \rrbracket \mid i \in I_0\}$ is an open cover, showing that the topology is compact.
    \end{enumerate}
    Now let $S(\A)$ be the subset of $G(\A)$ whose types are determined by singletons of tuples of elements, i.e. they are of the form $ \gtype(\{\vec a\}/\emptyset;\B)$. Then $S(\A)$ with the induced subspace topology is a Stone subspace of $G(\A)$ and, additionally, it is exactly (up to homeomorphism) the standard Stone space of first-order types. Interestingly, the above inclusion induces via Stone's duality a homomorphism from the algebra of $\fot$-formulas to the algebra of $\fol$-formulas. This homomorphism is given by the quotient map that sends an equivalence class $\Phi$ of $\fot$-formulas modulo logical equivalence to the equivalence class $\Psi$ of $\fol$-formulas modulo logical equivalence such that for all singleton teams $X$, and $\phi\in\Phi$ and $\psi\in\Psi$,
    \[
        \A\models_X\phi \iff \A\models_X\psi.
    \]
    Finally, we also remark that, using amalgamation, it is easy to verify that the Stone space $G(\A)$ does not depend on $\A$ but only on its first-order theory. Thus we can meaningfully speak of the Stone space of Galois types (or $\fot$-types) of relations of any theory which is first-order-complete
\end{remark}

We  now  work towards building a version of the monster model in our setting. We notice this follows the general idea of the usual construction of the monster model in AECs. However, we remark that one can construct monster-like objects in any accessible category which is closed under colimits and has the amalgamation property -- we refer the reader to  \cite[Thm.1]{rosicky1997accessible} and to \cite{lieberman2016classification} for an extensive discussion of this result.

\begin{definition}
    Let $\K$ be an AETC and $\kappa$ an infinite cardinal.
    \begin{enumerate}
        \item We say that $\A\in\K$ is \emph{$\kappa$-universal} if for every $\B\in\K$ with $|\rel(\B)|<\kappa$, there is a morphism $\f\colon\B\to\A$.
        \item We say that $\A\in\K$ is \emph{$\kappa$-model-homogeneous} if for all $\B,\C\in\K$ with $|\rel(\B)|=|\rel(\C)|<\kappa$, for all morphisms $\f\colon\B\to\A$ and $\g\colon\C\to\A$ and all $\K$-isomorphisms $\h\colon\B\to\C$, there is $\pi\in\Aut(\A)$ with $\pi\circ\f = \g\circ\h$, i.e. the below diagram commutes.
        \begin{center}
            \medskip
        	\begin{tikzcd}
                {\A} && {\A} \\
                && \\
                {\B} && {\C}
                \arrow["{\h}", from=3-1, to=3-3]
                \arrow["{\f}", from=3-1, to=1-1]
                \arrow["{\g}", from=3-3, to=1-3]
                \arrow["{\pi}", from=1-1, to=1-3]
            \end{tikzcd}
            \medskip
        \end{center}
    \end{enumerate}
    Then suppose further that $\K=\GMod(T)$ for a first-order-complete theory $T$ of $\foil$.
    \begin{enumerate}[resume]
        \item We say that $\A\in\K$ is \emph{strongly $\kappa$-homogeneous} if for any partial elementary team map $\f\colon\A\to\A$ with $|\dom(\f)|<\kappa$, there is $\pi\in\Aut(\A)$ with $\pi\supseteq\f$.
    \end{enumerate}
\end{definition}

\begin{proposition}\label{kappa.universality}
	Suppose $\K$ has JEP. Then for every cardinal $\kappa$ there exists a $\kappa$-universal structure $\A\in \K$.
\end{proposition}
\begin{proof}
    Let $\B_i$, $i<\mu$, list all models of $\K$ of power $<\kappa$ up to isomorphism. We define a chain $(\A_i, \f_{i,j})_{i<j\leq\mu}$ of models $\A_i\in\K$ and morphisms $\f_{i,j}\colon\A_i\to\A_{j}$ recursively as follows.
	\begin{enumerate}
		\item For $i=0$, we let $\A_i=\B_0$ and $\f_{0,0} = \id$.
		
		\item If $i=j+1$, then we let $\A_i$ be a structure into which $\A_j$ and $\B_j$ joint embed. We let $\f_{j,i}\colon\A_j\to\A_i$ and $\g_j\colon\B_j\to\A_i$ be the respective joint embeddings. For $l<j$, we let $\f_{l,i} = \f_{j,i}\circ\f_{l,j}$.
		
		\item If $i$ is limit, then we let $\A_i = \lim_{j<i}\A_j$ and let $\f_{j,i}\colon\A_j\to\A_i$, for $j<i$ be the direct limit morphisms.
	\end{enumerate}	
	Now consider any model $\B$ of size $\leq\kappa$. Then $\B = \B_i$ for some $i<\mu$. By construction, $\f_{i+1,\mu}\circ\g_i\colon\B_i\to\A_\mu$ is a morphism, which shows that $\A \coloneqq \A_\mu$ is $\kappa$-universal.
\end{proof}

\begin{lemma}\label{Partial elementary map extension lemma}\quad
    \begin{enumerate}
        \item Let $\K$ be an AETC with AP and $\kappa$ an infinite cardinal. Then for any $\A\in\K$, there is $\B\in\K$ and a morphism $\iota\colon\A\to\B$ such that for any $\C,\D\in\K$ of power $<\kappa$, morphisms $\f\colon\C\to\A$ and $\g\colon\D\to\A$, and a $\K$-isomorphism $\h\colon\C\to\D$, there is a morphism $\k\colon\A\to\B$ such that $\iota\circ\g\circ\h = \k\circ\f$, i.e. the below diagram commutes.
        \begin{center}
            \medskip
        	\begin{tikzcd}
                & {\B} & \\
                && \\
                {\A} && {\A} \\
                && \\
                {\C} && {\D}
                \arrow["{\h}", from=5-1, to=5-3]
                \arrow["{\f}", from=5-1, to=3-1]
                \arrow["{\g}", from=5-3, to=3-3]
                \arrow["{\k}", from=3-1, to=1-2]
                \arrow["{\iota}", from=3-3, to=1-2]
            \end{tikzcd}
            \medskip
        \end{center}

        \item Suppose that $\K=\GMod(T)$ for first-order-complete $T$. Then for any $\A\in\K$, there is $\B\in\K$ and a morphism $\iota\colon\A\to\B$ such that for any partial elementary team map $\f\colon\A\to\A$ there is a morphism $\g\colon\A\to\B$ with $\iota(\f(X)) = \g(X)$ for all $X\in\dom(\f)$.
    \end{enumerate}
\end{lemma}
\begin{proof}\quad
    \renewcommand{\i}{\mathcal{i}}
    \begin{enumerate}
        \item Let $\C_i$, $i<\mu_0$, list all $\C\in\K$ of power $<\kappa$ such that there is a morphism $\C\to\A$, up to $\K$-isomorphism. For each $i<\mu_0$, let $\f_{i,j}$, $j<\mu_1$, list all morphisms $\C_i\to\A$. Then let $\mu = \mu_0\cdot\mu_1$ and fix a bijection $\theta\colon\mu\to\mu_0\times\mu_1\times\mu_1$. Denote by $\theta_i$ for $i\leq 2$ the functions such that  $\theta(\alpha) = (\theta_0(\alpha),\theta_1(\alpha),\theta_2(\alpha))$ for all $\alpha<\mu$. 
        
        We define by recursion a directed system $(\B_i,\iota_{i,j})_{i\leq j\leq\mu}$ of objects and morphisms of $\K$. We begin by letting $\B_0 = \A$. At a limit $i$, we let $\B_i = \lim_{j<i}\B_j$ and let $\iota_{j,i}$ be the direct limit morphisms. When $i = j+1$, we amalgamate the morphisms $\iota_{0,j}\circ\f_{\theta_0(j),\theta_1(j)}\colon\C_{\theta_0(j)}\to\B_j$ and $\f_{\theta_0(j),\theta_2(j)}\colon\C_{\theta_0(j)}\to\A$ to obtain $\F\in\K$ and morphisms $\k_0\colon\B_j\to\F$ and $\k_1\colon\A\to\F$ with
        \[
            \k_0\circ\iota_{0,j}\circ\f_{\theta_0(j),\theta_1(j)} = \k_1\circ\f_{\theta_0(j),\theta_2(j)}.
        \]
        We let $\B_i = \F$ and $\iota_{j,i} = \k_0$ and define the other $\iota_{k,i}$ in the obvious way via composition.

        Finally, we let $\B = \B_\mu$ and $\iota = \iota_{0,\mu}$. Now, let $\C,\D\in\K$ have power $<\kappa$, $\h\colon\C\to\D$ be a $\K$-isomorphism, and $\f\colon\C\to\A$ and $\g\colon\D\to\A$ be morphisms. Then there is $i<\mu_0$ and an isomorphism $\pi\colon\C\to\C_i$. Since $\g\circ\h\circ\pi^{-1}$ and $\f\circ\pi^{-1}$ are morphisms $\C_i\to\A$, there are $j,k<\mu_1$ such that $\f_{i,j} = \g\circ\h\circ\pi^{-1}$ and $\f_{i,k} = \f\circ\pi^{-1}$. Let $\alpha<\mu$ be such that $\theta(\alpha) = (i,j,k)$. By construction, there is $\k_1\colon\A\to\B_{\alpha+1}$ such that
        \[
            \iota_{\alpha,\alpha+1}\circ\iota_{0,\alpha}\circ\f_{\theta_0(\alpha),\theta_1(\alpha)} = \k_1\circ\f_{\theta_0(\alpha),\theta_2(\alpha)}.
        \]
        Then denoting $\k = \iota_{\alpha+1,\mu}\circ\k_1$, we have
        \begin{align*}
            \iota\circ\g\circ\h &= \iota_{0,\mu}\circ\f_{i,j}\circ\pi = \iota_{\alpha+1,\mu}\circ\iota_{0,\alpha+1}\circ\f_{i,j}\circ\pi \\
            &= \iota_{\alpha+1,\mu}\circ(\iota_{\alpha,\alpha+1}\circ\iota_{0,\alpha}\circ\f_{\theta_0(\alpha),\theta_1(\alpha)})\circ\pi \\
            &= \iota_{\alpha+1,\mu}\circ(\k_1\circ\f_{\theta_0(\alpha),\theta_2(\alpha)})\circ\pi \\
            &= (\iota_{\alpha+1,\mu}\circ\k_1)\circ(\f_{i,k}\circ\pi) \\
            &= \k\circ\f,
        \end{align*}
        as desired.

        \item Similar. \qedhere
    \end{enumerate}
\end{proof}

\begin{theorem}\quad
    \begin{enumerate}
        \item Let $\K$ be an AETC with JEP and AP. Then for every cardinal $\kappa>\LS(\K)$ there is a $\kappa$-universal, $\kappa$-model-homogeneous $\M\in\K$.
        \item Suppose that $\K=\GMod(T)$ for first-order-complete $T$. Then there is a $\kappa$-universal, strongly $\kappa$-homogeneous $\M\in\K$.
    \end{enumerate}
\end{theorem}
\begin{proof}
    \renewcommand{\i}{\mathcal{i}}
    We prove the theorem in the case that $\K$ is an AETC with JEP and AP. The other case is similar. We define a directed system $(\M_i,\i_{i,j})_{i\leq j\leq\kappa^+}$ as follows.
    \begin{enumerate}
        \item $\M_0$ is a $\kappa$-universal structure whose existence is guaranteed by \cref{kappa.universality}.
        \item We let $\M_{i+1}$ and $\i_{i,i+1}$ be the structure $\B$ and the morphism $\iota$ given by \cref{Partial elementary map extension lemma} for $\A = \M_i$. All the other maps $\i_{k,i+1}$ are defined via composition.
        \item For $i$ limit, we let $\M_i = \lim_{j<i}\M_j$ and $\i_{j,i}$ be the direct limit embeddings.
    \end{enumerate}
    We show that $\M = \M_{\kappa^+}$ is as wanted.

    For $\kappa$-universality, let $\A\in\K$ have cardinality $<\kappa$. As $\M_0$ is $\kappa$-universal, we can find a morphism $\h\colon\A\to\M_0$. Now, $\i_{0,\kappa^+}$ is a morphism $\M_0\to\M$. Hence $\iota \coloneqq \i_{0,\kappa^+}\circ\h$ is a morphism $\A\to\M$.

    For $\kappa$-model-homogeneity, let $\A,\B\in\K$ be such that $|\A|,|\B|<\kappa$ and $|\rel(\A)|=|\rel(\B)|<\kappa$, let $\f\colon\A\to\M$ and $\g\colon\B\to\M$ be morphisms, and let $\h\colon\A\to\B$ be a $\K$-isomorphism. Now by the definition of the direct limit, $\rel(\M) = \bigcup_{i<\kappa^+}\ran(\i_{i,\kappa^+})$. Since $|\rel(\A)|=|\rel(\B)|<\kappa$ and $\kappa^+$ is regular, there is $\alpha<\kappa^+$ such that $\f[\rel(\A)]\cup\g[\rel(\B)]\subseteq\ran(\i_{\alpha,\kappa^+})$. For all $i\geq\alpha$, let $\f_i = \i_{i,\kappa^+}^{-1}\circ\f$ and $\g_i = \i_{i,\kappa^+}^{-1}\circ\g$. By coherence, $\f_i$ is a morphism $\A\to\M_i$ and $\g_i$ a morphism $\B\to\M_i$. Note that $\f_{\kappa^+} = \f$ and $\g_{\kappa^+} = \g$. We now define, by recursion on $i\geq\alpha$, partial team isomorphisms $\h_i\colon\M_i\to\M_i$ such that $\ran(\f_i)\subseteq\dom(\h_i)$ and $\h_i\circ\f_i = \g_i\circ\h$ for all $i$.
    \begin{enumerate}
        \item Suppose $i=\alpha$. We let $\h_i = \g_i\circ\h\circ\f_i^{-1}$. Then
        \[
            \h_i\circ\f_i = (\g_i\circ\h\circ\f_i^{-1})\circ\f_i = \g_i\circ\h,
        \]
        as desired.

        \item Suppose $i = j+1$ and $i$ is odd. By the induction hypothesis, $\h_j\circ\f_j = \g_j\circ\h$, whence $\g_j^{-1}\circ\h_j\circ\f_j$ is a $\K$-isomorphism $\A\to\B$. By the choice of $\M_i$ and $\i_{j,i}$, there is a morphism $\k\colon\M_j\to\M_i$ such that
        \[
            \i_{j,i}\circ\g_j\circ(\g_j^{-1}\circ\h_j\circ\f_j) = \k\circ\f_j,
        \]
        i.e. $\i_{j,i}\circ\h_j\circ\f_j = \k\circ\f_j$.
        We let $\h_i = \k\circ\i_{j,i}^{-1}$. Now notice that
        \begin{align*}
            \ran(\f_i) &= \ran(\i_{i,\kappa^+}^{-1}\circ\f) = \ran(\i_{j,i}\circ\i_{j,\kappa^+}^{-1}\circ\f) \\
            &\subseteq \ran(\i_{j,i}) = \dom(\i_{j,i}^{-1}) = \dom(\k\circ\i_{j,i}^{-1}) \\
            &= \dom(\h_i).
        \end{align*}
        Then
        \begin{align*}
            \h_i\circ\f_i &= (\k\circ\i_{j,i}^{-1})\circ(\i_{i,\kappa^+}^{-1}\circ\f) = \k\circ(\i_{j,\kappa^+}^{-1}\circ\f) = \k\circ\f_j \\
            &= \i_{j,i}\circ\h_j\circ\f_j = \i_{j,i}\circ\g_j\circ\h \\
            &= \i_{j,i}\circ(\i_{j,\kappa^+}^{-1}\circ\g)\circ\h = (\i_{i,\kappa^+}^{-1}\circ\g)\circ\h \\
            &= \g_i\circ\h,
        \end{align*}
        as desired.

        \item If $i = j+1$ is even, then similarly to above, we find a morphism $\k\colon\M_j\to\M_i$ such that $\k\circ\g_j = \i_{j,i}\circ\h_j^{-1}\circ\g_j$ and let $\h_i = \i_{j,i}\circ\k^{-1}$.

        \item If $i$ is a limit, we let $\h_i$ be the partial team isomorphism $\M_i\to\M_i$ with $\dom(\h_i) = \bigcup_{\alpha\leq j<i}\{X\in\rel(\M_i) \mid \i_{j,i}^{-1}(X)\in\dom(\h_j) \}$ such that
        \[
            \h_i(X) = \i_{j,i}(\h_j(\i_{j,i}^{-1}(X)))
        \]
        for the least $j$ such that $\i_{j,i}^{-1}(X)\in\dom(\h_j)$. It is straightforward to check that $\h_i$ is a partial isomorphism. Now let $Y\in\ran(\f_i)$. Then $Y = \f_i(X) = \i_{\alpha,i}(\f_\alpha(X))$ for some $X\in\rel(\A)$. By the induction hypothesis $\ran(\f_\alpha)\subseteq\dom(\h_\alpha)$, so $\i_{\alpha,i}^{-1}(Y) = \f_\alpha(X)\in\dom(\h_\alpha)$. By the definition of $\dom(\h_i)$, we have $Y\in\dom(\h_i)$.

        Let $X\in\rel(\A)$. By the above, $\alpha$ is the least $j<i$ such that $\i_{j,i}^{-1}(\f_i(X))\in\dom(\h_j)$. Hence,
        \begin{align*}
            \h_i(\f_i(X)) &= \i_{\alpha,i}(\h_\alpha(\i_{\alpha,i}^{-1}(\f_i(X)))) = \i_{\alpha,i}(\h_\alpha(\f_\alpha(X))) \\
            &= \i_{\alpha,i}(\g_\alpha(\h(X))) = \g_i(\h(X)).
        \end{align*}
        Thus $\h_i\circ\f_i = \g_i\circ\h$.
    \end{enumerate}

    Finally, we let $\pi = \h_{\kappa^+}$. We show that $\dom(\pi) = \ran(\pi) = \rel(\M)$.
    By definition,
    \[
        \dom(\pi) = \bigcup_{\alpha\leq i<\kappa^+} \{ X\in\rel(\M) \mid i_{i,\kappa^+}^{-1}(X)\in\dom(\h_i) \}.
    \]
    Let $X\in\rel(\M)$. Let $\gamma$ be the least even $i\in[\alpha,\kappa^+)$ such that $X \in\ran(\i_{i,\kappa^+})$. Let $\k$ be the morphism $\M_\gamma\to\M_{\gamma+1}$ from the construction of $\h_{\gamma+1}$. As $\dom(\k) = \rel(\M_\gamma)$, we have $\i_{\gamma,\kappa^+}^{-1}(X)\in\dom(\k)$. Then
    \[
        \k(\i_{\gamma,\kappa^+}^{-1}(X)) = \k(\i_{\gamma,\gamma+1}^{-1}(\i_{\gamma+1,\kappa^+}^{-1}(X))) = \h_{\gamma+1}(\i_{\gamma+1,\kappa^+}^{-1}(X)),
    \]
    whence $\i_{\gamma+1,\kappa^+}^{-1}(X)\in\dom(\h_{\gamma+1})$. But this means that $X\in\dom(\pi)$. Using a symmetric argument, and odd $\gamma$, one can show that $\pi$ is surjective. Moreover, by construction,
    \[
        \pi\circ\f = \h_{\kappa^+}\circ\f_{\kappa^+} = \g_{\kappa^+}\circ\h = \g\circ\h.
    \]
    Finally, since $\M$ is a general model and contains all singleton relations, we have that $\pi$ is element-total and element-surjective. Therefore, by  \cref{extension.of.isomorphisms} there is a $\tau$-automorphism $\pi'$ of $\M$ such that $\hat\pi'\supseteq\pi$. Hence $\pi$ ($=\hat\pi'\restriction\rel(\M)$) is an automorphism of the category $\K$ by closure under isomorphisms.
\end{proof}

The above $\kappa$-monster model $\M$ is $\kappa$-saturated in the following sense.
\begin{proposition}
    Let $\K$ be an AETC with JEP and AP, and let $\M$ be a  $\kappa$-universal and $\kappa$-model-homogeneous model of $\K$, for $\kappa>\LS(\K)$. Suppose that $\g\colon\M\to\B$ is a morphism. Then for any $\alpha,\beta<\kappa$, $\vec Y\in\rel(\B)^\alpha$ and $\vec D\in\rel(\M)^\beta$, there is $\vec X\in\rel(\M)^\alpha$ such that
    \[
        \gtype(\vec X\vec D / \emptyset; \M) = \gtype(\vec Y\g(\vec D) / \emptyset; \B).
    \]
\end{proposition}
\begin{proof}\quad
    \renewcommand{\i}{\mathcal i}
    Let $\mu = |\alpha|+|\beta| < \kappa$, and denote $\vec B = \g(\vec D)$. By the Löwenheim--Skolem property, there is $\A\in\K$ with $|\A|\leq\mu+\LS(\K)<\kappa$ and $|\rel(\A)|\leq\mu+\aleph_0<\kappa$ and a morphism $\f\colon\A\to\M$ such that $\vec D\in\ran(\f)^\beta$. Now, let $\vec A = \f^{-1}(\vec D)$.
        
    Again, by the Löwenheim--Skolem property, there is $\C\in\K$ with $|\C|\leq\mu+\LS(\K)<\kappa$ and $|\rel(\C)|\leq\mu+\aleph_0 < \kappa$, and a morphism $\h\colon\C\to\B$, such that
    \[
        \cl(\ran(\g\circ\f)\cup\{Y_i \mid i<\alpha\}) \subseteq \ran(\h).
    \]
    Let $\vec Z = \h^{-1}(\vec Y)$ and $\vec C = \h^{-1}(\vec B)$, and denote $\k = \h^{-1}\circ\g\circ\f$. Now $\g\circ\f = \h\circ\k$, so by coherence $\k$ is a morphism $\A\to\C$. By $\kappa$-universality of $\M$, there is a morphism $\i\colon\C\to\M$. By $\kappa$-model-homogeneity, there is $\pi\in\Aut(\M)$ such that $\pi\circ\f =(\i\circ\k)\circ \id_{\A} = \i\circ\k$. Now, let $\vec X = \pi^{-1}(\i(\vec Z))$. Notice that $\k(\vec A) = \vec C$, so
    \begin{align*}
        \vec D &= \pi^{-1}(\pi(\f(\vec A))) = \pi^{-1}(\i(\k(\vec A))) = \pi^{-1}(\i(\vec C)).
    \end{align*}
    Now, as $\pi^{-1}\circ\i$ is a morphism $\C\to\M$ and $(\pi^{-1}\circ\i)(\vec Z\vec C) = \vec X\vec D$, it follows that $\gtype(\vec X\vec D / \emptyset; \M) = \gtype(\vec Z\vec C / \emptyset; \C)$. On the other hand, as $\h$ is a morphism $\C\to\B$ and $\h(\vec Z\vec C) = \vec Y\vec B$, we have $\gtype(\vec Z\vec C / \emptyset; \C) = \gtype(\vec Y\vec B / \emptyset; \B)$. By transitivity of Galois types, we then have $\gtype(\vec X\vec D / \emptyset; \M) = \gtype(\vec Y\vec B / \emptyset; \B)$, as desired.
\end{proof}

Like classically, the monster model makes it easier to deal with Galois types: instead of amalgamating into a third model, having the same type can be now witnessed by an automorphism of the monster.

\begin{proposition}
    Let $\alpha<\kappa$ and $\vec X,\vec Y\in\rel(\M)^\alpha$. Then $\gtype(\vec X/\emptyset; \M) = \gtype(\vec Y/\emptyset; \M)$ if and only if there is $\pi\in\Aut(\M)$ such that $\pi(\vec X) = \vec Y$.
\end{proposition}
\begin{proof}
    \renewcommand{\i}{\mathcal{i}}
    If $\vec X$ and $\vec Y$ are conjugates by $\pi\in\Aut(\M)$, then the morphisms $\id_\M$ and $\pi$ witness that $\vec X$ and $\vec Y$ have the same type.
    
    Suppose that $\vec X$ and $\vec Y$ have the same type. Then there is $\A\in\K$ and morphisms $\f_0\colon\M\to\A$ and $\f_1\colon\M\to\A$ such that $\f_0(\vec X) = \f_1(\vec Y)$. By the Löwenheim--Skolem property, there is $\B\in\K$ with $|\B|,|\rel(\B)|<\kappa$ and a morphism $\g\colon\B\to\M$ such that $\vec X,\vec Y\in\ran(\g)^\alpha$. There is also $\C\in\K$ with $|\C|,|\rel(\C)|<\kappa$ and a morphism $\h\colon\C\to\A$ such that
    \[
        \cl(\ran(\f_0\circ\g)\cup\ran(\f_1\circ\g))\subseteq\ran(\h).
    \]
    Now by coherence, the maps $\k_i\coloneqq \h^{-1}\circ\f_i\circ\g$, $i<2$, are morphisms $\B\to\C$. By $\kappa$-universality of $\M$, there is a morphism $\i\colon\C\to\M$. Now $\i\circ\k_0$, $\g$ and $\i\circ\k_1$ are morphisms $\B\to\M$, so by $\kappa$-model-homogeneity, there are $\pi_0,\pi_1\in\Aut(\M)$ such that $\pi_0\circ\g = (\i\circ\k_0)\circ\id_\B = \i\circ\k_0$   and $\pi_1\circ(\i\circ\k_1) = \g\circ\id_\B = \g$. Let $\pi = \pi_1\circ\pi_0$. Then $\pi\in\Aut(\M)$ and
    \begin{align*}
        \pi(\vec X) &= \pi_1(\pi_0(\vec X)) = \pi_1(\pi_0(\g(\g^{-1}(\vec X)))) = \pi_1(\i(\k_0(\g^{-1}(\vec X)))) \\
        &= \pi_1(\i(\h^{-1}(\f_0(\g(\g^{-1}(\vec X)))))) = \pi_1(\i(\h^{-1}(\f_0(\vec X)))) \\
        &= \pi_1(\i(\h^{-1}(\f_1(\vec Y)))) = \pi_1(\i(\h^{-1}(\f_1(\g(\g^{-1}(\vec Y)))))) \\
        &= \pi_1(\i(\k_1(\g^{-1}(\vec Y)))) = \g(\g^{-1}(\vec Y)) \\
        &= \vec Y. \qedhere
    \end{align*}
\end{proof}

In particular, the previous lemma allows one to extend the notion of Galois type and take parameters into account. For a small set $\Z\subseteq\rel(\M)$ and given two sequences $\vec X,\vec Y\in\rel(\M)^\alpha$, we can define 
\[
    \gtype(\vec X/\Z; \M) = \gtype(\vec Y/\Z; \M)
\]
if there is $\pi\in\Aut(\M/\Z)$ such that $\pi(\vec X) = \vec Y$. One can then verify that $\gtype(\vec X/\Z; \M) = \gtype(\vec Y/\Z; \M)$ if and only if $\gtype(\vec X\vec Z/\emptyset; \M) = \gtype(\vec Y\vec Z/\emptyset; \M)$, where $\vec Z$ enumerates $\Z$. 

\begin{remark}
    To recap, in this section we introduced a notion of Galois type of a tuple of relations in an abstract elementary team category and built a monster model which is sufficiently saturated with respect to these types. However, we shall now see why this notion of type is not, in fact, a good basis for the definition of stability. Instead, it would seem that in the concrete example $\GMod(T)$ for a complete $\eso$-theory $T$ (cf.~\cref{Section_Categoricity}), the stability properties of the first-order reduct of $T$ are a more useful measure of complexity.
    
    More precisely, to see that Galois types in team semantics are ill-behaved from the stability-theoretic point of view, we consider the formula $\phi\coloneqq x\subseteq y$ in the context of an arbitrary $\eso$-theory $T$ with infinite models. We show that this formula has both a version of the \emph{indipendence property} (IP) and the \emph{strict order property} (SOP) in our context (see \cite[p.134]{tentziegler} for the standard definition of IP and SOP in first-order logic).

    First, consider any model $\A$ of $T$ and let $A=\{a_i\mid i<\omega \}$ be an enumeration of some countable subset thereof. Then for every  $I\subseteq \omega$ consider the set of formulas
    \begin{align*}
        \Gamma_I\coloneqq\{ y_i\subseteq x \mid i\in I     \} \cup \{\cneg (y_i\subseteq x) \mid i\notin I  \}.
    \end{align*}    
    Then consider a team $X_I$ with domain  $\dom(X_I) = \{x\} \cup\{y_i\mid i<\omega \}$ and such that $X_I[x]=\{a_i\mid i\in I \}$ and $X_I[y_i]=\{a_i\}$ for all $i<\omega$. Then we obtain $\A\models_{X_I} \Gamma_I$, showing that any theory $T$ with infinite models satisfies this version of IP.

    Similarly, we can also show that any theory $T$ with infinite model has the strict order property. Let again $\A$ be an arbitrary infinite model and $A=\{a_i\mid i<\omega \}$ a countable subset of it.  For every $i<\omega$, we let $A_i=\{a_j\mid i\leq j\}$. Consider the formula $\phi\coloneqq x\subseteq y$ and let $X_i$ be the team with domain $\{x_i\}$ such that $X_i[x_i]=A_i$. Then it is easy to verify that, for all $i<j<\omega$
    \[
        \A\models_{X_i\times X_j} \forall  y \; (y\subseteq x_i \wcimp y\subseteq x_j) \iff i\leq j.
    \]
    Which shows that any theory $T$ with infinite models has this version of SOP for team semantics.

    The two observations above hint at the fact that a robust classification theory for (complete) theories in $\eso$ (or $\foil$) cannot rely exclusively on the notion of Galois types introduced in this section. In particular, we obtain as a consequence that for any set of size $\kappa$ of parameters the number of resulting Galois types is $2^\kappa$, showing that the cardinality of the space of types does not allow to track any information about the underlying theory. 
        
    For these reasons, we take later  in \cref{Section_Categoricity} a different approach. In particular, given a complete $\eso$-theory $T$, we shall consider the (standard) stability properties of its first-order reduct $T^*$. We shall apply this approach especially to the problem of transferring categoricity among different cardinals, and prove two results in this direction. We take these preliminary findings as a hint that the standard notions of stability and independence also have a key role in the present second-order context and may be the correct ones for studying the model theory of complete theories in $\eso$ or $\foil$.
\end{remark}

\section{Application -- Lindström's Theorem for $\fot$}\label{Lindstrom}

A celebrated result by \textcite{lindstrom1969extensions} characterises first-order logic as a maximal abstract logic satisfying both the compactness and the downwards Löwenheim-Skolem theorems. Moreover, he also proved in \cite{lindstrom1974characterizing} that first-order logic is maximal among the abstract logics which are both compact and have the so-called Tarski union property. A similar result by \textcite{sgro1977maximal} shows that first-order logic is also a maximal abstract logic to satisfy the Łoś' theorem.

Motivated by these results, we provide an example of application of the abstract machinery of the previous sections and prove two maximality results for $\fot$ in terms of its model-theoretic properties. We believe our results could also be obtained from the classical Lindström's theorem via \cref{Translation between FOT and FO} with a clever translation, but here we merely demonstrate the use of our newly acquired tools. We start by defining a suitable notion of abstract logic in the setting of team semantics, which we call \emph{abstract team logic}.  Given a signature $\tau$, we denote by $\Str(\tau)$ the class of all $\tau$-structures.

In this section, we make a distinction between teams and their underlying relations, to allow teams with infinite domains. However, results about team maps are only applied in the case where the argument of the map is a finitary relation.

\begin{definition}
    An \emph{abstract team logic} is a pair $L=(\Fml,\Sat)$, where
    \begin{itemize}
        \item $\Fml$ is a class together with a function $\Fv$ that maps each element $\phi\in\Fml$ to a (possibly infinite) set $\Fv(\phi)$ of variables, and
        \item $\Sat$ is a class of triples $(\A,X,\phi)$, where $\phi\in\Fml$, $\A$ is a structure (of some signature) and $X$ is a team of $\A$ with $\dom(X)\supseteq\Fv(\phi)$,
    \end{itemize}
    satisfying the following.
    \begin{enumerate}
        \item $L$ is \emph{closed under isomorphisms}, i.e. if $\pi\colon\A\to\B$ is an isomorphism between $\tau$-structures $\A$ and $\B$, then for all $\phi\in\Fml$, $D\supseteq\Fv(\phi)$ and $X\subseteq\A^D$,
        \[
            (\A,X,\phi)\in\Sat \iff (\B,\pi^*(X),\phi)\in\Sat,
        \]
        where $\pi^*(X)=\{ \pi\circ s \mid s\in X \}$.

        \item $L$ satisfies the \emph{occurence condition}, i.e. for any $\phi\in\Fml$ there is a signature $\sigma$ such that for any $\tau$ and $\A\in\Str(\tau)$, $\A\models_X\phi$ implies $\tau\supseteq\sigma$, and for any $\tau\supseteq\sigma$, $D\supseteq\Fv(\phi)$, $\A\in\Str(\tau)$ and $X\subseteq\A^D$,
        \[
            (\A,X,\phi)\in\Sat \iff (\A\restriction\sigma,X,\phi)\in\Sat.
        \]

        \item $L$ is \emph{closed under renaming}, i.e. if $\tau$ and $\tau'$ are signatures and $\pi\colon\tau\to\tau'$ is a bijection that preserves the type and the arity of symbols, then for any $\phi\in\Fml$ there is $\psi\in\Fml$ with $\Fv(\psi)=\Fv(\phi)$ such that for all $D\supseteq\Fv(\phi)$, $\A\in\Str(\tau)$ and $X\subseteq\A^D$, we have
        \[
            (\A,X,\phi)\in\Sat \iff (\pi(\A),X,\psi)\in\Sat,
        \]
        where $\pi(\A)$ is the $\tau'$-structure with $\dom(\pi(\A)) = \dom(\A)$ and $\pi(R)^{\pi(\A)} = R^\A$ for any symbol $R\in\tau$.
    \end{enumerate}
\end{definition}

If $\phi\in\Fml$, we simply write $\phi\in L$. We say that $\phi$ is a $\sigma$-formula if $\sigma$ is as in the occurrence condition, and we denote $\sigma$ by $\Sig(\phi)$. If $(\A,X,\phi)\in\Sat$, we write $\A\models_X\phi$. If $\Sigma\subseteq\Fml$, we write $\A\models_X\Sigma$ if $\A\models_X\phi$ for all $\phi\in\Sigma$. We write $\Mod_L(\phi)$ for the class of all $(\A,X)$ such that $\A\models_X\phi$ and $X\neq\emptyset$. If $\A$ is a $\tau$-structure, we also denote by $\Th_L(\A,X)$ the set of all $\tau$-formulas $\phi\in L$ such that $\A\models_X\phi$. We write $\Th_L(\A)$ for $\Th_L(\A,\{\emptyset\})$. Note that if $\phi\in\Th_L(\A,X)$, then $\Fv(\phi)\subseteq\dom(X)$. We write $(\A,X)\equiv_L(\B,Y)$ if $\Th_L(\A,X) = \Th_L(\B,Y)$. We say that a set $\Sigma\subseteq\Fml$ is \emph{consistent} if $\bigcap_{\phi\in\Sigma}\Mod_L(\phi)\neq\emptyset$. We say that $\phi$ is consistent if $\{\phi\}$ is, and $\phi$ is consistent with $\Sigma$ if $\Sigma\cup\{\phi\}$ is consistent. We define the following properties of abstract team logics.

\begin{definition}
    We say that an abstract team logic $L$ is \emph{regular} if the following hold.
    \begin{enumerate}
        \item $L$ is \emph{closed under conjunction}, i.e., for all $\phi,\psi\in L$ there is $\chi\in L$ such that $\Fv(\chi)=\Fv(\phi)\cup\Fv(\psi)$ and for all $D\supseteq\Fv(\chi)$, and for all $\A$ and $X\subseteq\A^D$,
        \[
            \A\models_X\chi \iff \text{$\A\models_X\phi$ and $\A\models_X\psi$}.
        \]
        Furthermore, if $\phi$ and $\psi$ are $\sigma$-formulas, also $\chi$ can be chosen to be a $\sigma$-formula. We denote this $\chi$ by $\phi\land\psi$. 

        \item $L$ is \emph{closed under disjunction}, i.e., for all $\phi,\psi\in L$ there is $\chi\in L$ such that $\Fv(\chi)=\Fv(\phi)\cup\Fv(\psi)$ and for all $D\supseteq\Fv(\chi)$, and for all $\A$ and $X\subseteq\A^D$,
        \[
            \A\models_X\chi \iff \text{$\A\models_X\phi$ or $\A\models_X\psi$}.
        \]
        Furthermore, if $\phi$ and $\psi$ are $\sigma$-formulas, also $\chi$ can be chosen to be a $\sigma$-formula. We denote this $\chi$ by $\phi\ivee\psi$. 

        \item\label{closure under classical negation} $L$ is \emph{closed under weak classical negation}, i.e., for all $\phi\in L$ there is $\psi\in L$ such that $\Fv(\phi)=\Fv(\psi)$ and for all $D\supseteq\Fv(\psi)$, $\A$ and \emph{nonempty} $X\subseteq\A^D$,
        \[
            \A\models_X\psi \iff \A\nmodels_X\phi.
        \]
        Furthermore, if $\phi$ is a $\sigma$-formula, also $\psi$ can be chosen to be a $\sigma$-formula. We denote this $\psi$ by $\cneg\phi$.
    \end{enumerate}

    
    We say that $L$ is \emph{team-finitary} if for any $\phi\in L$ there is a finite set $D\supseteq\Fv(\phi)$ such that for every $(\A,X)\in\Mod_L(\phi)$, there is a team $Y\subseteq\A^D$ such that
    $\A\models_Y\phi$.
    In particular, $\Fv(\phi)$ is finite.
\end{definition}

\begin{definition}
    Let $L$ be an abstract team logic.
    \begin{enumerate}
        \item We say that $L$ is \emph{compact} if for any $\Sigma\subseteq\Fml$, $\Sigma$ is consistent if and only if every finite subset of $\Sigma$ is.

        \item We say that a team embedding $\f\colon\A\to\B$ preserves $\phi\in L$ if for all $D\supseteq\Fv(\phi)$ and $X\subseteq\A^D$,
        \[
            \A\models_{X}\phi \iff \B\models_{\f(X)}\phi.
        \]
        We say that $L$ has \emph{direct limits} if the following holds. Let $(\A_i,\f_{i,j})_{i,j\in I,i\leq j}$ be a directed system of $\tau$-structures, let $\B = \lim_{i\in I}\A_i$ and denote for all $i\in I$ by $\g_i$ the direct limit embedding $\A_i\to\B$. Then for all $\phi\in L$, if $\f_{i,j}$ preserves $\phi$ for all $i,j\in I$, $i\leq j$, then $\g_i$ preserves $\phi$ for all $i\in I$. This is the team logic counterpart of the Tarski union property.

        \item We say that $L$ has a \emph{Łoś' theorem} if for any set $I$, $\tau$-structures $\A_i$, teams $X_i\subseteq \A_i^D$ and ultrafilter $\U$ on $I$, for any formula $\phi\in L$:
        \[
    		\{i\in I \mid \A_i\models_{X_i}\phi\}\in\U \implies \prod_{i\in I} \A_i/\U \models_{\prod_{i\in I} X_i/\U } \phi.
		\]
        We say that $L$ has a \emph{strong Łoś' theorem} if the above holds in both directions:
		\[
    		\{i\in I \mid \A_i\models_{X_i}\phi\}\in\U \iff \prod_{i\in I} \A_i/ \U \models_{\prod_{i\in I} X_i/ \U } \phi.
		\]        
    \end{enumerate}
\end{definition}

We define a partial order among abstract team logics as follows.
\begin{definition}
    Let $L,L'$ be two abstract team logics. We write $L\leq L'$ if for any $\tau$-formula $\phi\in L$ there is a $\tau$-formula $\phi'\in L'$ such that $\phi\equiv\phi'$, i.e., for all $\A\in\Str(\tau)$ and all teams $X$ of $\A$,
    \[
        \A\models_X\phi \iff \A\models_X\phi'.
    \]
    We write $L\equiv L'$ if $L\leq L'$ and $L'\leq L$.
\end{definition}

If $\phi,\psi\in L$ and $\Mod_L(\phi) \subseteq \Mod_L(\psi)$, then we write $\phi\models\psi$. We write $\phi\models\psi$ even if $\phi\in L$, $\psi\in L'$ and $\Mod_L(\phi) \subseteq \Mod_{L'}(\psi)$. We write $\phi\equiv\psi$ if $\phi\models\psi$ and $\psi\models\phi$. Note that if $\phi$ and $\psi$ are consistent and $\phi\equiv\psi$, we must have $\Sig(\phi) = \Sig(\psi)$. If $\Sigma\subseteq\Fml$ and $\phi\in L'$, we write $\Sigma\models\phi$ if $\bigcap_{\psi\in\Sigma}\Mod_L(\psi)\subseteq\Mod_{L'}(\phi)$.

\begin{lemma}
    Let $L$ be a regular abstract team logic.
    \begin{enumerate}
        \item Let $\A$ be a $\tau$-structure and $X\subseteq\A^D$ nonempty. Then $\Th_L(\A,X)$ is complete with respect to $\tau$-formulas of $L$ whose free variables are contained in $D$, i.e. if $\phi\in L$ is a $\tau$-formula and $\Fv(\phi)\subseteq D$, then either $\phi\in\Th_L(\A,X)$ or $\cneg\phi\in\Th_L(\A,X)$.

        \item For any $\phi\in L$, $\phi \equiv \cneg\cneg\phi$.

        \item Let $\Sigma\cup\{\phi \}\subseteq L$. Then $\Sigma\cup\{\phi\}$ is inconsistent if and only if $\Sigma\models\cneg\phi$.
    \end{enumerate}
\end{lemma}
\begin{proof}\quad
    \begin{enumerate}
        \item Since $X$ is nonempty, for any $\phi\in L$ we have
        \[
            \A\models_X\cneg\phi \iff \A\nmodels_X\phi.
        \]
        Clearly exactly one of the options $\A\models_X\phi$ and $\A\models_X\cneg\phi$ is true.

        \item Trivial.

        \item Suppose that $\Sigma\models\cneg\phi$. Then for any $\A$ and $X$ such that $\A\models_X\Sigma$, we have $\A\models_X\cneg\phi$. This means that for such $\A$ and $X$, whenever $X$ is nonempty, $\A\nmodels_X\phi$, which means that there are no $\A$ and nonempty $X$ such that $\A\models_X\Sigma\cup\{\phi\}$. Hence $\Sigma\cup\{\phi\}$ is inconsistent.

        On the other hand, if $\Sigma\cup\{\phi\}$ is inconsistent, then $\bigcap_{\psi\in\Sigma}\Mod_L(\psi)\cap\Mod_L(\phi) = \bigcap_{\psi\in\Sigma\cup\{\phi\}}\Mod_L(\psi) = \emptyset$. Hence for any $\A$ and nonempty $X$, if $\A\models_X\Sigma$, we have $\A\nmodels_X\phi$. But this means that $\A\models_X\cneg\phi$. Hence $\Sigma\models\cneg\phi$. \qedhere
    \end{enumerate}
\end{proof}

We can now prove the team logic counterparts to the results of Lindström and Sgro. The fist part of the following theorem characterises $\fot$ in the context of regular abstract team logics and adapts \cite{lindstrom1974characterizing} in our context, using direct limits (i.e. the appropriate variant of the Tarski union property). The second part applies to all (possibly irregular) abstract team logics and adapts \cite{sgro1977maximal}, characterizing $\fot$ in terms of Łoś' theorem.

\begin{theorem}\quad
    \begin{enumerate}
        \item The logic $\fot$ is maximal among all regular abstract team logics that are team-finitary, compact and closed under direct limits.
        \item The logic $\fot$ is maximal among all 
        abstract team logics that are team-finitary and have a strong Łoś' theorem.
    \end{enumerate}
\end{theorem}
\begin{proof}\quad
    \renewcommand{\i}{\mathcal i}
    \renewcommand{\j}{\mathcal j}
    \begin{enumerate}
        \item Suppose this is not the case. Then there is a regular abstract team logic $L$ such that $\fot\leq L\nleq \fot$ and $L$ is team-finitary, compact and closed under direct limits. Let $\phi\in L$ be such that $\phi\not\equiv\psi$ for all $\psi\in\fot$. By the occurrence condition, there is $\tau$ such that $\phi$ is a $\tau$-formula, and by team-finitarity there is a finite $D\supseteq\Fv(\phi)$ such that whenever $\phi$ is satisfied in a team, it is satisfied in a team with domain $D$. Let $\Sigma$ be the set of all $\tau$-formulas $\psi\in\fot$ such that $\phi\models\psi$ and $\Fv(\psi)\subseteq D$. We show that $\Sigma\cup\{\cneg\phi\}$ is consistent. Let $\Sigma'\subseteq\Sigma$ be finite. Now if $\Sigma'\cup\{\cneg\phi\}$ were inconsistent, then we would have $\Sigma'\models\cneg\cneg\phi$, whence $\bigwedge\Sigma'\models\phi$. As $\Sigma$ is clearly closed under conjunction, we would also have $\bigwedge\Sigma'\in\Sigma$, whence $\phi\models\bigwedge\Sigma'$. But then $\phi\equiv\bigwedge\Sigma'$, contradicting the choice of $\phi$. Hence we conclude that $\Sigma'\cup\{\cneg\phi\}$ is consistent. Thus by compactness of $L$, $\Sigma\cup\{\cneg\phi\}$ is consistent. Thus there is a $\tau$-structure $\A$ and nonempty $X\subseteq\A^D$ with $\A\models_X\Sigma\cup\{\cneg\phi\}$.
        
        Let $\X = \{\{a\} \mid a\in\A\}\cup\{X[\vx]\}$, where $\vx$ lists $D$. Now we consider the set $\Diag(\A,\X)\cup\{\phi,\theta(\underline{X[\vx]}, \vx)\}$, where $\theta$ is as in \cref{FOT-formula that characterises being a predicate}, and show that it is consistent. Let $\Sigma'\subseteq\Diag(\A,\X)$ be finite. Let $a_0,\dots,a_{n-1}$ enumerate all $a\in\A$ such that $\underline{\{a\}}$ occurs in a formula of $\Sigma'$. Now there is a first-order $\tau$-formula $\psi(y_0,\dots,y_{n-1},R)$, where $R$ is a second-order variable with arity $|D|$, such that for any $\tau$-structure $\B$ and its $\tau\cup\{\underline{\{a_0\}},\dots,\underline{\{a_{n-1}\}},\underline X\}$-expansion $\hat\B$, we have
        \[
            \hat\B\models\bigwedge\Sigma'\iff \B\models\psi(b_0,\dots,b_{n-1},\underline{X[\vx]}^{\hat\B}),
        \]
        where $b_i$ is the unique element inhabiting $\underline{\{a_i\}}^{\hat\B}$.
        Let $\psi^*(\vx)$ be an $\fot$-translation of $\exists y_0\dots\exists y_{n-1}\psi(y_0,\dots,y_{n-1},R)$, i.e. such a $\tau$-formula of $\fot$ that for any $\tau$-structure $\B$ and $Y\in\B^D$,
        \[
            \B\models\exists y_0\dots\exists y_{n-1}\psi(y_0,\dots,y_{n-1},Y[\vx]) \iff \B\models_Y\psi^*.
        \]
        Now suppose for a contradiction that $\Sigma'\cup\{\phi,\theta(\underline{X[\vx]},\vx)\}$ is inconsistent. Then so is $\{\psi^*,\phi\}$. This means that $\phi\models\cneg\psi^*$, whence $\cneg\psi^*\in\Sigma$. As $\A\models_X\Sigma$, we have $\A\models_X\cneg\psi^*$. Since $X\neq\emptyset$, we obtain $\A\nmodels_X\psi^*$. Thus $\A\nmodels\exists y_0\dots\exists y_{n-1}\psi(y_0,\dots,y_{n-1},X[\vx])$, which is a contradiction since $\A\models\psi(a_0,\dots,a_{n-1},X[\vx])$. Therefore $\Sigma'\cup\{\phi,\theta(\underline{X[\vx]},\vx)\}$ is consistent and, by compactness of $L$, so is $\Diag(\A,\X)\cup\{\phi,\theta(\underline{X[\vx]},\vx)\}$. Thus there is a $\tau(\X)$-structure $\hat\B$ such that $\hat\B\models\Diag(\A,\X)\cup\{\phi,\theta(\underline{X[\vx]},\vx)\}$.
        
        Let $\B = \hat\B\restriction\tau$. Then by \cref{diagram.team.embeddings}, there is partial elementary team map $\f\colon\A\to\B$ with $\dom(\f) = \cl(\X)$. Now notice that $\A$, $\B$, $X$ and $\f$ are such that
        \[
            \A\models_X\cneg\phi \quad\text{and}\quad \B\models_{\f(X)}\phi.
        \]

        Now we define, by recursion on $n<\omega$, $\tau$-structures $\A_n$ and $\B_n$, sets $\X_n\subseteq\R(\A_n)$ and $\Y_n\subseteq\R(\B_n)$, teams $X_n\in\X_n$ and $Y_n\in\Y_n$, and partial elementary team maps $\mathcal i_n\colon\A_n\to\A_{n+1}$, $\mathcal j_n\colon\B_n\to\B_{n+1}$, $\f_n\colon\A_n\to\B_n$ and $\g_n\colon\B_n\to\A_{n+1}$, satisfying the following.
        \begin{enumerate}
            \item $\X_n = \{X_n\}\cup\{\{a\}\mid a\in\A_n\}$ and $\Y_n = \{Y_n\}\cup\{\{b\} \mid b\in\B_n\}$.
            \item $\A_n\models_{X_n}\cneg\phi$ and $\B_n\models_{Y_n}\phi$.
            \item $\dom(\i_n) = \dom(\f_n) = \cl(\X_n)$ and $\dom(\j_n) = \dom(\g_n) = \cl(\Y_n)$.
            \item $\ran(\i_n)\cup\ran(\g_n)\subseteq\cl(\X_{n+1})$ and $\ran(\j_n)\cup\ran(\f_n)\subseteq\cl(\Y_{n+1})$.
            \item $Y_n = \f_n(X_n)$, $Y_{n+1} = \j_n(Y_n)$ and $X_{n+1} = \i_n(X_n) = \g_n(Y_n)$.
            \item The resulting diagram commutes.
        \end{enumerate}
        \begin{center}
            \bigskip
            \begin{tikzcd}
                {\A_0} && {\A_1} && {\A_2} && {\A_3} && {} \\
                &&&&&& \\
                && {\B_0} && {\B_1} && {\B_2} && {}
                \arrow["{\i_0}", from=1-1, to=1-3]
                \arrow["{\i_1}", from=1-3, to=1-5]
                \arrow["{\i_2}", from=1-5, to=1-7]
                \arrow["{\j_0}"', from=3-3, to=3-5]
                \arrow["{\j_1}"', from=3-5, to=3-7]
                \arrow["{\f_0}", from=1-1, to=3-3]
                \arrow["{\f_1}", from=1-3, to=3-5]
                \arrow["{\f_2}", from=1-5, to=3-7]
                \arrow["{\g_0}"', from=3-3, to=1-3]
                \arrow["{\g_1}"', from=3-5, to=1-5]
                \arrow["{\g_2}"', from=3-7, to=1-7]
                \arrow[dotted, from=3-7, to=3-9]
                \arrow[dotted, from=1-7, to=1-9]
            \end{tikzcd}
            \bigskip
        \end{center}
        
        \noindent We start by letting $\A_0\coloneqq\A$, $X_0\coloneqq X$, $\B_0\coloneqq\B$ and $Y_0\coloneqq\f(X)$. Then we let $\X_0 = \X = \{X_0\}\cup\{\{a\} \mid a\in\A_0\}$, $\Y_0 = \{Y_0\}\cup\{\{b\} \mid b\in\B_0 \})$, and $\f_0 = \f$. These are all clearly as desired.
        
        Suppose we have constructed $\A_i$, $\B_i$, $\X_i$, $\Y_i$, $X_i$, $Y_i$ and $\f_i$ for $i\leq n$ and $\i_i$, $\j_i$ and $\g_i$ for $i<n$. We now construct $\A_{n+1}$, $\B_{n+1}$, $\f_{n+1}$, $\i_n$, $\j_n$ and $\g_n$. Let $\Gamma$ be the set
        \[
            \hspace{3em} \{\cneg \phi,\theta(\underline{X_n},\vx)\} \cup \Diag(\A_n,\X_n) \cup \Diag(\B_n,\Y_n) \cup \{ \text{``$\underline{\f_n(R)}=\underline R$''} \mid R\in\X_n \},
        \]
        where ``$S = S'$'' is a shorthand for $\forallone\vec w (S(\vec w)\wcequiv S'(\vec w))$.
        We use compactness of $L$ to show that $\Gamma$ is consistent. By the induction hypothesis, $\A_n\models_{X_n}\cneg \phi$, and clearly $\A_n\models_{X_n}\{\theta(\underline{X_n},\vec x)\}\cup\Diag(\A_n,\X_n)$. Now, look at a finite $\Gamma'\subseteq\Diag(\B_n,\Y_n)$. Then $\bigwedge\Gamma'$ is equivalent to a first-order $\tau(\Y_n)$-sentence of the form $\psi(\underline{\f_n(\{a_0\})},\dots,\underline{\f_n(\{a_{k-1}\})},\underline{\{b_0\}},\dots,\underline{\{b_{m-1}\}},\underline{Y_n})$ for some $a_0,\dots,a_{k-1}\in\A_n$ and $b_0,\dots,b_{m-1}\in\B_n$ such that $\{b_i\}\notin\ran(\f_n)$ for all $i<m$. Now
        \[
            \hspace{2em}\B_n\models\exists y_0\dots\exists y_{k-1}\psi(\f_n(a_0),\dots,\f_n(a_{k-1}),y_0,\dots,y_{m-1},\f_n(X_n)),
        \]
        so as $\f_n$ is elementary, we have
        \[
            \A_n\models\exists y_0\dots\exists y_{k-1}\psi(a_0,\dots,a_{k-1},y_0,\dots,y_{m-1},X_n),
        \]
        so we can find $a_k,\dots,a_{m+k-1}\in\A_n$ such that $\bigwedge\Gamma'$ is satisfied in an expansion of $\A_n$ that interprets $\underline{\{b_i\}}$ as $\{a_{k+i}\}$, $\underline{\f(\{a_i\})}$ as $\{a_i\}$ and $\underline{Y_n}$ as $X_n$. Such an expansion clearly also satisfies the sentences ``$\underline{\f_n(R)}=\underline R$''. Then by compactness, $\Gamma$ has a model, so we let $\hat\A_{n+1}$ be one such. As $\hat\A_{n+1}\models\Diag(\A_n,\X_n)$, we can find a partial elementary team map $\i_n\colon\A_n\to\hat\A_{n+1}$ with $\dom(\i_n) = \cl(\X_n)$ and $\underline R^{\hat\A_{n+1}} = \i_n(R)$ for all $R\in\X_n$. As $\hat\A_{n+1}\models\Diag(\B_n,\Y_n)$, we can find a partial elementary team map $\g_n\colon\B_n\to\hat\A_{n+1}$ with $\dom(\g_n)=\cl(\Y_n)$ and $\underline R^{\hat\A_{n+1}} = \g_n(R)$ for all $R\in\Y_n$. As $\hat\A_{n+1}\models\text{``$\underline{\f_n(R)} = \underline R$''}$ for all $R\in\X_n$, we have
        \[
            \g_n(\f_n(R)) = \underline{\f_n(R)}^{\hat\A_{n+1}} = \underline R^{\hat\A_{n+1}} = \i_n(R)
        \]
        for all $R\in\X_n$, and hence $\i_n = \g_n\circ\f_n$. Then we can let $\A_{n+1} = \hat\A_{n+1}\restriction\tau$ and $\X_{n+1} = \{X_n\}\cup\{\{a\} \mid a\in\A_{n+1}\}$ as desired.

        The model $\B_{n+1}$, the set $\Y_{n+1}$ and the maps $\j_n$ and $\f_{n+1}$ are constructed in a similar manner, using $\B_n$, $\Y_n$, $\A_{n+1}$, $\X_{n+1}$ and $\g_n$.

        In the end, we let $\A_\omega = \lim_{i<\omega}\A_i$, $X_\omega = \i_\omega(X_0)$, where $\i_\omega$ is the direct limit map $\A_0\to\A_\omega$ of the system $(\A_i)_{i<\omega}$, and $\B_\omega = \lim_{i<\omega}\B_i$ and $Y_\omega = \j_\omega(Y_0)$, where $\j_\omega$ is the direct limit map $\B_0\to\B_\omega$ of the system $(\B_i)_{i<\omega}$. Now, since $L$ has direct limits, we have $\A_\omega\models_{X_\omega}\cneg\phi$ and $\B_\omega\models_{Y_\omega}\phi$. But as both $(\A_i)_{i<\omega}$ and $(\B_i)_{i<\omega}$ are cofinal in the directed system, we must have $\A_\omega\cong\B_\omega$ and furthermore, the isomorphism maps $X_\omega$ to $Y_\omega$, which is impossible.

        \item Suppose this is not the case. Then there is an abstract positive team-finitary team logic $L$ which has a strong Łoś' theorem with the following property: there exists a formula $\phi\in L$ such that for all $\psi\in \fot$ either
        \begin{enumerate}
            \item there is a model $\A_\psi$ and a nonempty team $X_\psi$ such that $\A_\psi\nmodels_{X_\psi} \phi$ and $ \A_\psi\models_{X_\psi} \psi$, or
            \item there is a model $\B_\psi$ and a nonempty team $Y_\psi$ such that $\B_\psi\models_{Y_\psi} \phi$ and  $\B_\psi\nmodels_{Y_\psi} \psi$.
        \end{enumerate}
        We also remark that since $L$ has a strong Łoś' theorem then it is also compact, as the usual ultraproduct proof works in this setting (see \cite[Theorem 3.14]{puljujärvi2022compactness}).

        Let $D$ be a set of variables satisfying the team-finitarity condition for $\phi$ and $\tau$ a signature satisfying the occurrence condition for $\phi$. Consider the set $\Gamma$ of all  $\tau$-formulas $\psi\in\fot$ such that $\phi\models\psi$ and $\Fv(\psi)\subseteq D$. Also, we let $\S$ be the family of finite subsets of $\Gamma$. Since $\phi\models \Gamma$, it follows by the choice of $\phi$ that for every $\Sigma\in \S$ there are a model $\A_\Sigma$ and a nonempty team $X_\Sigma$ on $\A_\Sigma$ such that $\dom(X_\Sigma)=D$,  $\A_\Sigma\models_{X_\Sigma} \Sigma$ and $\A_\Sigma\nmodels_{X_\Sigma} \phi$. For every $\psi\in \Gamma$, we let  $[\psi]=\{ \Sigma\in \S \mid \psi \in \Sigma \}$ and $F=\{[\psi] \mid \psi \in \Gamma  \}$.   Then, given  $\psi_0,\dots,\psi_{n-1}\in \Gamma$ we have that $\{ \psi_i \mid i<n\}\in \bigcap_{i<n}[\psi_i]$, showing that $F$ has the finite intersection property. It follows that we can extend $F$ to an ultrafilter $\U$ on $\S$. 
    
        Consider now the ultraproduct $\prod_{\Sigma\in\S}\A_\Sigma/\U$. For every formula $\psi\in  \Gamma$ we have that $[\psi]\in \U$ and for every $\Sigma'\supseteq \{\psi \}$ we have  that $\A_{\Sigma'}\models_{X_{\Sigma'}} \psi$. It follows by the Los' Theorem of $\fot$ that $\prod_{\Sigma\in\S}\A_\Sigma/\U\models_{\prod_{\Sigma\in\S}X_\Sigma/\U} \psi$, and so that $\prod_{\Sigma\in\S}\A_\Sigma/\U\models_{\prod_{\Sigma\in\S}X_\Sigma/\U} \Gamma$. Moreover, for any $\Sigma\in\S$ we have that $\A_\Sigma\nmodels_{X_{\Sigma}}  \phi$, thus by the strong Łoś' theorem  of $L$  we obtain that $\prod_{\Sigma\in\S}\A_\Sigma/\U\nmodels_{\prod_{\Sigma\in\S}X_\Sigma/\U} \phi $.

        Now consider the theory $T= \Th_\fot(\prod_{\Sigma\in\S}\A_\Sigma/\U,\prod_{\Sigma\in\S}X_\Sigma/\U)$. We claim this is consistent with $\phi$. If not, by compactness of $L$, there is a finite set of formulas  $\Delta\subseteq T$ such that $\Delta\cup \{\phi\}$ is not consistent, whence $\phi\models\cneg\bigwedge\Delta$. It follows that $\cneg\bigwedge\Delta \in \Gamma$ and so
        \[
            \prod_{\Sigma\in\S}\A_\Sigma/\U\models_{\prod_{\Sigma\in\S}X_\Sigma/\U} \cneg\bigwedge\Delta,
        \]
        meaning that $\cneg\bigwedge\Delta\in T$. As $T$ is closed under conjunction, we also have $\bigwedge\Delta\in T$, whence $(\bigwedge\Delta)\land(\cneg\bigwedge\Delta)\in T$. But this is impossible, as $\prod_{\Sigma\in\S}X_\Sigma/\U$ is nonempty.
        
        We conclude that there are a $\tau$-structure $\B$ and a team $Y$ on $\B$ such that $\Th_\fot(\prod_{\Sigma\in\S}\A_\Sigma/\U,\prod_{\Sigma\in\S}X_\Sigma/\U)=    \Th_\fot(\B,Y)$ and $\B\models_{Y}\phi$. Notice that, by the choice of $D$ and the fact that $D=\dom(\prod_{\Sigma\in\S}X_\Sigma/\U)$, we have that both these teams are finitary. Then, by the translation from $\fot$ to $\fol$ it follows that
        \[
            (\B,Y) \equiv_\fol \left(\prod_{\Sigma\in\S}\A_\Sigma/\U, \prod_{\Sigma\in\S}X_\Sigma/\U\right)
        \]
        and so, by Keisler--Shelah, there are two elementary maps $\pi_0\colon\B\to\C$ and $\pi_1\colon\prod_{\Sigma\in\S}\A_\Sigma/\U\to \C $ into some common ultrapower $\C$ (up to isomorphism). In particular, notice that since we are expanding the signature with symbols for $Y$ and $\prod_{\Sigma\in\S}X_\Sigma/\U$, it follows that $\pi_0(Y)=\pi_1(\prod_{\Sigma\in\S}X_\Sigma/\U)$. By assumption, $L$ has a strong Łoś' theorem, and so does $\fot$. Thus it follows that $\C\models_{\pi_0(Y)}\phi$ and $\C\nmodels_{\pi_1(\prod_{\Sigma\in\S}X_\Sigma/\U)}\phi$, which contradicts $\pi_0(Y)=\pi_1(\prod_{\Sigma\in\S}X_\Sigma/\U)$.  \qedhere
    \end{enumerate}
\end{proof}

\section{Addendum -- Complete Theories and Categoricity in $\eso$}\label{section:complete_theories}

As emphasized e.g.\ in \cite{baldwin_categoricity}, the problem of categoricity is one of the most classical topics in model theory and clasification theory. In this last section, we begin to consider the problem of categoricity in the setting of existential second-order logic and of independence logic. To the best of our knowledge, this is the first study where categoricity is investigated in this context. This is in stark contrast to the categoricity problem for full second-order logic, where the issue of categoricity has been the subject of several studies, as it also bears very interesting connections to set theory (we refer the reader to \cite{sep-logic-higher-order} for an overview, and to \cite{SAARINEN_2025} for some recent results).

We proceed in this section as follows. In \cref{Section: Complete Theories}, we study some basic properties of complete $\eso$-theories and relate them to so-called resplendent models. Then, in \cref{Section_Categoricity}, we use stability theory to prove some versions of categoricity transfer for complete theories in $\eso$. We prove in, \cref{section.downwards.categoricity}, a downwards categoricity transfer result and, in \cref{section.upwards.categoricity}, an upwards categoricity transfer one. We try to provide details for most of the proofs. However, given the more technical aspects of the following sections, and most specifically of \cref{section.upwards.categoricity}, we will have to use technical notions and definitions from the literature without providing much explanations. We refer the reader to classical source books in stability and classification theory for the necessary background, most specifically to \cite{baldwin_categoricity,tentziegler,poizat2012course,baldwin2017fundamentals,shelah1990classification}. We emphasize that the results in the final section provide, in our eyes, some interesting connections between $\eso$ and classical model theory which have not been, so far, considered in the literature.

\subsection{Complete Theories \& Resplendent Models}\label{Section: Complete Theories}

In this section, we study complete theories in $\eso$ and their models. As we are always dealing with sentences, in all results one may replace first-order logic by $\fot$ and $\eso$ by $\foil$. The main result in this section is \cref{complete.resplendent}, which relates our setting of elementary team embeddings to the classical notion of \emph{resplendency} from first-order model theory. We will employ resplendent models also later in the proof of \cref{ESO categoricity transfer up to uncountable}.

Recall (\cref{definition: complete eso theory}) that a theory $T$ of $\eso$ is complete if for all $\A,\B\models T$ and sentences $\phi$ of $\eso$,
\[
    \A\models\phi \iff \B\models\phi.
\]
Note that a theory which is purely first order is, of course, complete in this sense if and only if it is complete in the usual sense. A theory of $\eso$ is first-order-complete if all its models are elementarily equivalent (\cref{definition: complete foil-theory and first-order-completeness}).

\begin{proposition}\label{prop:unique_eso_completion}
    Let $T$ be a complete first-order $\tau$-theory. Then there is a unique $\eso$-theory $\bar T$ such that
    \begin{enumerate}
        \item $\bar T\supseteq T$,
        \item $\bar T$ is closed under logical consequence and
        \item $\bar T$ is complete.
    \end{enumerate}
\end{proposition}
\begin{proof}
    Let $\kappa = |\tau|+\aleph_0$, and let $\phi_i$, $i<\kappa$, enumerate all $\tau$-sentences of $\eso$. We define $\bar T$ by recursion. Let $T_0 = T$ and $T_i = \bigcup_{j<i}T_j$ for $i$ limit. For $i = j+1$, we let $T_i = T_j\cup\{\phi_j\}$ if it is consistent and $T_i = T_j$ otherwise. Then we let $\bar T = \bigcup_{i<\kappa}T_i$. By compactness of $\eso$, $\bar T$ is consistent, and as such it is clearly a maximal consistent extension of $T$ and hence closed under logical consequence. Now, let $\A$ and $\B$ be models of $\bar T$ and $\phi$ a sentence of $\eso$. Now if $\phi\in\bar T$, we have $\A\models\phi$ and $\B\models\phi$, as both are models of $\bar T$. On the other hand, if $\phi\notin\bar T$, then $\bar T\cup\{\phi\}$ is inconsistent, whence $\A\nmodels\phi$ and $\B\nmodels\phi$. Hence $\bar T$ is complete.
    
    For uniqueness, let $T'$ be another complete theory in $\eso$ extending $T$. If $T'\cup\bar T$ is consistent, then $T'\subseteq\bar T$, as $\bar T$ is a maximal consistent theory, whence the closure of $T'$ under logical consequence is exactly $\bar T$. Thus it suffices to show that $T'\cup\bar T$ is consistent. Let $\A\models T'$ and $\B\models\bar T$. As $T$ is first-order-complete, it has the joint embedding property (by \cref{class.properties}). Hence there is $\C$ and elementary team embeddings $\f\colon\A\to\C$ and $\g\colon\B\to\C$. Since elementary team embeddings are independence team embeddings, $\A\models\phi$ implies $\C\models\phi$ and $\B\models\phi$ implies $\C\models\phi$ for all $\phi\in\eso$. Hence $\C\models T'\cup\bar T$.
\end{proof}

\begin{definition}
    We say that a $\tau$-structure $\A$ is \emph{complete} if for any other $\tau$-structure $\B$ such that there is an elementary team embedding $\A\to\B$, we have
    \[
        \A\models\phi \iff \B\models\phi
    \]
    for all sentences $\phi\in\eso$. 
\end{definition}

\begin{proposition}\label{complete-theory-models}
    For a first-order-complete theory $T$ of $\eso$ and a structure $\A$, the following are equivalent.
    \begin{enumerate}
        \item $\A\models\bar T$, where $\bar T$ is the unique completion of $T$ given by \cref{prop:unique_eso_completion}.
        \item $\A$ is a complete model of $T$.
    \end{enumerate}
\end{proposition}
\begin{proof}
    Suppose first that $\A\models\bar T$. Let $\f\colon\A\to\B$ be an elementary team embedding, and recall by \cref{elementary.team.embeddings.are.independence} that $\f$ is also an independence team embedding. Then, in particular, for all sentences $\phi\in\eso$ we have that $\A\models\phi$ entails $\B\models\phi$. It follows that $\B\models\bar T$. But then, as $\bar T$ is complete, we have
    \[
        \A\models\phi \iff \B\models\phi
    \]
    for all sentences $\phi\in\eso$. Hence $\A$ is complete.

    Then, suppose that $\A$ is a complete model of $T$. Let $\B\models\ESOTh(\A)$. By joint embedding, there is $\C$ and elementary team embeddings $\f\colon\A\to\C$ and $\g\colon\B\to\C$. Let $\phi$ be a sentence of $\eso$. Now, if $\A\models\phi$, then $\phi\in\ESOTh(\A)$, whence $\B\models\phi$. On the other hand, if $\A\nmodels\phi$, since $\f$ is an elementary team embedding, by completeness of $\A$ we have $\C\nmodels\phi$. If it were the case that $\B\models\phi$, since $\g$ is an independence team embedding, we would have $\C\models\phi$, a contradiction. Hence $\B\nmodels\phi$. Therefore $\ESOTh(\A)$ is complete, whence $\bar T = \ESOTh(\A)$.
\end{proof}

We conclude this section by highlighting a connection between complete structures and so-called resplendent models. The following can be found in \cite[Ch.~9.3]{poizat2012course}.

\begin{definition}
    A $\tau$-structure $\A$ is \emph{resplendent} if
    for any elementary extension $\B$ of $\A$ and a $\tau$-formula $\phi(x_0,\dots,x_{n-1})$ of $\eso$, we have
    \[
        \B\models\phi(a_0,\dots,a_{n-1}) \iff \A\models\phi(a_0,\dots,a_{n-1})
    \]
    for all $a_0,\dots,a_{n-1}\in\A$.
\end{definition}

\noindent For the proof of the following fact, we refer the reader again to \cite[Ch.~9.3]{poizat2012course}.
 
\begin{fact}
    Let $T$ be a first-order theory. Then every model of $T$ with cardinality $\geq|T|$ has a resplendent elementary extension of the same cardinality.
\end{fact}

\begin{proposition}\label{complete.resplendent}
    The following are equivalent for a structure $\A$.
    \begin{enumerate}
        \item $\A_\A$ is complete.
        \item $\A$ is resplendent.
    \end{enumerate}
\end{proposition}
\begin{proof}
    Suppose first that $\A$ is resplendent. To show that $\A_\A$ is complete, let $\B$ be a $\tau(\A)$-structure and $\f\colon\A_\A\to\B$ an elementary team embedding, and let $\phi$ be a $\tau(\A)$-sentence of $\eso$ such that $\B\models\phi$. By \cref{From team embedding to a substructure}, we may assume that $\B$ is an elementary extension of $\A_\A$ and $\f$ is the identity on singletons. Now $\phi = \psi(\underline a_0,\dots,\underline a_{n-1})$ for some $\tau$-formula $\psi(x_0,\dots,x_{n-1})\in\eso$ and $a_0,\dots,a_{n-1}\in\A$, so as $\B\models\phi$, we have $\B\restriction\tau\models\psi(a_0,\dots,a_{n-1})$. Now $\B\restriction\tau$ is an elementary extension of $\A$, so by the resplendence of $\A$, we obtain $\A\models\psi(a_0,\dots,a_{n-1})$, whence $\A_\A\models\phi$. Hence $\A_\A$ is complete.

    Conversely, suppose that $\A_\A$ is complete. Let $\B$ be an elementary extension of $\A$ and let $\phi(x_0,\dots,x_{n-1})\in\eso$ be such that $\B\models\phi(a_0,\dots,a_{n-1})$ for some $a_0,\dots,a_{n-1}\in\A$. Since $\B$ is an elementary extension of $\A$, we have $\A_\A\equiv\B_\A$, so by the joint embedding property there is a $\tau(\A)$-structure $\C$ and elementary team embeddings $\f\colon\A_\A\to\C$ and $\g\colon\B_\A\to\C$. Now $\B_\A\models\phi(\underline a_0,\dots,\underline a_{n-1})$, so as $\g$ is an independence team embedding, we have $\C\models\phi(\underline a_0,\dots,\underline a_{n-1})$. As $\A_\A$ is complete and $\f$ is an elementary team embedding, we have $\A_\A\models\phi(\underline a_0,\dots,\underline a_{n-1})$. But then $\A\models\phi(a_0,\dots,a_{n-1})$. Thus $\A$ is resplendent.
\end{proof}

\subsection{Categoricity Transfer}\label{Section_Categoricity}

As we mentioned above, we consider in this section the categoricity problem for complete theories in $\eso$ (and $\foil$). The main results from this section are two categoricity tranfer results between countable and uncountable cardinals for complete $\eso$-theories, i.e., \cref{ESO categoricity transfer down to countable} and \cref{ESO categoricity transfer up to uncountable}.  We assume some familiarity with the standard notions from classification theory and refer to \cite{baldwin_categoricity,poizat2012course,baldwin2017fundamentals,shelah1990classification,tentziegler} for  background and motivation. We also stress that the following results provide a first study of the spectrum function for existential second order theories. In connection to this topic, we remark the following fact.
\begin{fact}
    Let $T$ be a complete theory in $\eso$ and $T^*$ its first-order reduct.  If $T^*$ is unstable, then $T$ has $2^\kappa$ models in any cardinal $\kappa>|T|$.
\end{fact}
\noindent If $T$ is a complete $\eso$-theory, the class of its models is clearly pseudoelementary, as it is the class of models of the theory $T'$ obtained by adding to the signature witnesses for every existential second-order statement in $T$. It follows that the fact above is an immediate corollary of Shelah's general result on the number of models of pseudoelementary classes \cite[Ch. VIII, Thm 2.1]{shelah1990classification}.

\subsubsection{Downwards Categoricity Transfer}\label{section.downwards.categoricity}

We show in this section that if the first-order reduct $T^*$ of an $\eso$-theory $T$ in a finite signature is uncountably categorical, then the theory $T$ is also $\aleph_0$-categorical. It follows that $T$ is quasi-categorical, i.e., it has a unique model (up to isomorphims) in all cardinals for which it has a model to begin with. 

The proof of this fact is essentially a corollary of Baldwin and Lachlan's study of uncountably categorical theories (see e.g.~\cite[Section 6.3]{tentziegler}). In particular, recall that if $T$ is an uncountably categorical first-order theory, then
\begin{enumerate}
    \item $T$ is totally transcendental and thus has a (countable) prime model $\A_0$,
    \item there are a formula $\phi(x,\vy)$ and $\vec a\in\A_0$ such that $\phi(x,\vec b)$ is strongly minimal for every $\vec b\models\foltype(\vec a/\emptyset)$, and
    \item since $\A_0$ is atomic, the type $\foltype(\vec a/\emptyset)$ is isolated by some $\psi(\vy)$.
\end{enumerate}
Then $\phi(x,\vec a)$ determines a pregeometry whose dimension does not depend on the choice of $\vec a$, as long as it realizes $\psi$. For $\A\models T$, we shall write $\dim_{\phi,\psi}(\A)$ for the dimension of the pregeometry $(\phi(\A,a),\acl)$. We denote the dimension of $\A_0$ by $m_0$. The following well-known theorem due to Baldwin and Lachlan \cite[Theorem 6.3.7]{tentziegler} classifies all countable models of uncountably categorical theories.

\begin{theorem}[Baldwin--Lachlan]
    Let $T$ be an uncountably categorical first-order theory in a countable signature. Then for each cardinal $\kappa\geq m_0$, there is a unique model $\A_\kappa$ with $\dim_{\phi,\psi}(\A_\kappa)=\kappa$.
\end{theorem}

In particular, since $m_0\leq\aleph_0$, this means that either $T$ is countably categorical, or it has $\aleph_0$-many countable models up to isomorphism. We shall see now how its $\eso$-completion is satisfied only by the countable model of dimension $\aleph_0$.

\begin{theorem}\label{ESO categoricity transfer down to countable}
    If $T$ is a complete theory in $\eso$ in a finite signature and its first-order reduct $T^*$ is uncountably categorical, then $T$ is categorical in all infinite cardinalities.
\end{theorem}
\begin{proof}
    First, notice that since the signature $\tau$ of $T$ is finite, $\eso$ can talk about automorphisms of a $\tau$-structure. In particular, for any second-order function variable $f$, there is a $\tau$-formula $\aut(f)$ such that
    \[
        \A\models\aut(\pi) \iff \text{$\pi$ is a $\tau$-automorphism of $\A$}.
    \]
    Also, $\eso$ can express when a set is infinite, namely there is a formula $\inf(X)$ in $\eso$ such that:
    \[
        \A\models \inf(A) \iff  |A|\geq \aleph_0.
    \]
    
    Let $\A_0$ be a prime model of $T^*$, and let $\phi(x,\vec{a})$ be strongly minimal, with $\vec{a}\in\A_0$ being a tuple $\vec{a}=(a_0,\dots,a_{m-1})$ for some $m<\omega$. Let $\theta(\vec{v})$ isolate $\foltype(\vec{a}/\emptyset)$. Now, for all $n<\omega$, let $\Phi_n$ be the following sentence of $\sol$:
    \begin{gather*}
        \forall v\big(\theta(\vec{v})\to\forall x_0\dots\forall x_{n-1}\exists X\big(\inf(X)\land\forall y(X(y)\to\phi(y,\vec{v}))\land\psi\big) \big),
    \end{gather*}
    where $\psi(\vec{v},x_0,\dots,x_{n-1},X)$ is the formula
    \[
        \forall y\forall z\bigg(\big(X(y)\land X(z)\big)\to\exists f\bigg(\aut(f) \land f(y) = z \land f(\vec{v}) = \vec{v} \land \bigwedge_{i<n}f(x_i) = x_i\bigg)\bigg).
    \]
    In other words, $\Phi_n$ expresses that whenever $\vec{c}$ is such that $\varphi(x,\vec{c})$ is strongly minimal, then for any $d_0,\dots,d_{n-1}$ there are infinitely many elements of the strongly minimal set such that any two of them are conjugates over $\vec{c},d_0,\dots,d_{n-1}$. One can easily check that $\Phi_n\in\eso$.

    First of all, we show that $T\models\Phi_n$ for all $n<\omega$. Let $\B$ be an uncountable elementary extension of $\A_0$. As $T$ has a model in every infinite cardinality (because of Löwenheim--Skolem) and $T^*$ is uncountably categorical, we must have $\B\models T$. Also, recall that since $T$ is uncountably categorical it follows that $\B$ is saturated. It is enough to show that $\B\models\Phi_n$ for all $n<\omega$. For this, fix $n$, let $\vec{b}\in\B$ be a tuple $\vec{b}=(b_0,\dots,b_{m-1})$ such that $\B\models\theta(\vec{b})$ and let $e_0,\dots,e_{n-1}\in\B$ be arbitrary. Now $\varphi(x,\vec{b})$ is strongly minimal, so it has a unique non-forking extension $p\in S(\{b_0,\dots,b_{m-1},e_0,\dots,e_{n-1}\})$. In particular, $p$ is not algebraic. Hence, as $\B$ is $\aleph_0$-saturated, $p$ is realized in $\B$, and as $p$ is not algebraic, there are actually infinitely many $c\in\B$ with $c\models p$. Let $C$ be the set of them. Now $C\subseteq\phi(\B,\bar{b})$ and any two elements of $C$ are conjugates over $\{b_0,\dots,b_{m-1},e_0,\dots,e_{n-1}\}$. Hence we have that $\B\models\Phi_n$.

    In order to show that $T$ is $\aleph_0$-categorical, it is enough to show that any countable model $\A$ of $T$ has dimension $\aleph_0$. Let $\A$ be a countable model of $T$ and let $\phi(\A,\vec{a})$ be strongly minimal with $\vec{a}=(a_0,\dots,a_{m-1})$. We construct by recursion an independent sequence $e_i$, $i<\omega$, in the pregeometry $(\phi(\A,\vec{a}),\acl)$. Suppose that we have already found independent $e_0,\dots,e_{n-1}$. As $T\models\Phi_n$, we have $\A\models\Phi_n$. Hence there is an infinite set $D\subseteq\phi(\A,\vec{a})$ such that any two elements of $D$ are conjugates over $\{a_0,\dots,a_{m-1},e_0,\dots, e_{n-1}\}$. It follows that all elements of $D$ have the same type over $\{a_0,\dots,a_{m-1},c_0,\dots, c_{n-1}\}$ and, since $D$ is infinite, it must be that $D\subseteq\phi(\A,\vec{a})\setminus\acl(\{a_0,\dots,a_{m-1},e_0,\dots, e_{n-1}\})$. Finally, it follows that any element of $D$ is algebraically independent of $a_0,\dots,a_{m-1},e_0,\dots, e_{n-1}$, whence it can be chosen as $e_n$.
\end{proof}

\begin{example}
    The motivating example for the result above is given by the theories of algebraically closed fields. In particular, let $T_{\mathrm{ACF}_0}$ be the theory of algebraically closed fields of characteristic 0 and let $\bar T_{\mathrm{ACF}_0}$ be its (unique) $\eso$-completion. Then the previous theorem shows that $\bar T_{\mathrm{ACF}_0}$ is $\aleph_0$-categorical and its only countable model is the algebraic closure of the field $\mathbb Q(t_0,t_1,\dots)$ of polynomial fractions over the rational numbers.
\end{example}

\subsubsection{Upwards Categoricity Transfer}\label{section.upwards.categoricity}

We conclude the paper by proving, in this section, an upwards categoricity transfer result for complete theories in existential second-order logic. In particular, we show that under the assumptions of $\omega$-stability and 1-basedness, $\aleph_0$-categoricity of the first-order reduct $T^*$ entails uncountable categoricity of the theory $T$. 

We follow in this section the standard notation and terminology from classification theory. In particular, we use the symbol $\ind$ to refer to forking independence and we identify the strong type $\stp(\vec a/C)$ of a tuple $\vec a$ over parameters $C$ with the set of all equivalence classes $E(\vx,\vec a)$, where $E$ is any finite equivalence relation definable over $C$. Recall in particular that, in this context, ``finite'' means that $E$ has only finitely many equivalence classes, not that it is finite as a set. The set of finite equivalence relations over $C$ we denote by $\FE(C)$.  As the notion of a $1$-based theory is less standard, we recall its definition and one of its key properties. The result is folklore, but we refer the reader to \cite[Fact 3.1]{de2003geometry} for a proof. The converse of the lemma is also true, but we shall not need it. In this subsection, we denote by $\M$ the ordinary elementary monster model of the first-order theory $T^*$, which is a $\kappa$-universal, $\kappa$-homogeneous and $\kappa$-saturated model for a large enough cardinal $\kappa$.

\begin{definition}
    A first-order theory $T$ is 1-based if for all $A,B\subseteq \M^{\mathrm{eq}}$, we have $A\Ind_{\acleq(A)\cap \acleq(B)} B$.
\end{definition}

\begin{lemma}\label{1-based.lemma}
    Suppose $T$ is 1-based and $C$ is finite. If $(a_i)_{i\in I}$ is $C$-indiscernible, then for all $0<i<j$ we have $a_i\ind_{Ca_0}(a_j)_{i\neq j\in I}$.
\end{lemma}

Many steps in the proof of the following theorem are well-known results among stability theorists, but we include them for completeness.

\begin{theorem}\label{ESO categoricity transfer up to uncountable}
    Let $T$ be a complete $\eso$-theory and let $T^*$ be its first-order reduct. If $T^*$ is  $\omega$-stable, $\aleph_0$-categorical and  1-based, then $T$ is uncountably categorical.
\end{theorem}
\begin{proof}
    Let $\lambda$ be an uncountable cardinal. We show that all models of $T$ of cardinality $\lambda$ are saturated, from which $\lambda$-categoricity follows. Let $\A$ be such a model, and let $p$ be a type over some parameter set of cardinality $<\lambda$. The general strategy of the proof is to show that $p$ is realised in $\A$. Without loss of generality, $p\in S_1(\B)$ for some elementary submodel $\B\preceq\A$ of power $<\lambda$. We denote $\kappa = |\B|$.

    Let $\M\succeq \A$ be the (first-order) monster model of $T^*$. Since $T^*$ is $\omega$-stable, we may have $\M$ be saturated. Note that a saturated structure is resplendent: as it has a resplendent elementary extension of the same cardinality and is itself saturated, it must be isomorphic to the extension. It thus follows from \cref{complete-theory-models} and \cref{complete.resplendent} that $\M\models T$.
    
    Let $a\in \M$ realize $p$. By superstability of $T^*$, there is a finite $B\subseteq\B$ such that $a\ind_{B}\B$. Since $T^*$ is $\aleph_0$-categorical, by Ryll-Nardzewski, there are only finitely many non-equivalent formulas with free variables $x,y$ and parameters from $B$. In particular, there is a finite list of formulas $\delta_i(x,y,\vec d_i)$, $i<n$, defining all finite equivalence relations over $B$. We let $E_i\in\FE(B)$ be the relation defined by $\delta_i$. Now, for each $i<n$, pick $c_i\in\B$ such that $(a,c_i)\in E_i$. The type $\foltype(a / B\cup\{c_i\mid i<n\})$ has a finite number of parameters, so by $\aleph_0$-categoricity it is isolated by some formula and thus  realised in $\B$ by some element $b$. Then we have, for all $i<n$,
    \begin{align*}
        (a,c_i)\in E_i &\iff \delta_i(x,c_i,\vec d_i)\in \foltype(a / B\cup\{c_i\mid i<n\}) \\
        &\iff \delta_i(x,c_i,\vec d_i)\in \foltype(b / B\cup\{c_i\mid i<n\}) \\
        &\iff (b,c_i)\in E_i,
    \end{align*}
    and so we obtain that $\stp(a / B) = \stp(b / B)$.

    Let $C = B\cup\{b\}$. We claim that $\foltype(a / C)$ ($= p\restriction C$) is stationary, i.e. has a unique non-forking extension over $\B$. To this end, suppose that $c,d\models\foltype(a/C)$, $c\ind_C\B$ and $d\ind_C\B$. Since $a\ind_B\B$, by monotonicity $a\ind_B C$. As $\foltype(c / C) = \foltype(a / C)$, we then obtain $c\ind_B C$. But then from $c\ind_B C$ and $c\ind_C \B$, transitivity gives $c\ind_B \B$. Similarly $d\ind_B \B$. 
    Now, for every $E\in\FE(B)$, if $\delta(x,y,\vec d)$ is the defining formula of $E$, then since $(a,b)\in E$, we have $\delta(x,b,\vec d)\in\foltype(a/C)$ and so $\delta(x,b,\vec d)\in\foltype(c/C),\foltype(d/C)$. Hence $(c,b),(d,b)\in E$. It follows that $(c,d)\in E$. Hence $\stp(c/C) = \stp(d/C)$. Thus, as $c$ and $d$ have the same strong type over $C$ and both of them are independent of $\B$ over $C$, by stationarity of strong types this means that $\foltype(c/\B) = \foltype(d/\B)$. Hence $\foltype(a/C)$ is stationary.
    In particular, we have so far found an element $a$ and a finite $C\subseteq\B$ such that $a\ind_C\B$, $\foltype(a/C)$ is stationary and $a\models p$.

    Now, since $\M$ is saturated, we can find a sequence $(b_i)_{i<|\M|}$ such that $b_i\models p\restriction C$ for all $i<|\M|$  and $b_i\ind_C b_j$ for all $i,j<|\M|$. We claim that the existence of such sequence can be expressed via a sentence of $\eso$. First we define the following formulas:
    \begin{enumerate}
        \item We write $\max(X)$ for the formula saying that there is a bijection between the background structure and $X$. This is clearly expressible in $\eso$.
        
        \item We let $\psi(x,\vec{y})$ be a formula isolating $\foltype(aC/\emptyset)$ and $\chi(\vec{y})$ be a formula isolating $\foltype(C/\emptyset)$. The existence of such formulas follows from $\aleph_0$-cate\-goricity.
        
        \item We let $\theta(x,y,\vz)$ be a first-order formula such that $\M\models\theta(d_0,d_1,\vec e)$ if and only if $d_0\ind_{\vec e}d_1$. The reason why such a formula exists is the following. First of all, if $d_0\ind_{\vec e} d_1$ and $\pi\in\Aut(\M)$, then $\pi(d_0)\ind_{\pi(\vec e)}\pi(d_1)$. Hence the property of $d_0$ being independent of $d_1$ over $\vec e$ only depends on the type $\foltype(d_0d_1\vec e / \emptyset)$. By Ryll-Nardzewski, there are only finitely many $|\vec e|+2$-types over $\emptyset$, so in particular there are only finitely many $q$ such that if $d_0d_1\vec e\models q$, then $d_0\ind_{\vec e}d_1$, so let $q_i$, $i<m$, list all of them. For each $i<m$, let $\theta_i(x,y,\vz)$ isolate $q_i$. Now clearly
        \[
            d_0\ind_{\vec e}d_1 \iff \M\models\bigvee_{i<m}\theta_i(d_0,d_1,\vec{e}).
        \]
        So we may choose $\theta = \bigvee_{i<m}\theta_i$.
    \end{enumerate}
    We then let $\phi$ be the sentence
    \begin{align*}
        \forall \vec{z}\ \big(&\chi(\vec{z})\to \exists X \big( \max(X) \land \forall x(X(x)\to\psi(x,\vec{z})) \land{} \\
        &\; \forall x\forall y((X(x)\land X(y)\land\neg x = y)\to\theta(x,y,\vec z))\big)\big).
    \end{align*}
    The sequence $(b_i)_{i<|\M|}$ witnesses that $\M\models\phi$, whence $T\models\phi$, and so $\A\models\phi$. Then in $\A$, we can find a sequence of length $|\A|$ of elements that are independent of each other over $C$ and realize the type $\foltype(aC/\emptyset)$. Hence they also realize the type $\foltype(a/C)$. It follows that, in $\A$, there is a sequence of distinct elements $(a_i)_{i<\kappa^+}$ such that $a_i\models p\restriction C$ for all $i<\kappa^+$  and $a_i\ind_C a_j$ for all $i,j<\kappa^+$. 

    We use some combinatorics to turn the former into a $C$-indiscernible sequence. By the locality property of non-forking and superstability of $T^*$, for any limit ordinal $i<\kappa^+$, there is $\alpha_i<i$ such that $a_i\ind_{C\cup \{a_j\mid j<\alpha_i \}} \{a_j \mid j<i\}$. Since the set of all limit ordinals below $\kappa^+$ is stationary in $\kappa^+$, we may apply Fodor's Lemma (see e.g. \cite[Thm. 8.7]{jech}) to find $\alpha<\kappa^+$ and a stationary $X\subseteq\kappa^+$ such that $a_i\ind_{C\cup \{a_j\mid j\in \alpha\cap X \}}\{a_j \mid j\in i\cap  X\}$ for all $i\in X$. Now, by locality, for every $i\in X$ there is a finite set $C_i\subseteq C\cup \{a_j\mid j\in\alpha\cap X  \}$ such that $a_i\ind_{C_i} \{a_j \mid j\in i\cap  X\}$. Since $|(\alpha\cap X)^{<\omega}|\leq\kappa$, there are at most $\kappa$-many such finite sets $C_i$, so by the pidgenhole principle there is $Y\subseteq X$ of power $\kappa^+$ and a finite set $D\subseteq C\cup \{ a_j\mid j<\alpha \}$ such that $C\subseteq D$ and for all $i\in Y$ we have $a_i\ind_{D} \{a_j \mid j\in i\cap Y \}$. Furthermore, we can choose $Y$ so that $\stp(a_i/D) = \stp(a_j/D)$ for all $i\in Y$ because the number of strong types over a finite set of parameters is countable. Then, by the stationarity of strong types, we obtain in particular that $(a_i)_{i\in Y}$ is indiscernible over $C$.

    By reindexing, we may assume that $Y=\kappa^+$, so $(a_i)_{i<\kappa^+}$ is $C$-indiscernible. Since $T^*$ is 1-based, it follows from \cref{1-based.lemma} that $a_1\ind_{Ca_0}(a_i)_{1<i<\kappa^+}$. Since we also have that $a_1\ind_{C}C a_0$, it follows by the transitivity of non-forking that $a_1\ind_{C}(a_i)_{1<i<\kappa^+}$.

    By locality of non-forking, for all finite tuples $\vec{b}\in\B^{<\omega}$, there is an ordinal $\gamma_{\vec{b}}$ such that $\vec{b}\ind_{C\cup \{a_j\mid 0<j< \gamma_{\vec{b}}\}} (a_j)_{0<j<\kappa^+}$. Let $\gamma=\sup\{\gamma_{\vec{b}} \mid \vec{b}\in \B^{<\omega} \}$. Now, as $\kappa^+$ is regular and $|\B^{<\omega}| = \kappa < \kappa^+$, it follows that $\gamma<\kappa^+$. Then $\vec{b}\ind_{C\cup \{a_j\mid 0<j< \gamma\}} (a_j)_{0<j<\kappa^+}$ for all $\vec b\in\B^{<\omega}$, which means by definition $\B\ind_{C\cup \{a_j\mid 0<j< \gamma\}} (a_j)_{0<j<\kappa^+}$. By symmetry and monotonicity of non-forking, this yields $a_\gamma \ind_{C\cup \{a_j\mid 0<j< \gamma\}} \B $. Moreover, by the fact that $(a_i)_{0<i<\kappa^+}$ is independent over $C$, we also have that $a_\gamma \ind_{C} C\cup \{a_i \mid 0<i <\gamma \}   $. It follows by the transitivity of non-forking that $a_\gamma \ind_{C} \B$. Finally, since we showed that $p\restriction C = \foltype(a/C)$ is stationary, we obtain $a_\gamma\models p$. As $a_\gamma\in\A$, this means that $p$ is realized in $\A$, which concludes the proof.
\end{proof}

\begin{example}
    Let $T_E$ be the first-order theory of a single equivalence relation $E(x,y)$ that partitions the domain into infinitely many infinite equivalence classes. Let $\bar T_E$ be its unique $\eso$-completion. Obviously $T_E$ is $\aleph_0$-categorical, and it is straightforward to verify that it is also $\omega$-stable and 1-based. Then the previous theorem shows that $\bar T_E$ is also uncountably categorical.
\end{example}

\printbibliography

@article{SAARINEN_2025, title={On the categoricity of complete second order theories}, DOI={10.1017/jsl.2025.32}, journal={The Journal of Symbolic Logic}, author={Saarinen, Tapio and Väänänen, Jouko Antero and Woodin, William Hugh}, year={2025}}

@book{baldwin_categoricity,
  title={Categoricity},
  author={Baldwin, J. T.},
  year={2009},
volume={50},
  publisher={American Mathematical Society}
}

@article{Hodges,
	author = {W. Hodges},
	volume = {5},
	pages = {539--563},
	publisher = {Oxford University Press},
	title = {Compositional Semantics for a Language of Imperfect Information},
	year = {1997},
	number = {4},
	journal = {Logic Journal of the IGPL}
}

@article{gradel2013dependence,
  title={Dependence and independence},
  author={Gr{\"a}del, Erich and V{\"a}{\"a}n{\"a}nen, Jouko},
  journal={Studia Logica},
  volume={101},
  number={2},
  pages={399--410},
  year={2013},
  publisher={Springer}
}

@article{lindstrom1969extensions,
  title={On extensions of elementary logic},
  author={Lindstr{\"o}m, Per},
  journal={Theoria},
  volume={35},
  number={1},
  pages={1--11},
  year={1969},
  publisher={Wiley Online Library}
}

@article{galliani2012inclusion,
  title={Inclusion and exclusion dependencies in team semantics—on some logics of imperfect information},
  author={Galliani, Pietro},
  journal={Annals of Pure and Applied Logic},
  volume={163},
  number={1},
  pages={68--84},
  year={2012},
  publisher={Elsevier}
}

@book{poizat2012course,
  title={A course in model theory: an introduction to contemporary mathematical logic},
  author={Poizat, Bruno},
  year={2012},
  publisher={Springer}
}

@article{de2003geometry,
  title={The geometry of 1-based minimal types},
  author={De Piro, Tristram and Kim, Byunghan},
  journal={Transactions of The American Mathematical Society},
  volume={355},
  number={10},
  pages={4241--4263},
  year={2003}
}

@inproceedings{shelah2006classification,
	title={Classification of non elementary classes II},
    subtitle={Abstract elementary classes},
	author={Shelah, Saharon},
	booktitle={Classification Theory: Proceedings of the US-Israel Workshop on Model Theory in Mathematical Logic held in Chicago, Dec. 15--19, 1985},
	pages={419--497},
	year={2006},
	organization={Springer}
}

@article{lieberman2016classification,
  title={Classification theory for accessible categories},
  author={Lieberman, Michael and Rosický, Jiří},
  journal={The Journal of Symbolic Logic},
  volume={81},
  number={1},
  pages={151--165},
  year={2016},
  publisher={Cambridge University Press}
}

@article{rosicky1997accessible,
  title={Accessible categories, saturation and categoricity},
  author={Jiří Rosický},
  journal={The Journal of Symbolic Logic},
  volume={62},
  number={3},
  pages={891--901},
  year={1997},
  publisher={Cambridge University Press}
}

@article {MR4594292,
    AUTHOR = {Kontinen, Juha and Yang, Fan},
     TITLE = {Complete logics for elementary team properties},
   JOURNAL = {The Journal of Symbolic Logic},
    VOLUME = {88},
      YEAR = {2023},
    NUMBER = {2},
     PAGES = {579--619},
      ISSN = {0022-4812,1943-5886},
   MRCLASS = {03B60},
  MRNUMBER = {4594292},
       DOI = {10.1017/jsl.2022.80},
       URL = {https://doi.org/10.1017/jsl.2022.80},
}

@article{beke2012abstract,
  title={Abstract elementary classes and accessible categories},
  author={Beke, Tibor and Rosick{\'y}, Jir{\'\i}},
  journal={Annals of Pure and Applied Logic},
  volume={163},
  number={12},
  pages={2008--2017},
  year={2012},
  publisher={Elsevier}
}

@book{adamek1994locally,
  title={Locally presentable and accessible categories},
  author={Ad{\'a}mek, Ji{\v{r}}{\'\i} and Rosick{\'y}, Ji{\v{r}}{\'\i}},
  year={1994},
  publisher={Cambridge University Press}
}

@inbook{lindstrom1974characterizing,
  title={On characterizing elementary logic},
  author={Lindstr{\"o}m, Per},
  booktitle={Logical Theory and Semantic Analysis: Essays Dedicated to Stig Kanger on His Fiftieth Birthday},
  pages={129--146},
  year={1974},
  publisher={Springer}
}

@article{sgro1977maximal,
  title={Maximal logics},
  author={Sgro, Joseph},
  journal={Proceedings of the American Mathematical Society},
  volume={63},
  number={2},
  pages={291--298},
  year={1977}
}

@article{VAANANEN2010817,
title = {Dependence of variables construed as an atomic formula},
journal = {Annals of Pure and Applied Logic},
volume = {161},
number = {6},
pages = {817-828},
year = {2010},
note = {The proceedings of the IPM 2007 Logic Conference},
issn = {0168-0072},
doi = {https://doi.org/10.1016/j.apal.2009.06.009},
url = {https://www.sciencedirect.com/science/article/pii/S0168007209001195},
author = {Jouko Väänänen and Wilfrid Hodges}
}

@article{kamsma2020kim,
  title={The Kim--Pillay theorem for abstract elementary categories},
  author={Kamsma, Mark},
  journal={The Journal of Symbolic Logic},
  volume={85},
  number={4},
  pages={1717--1741},
  year={2020},
  publisher={Cambridge University Press}
}

@article{puljujärvi2022compactness,
        title={Compactness in team semantics},
  author={Puljuj{\"a}rvi, Joni and Quadrellaro, Davide Emilio},
  journal={Mathematical Logic Quarterly},
  volume={70},
  number={2},
  pages={142--161},
  year={2024},
  publisher={Wiley Online Library}
}

@article{magidor2011lowenheim,
  title={On L{\"o}wenheim--Skolem--Tarski numbers for extensions of first order logic},
  author={Magidor, Menachem and V{\"a}{\"a}n{\"a}nen, Jouko},
  journal={Journal of Mathematical Logic},
  volume={11},
  number={01},
  pages={87--113},
  year={2011},
  publisher={World Scientific}
}

@phdthesis{luck2020team,
  title={Team logic: axioms, expressiveness, complexity},
  author={L{\"u}ck, Martin},
  year={2020},
  school={Hannover: Institutionelles Repositorium der Leibniz Universit{\"a}t Hannover}
}

@book{Vaananen2007-VNNDLA,
	author = {Jouko V\"{a}\"{a}n\"{a}nen},
	title = {Dependence Logic: A New Approach to Independence Friendly Logic},
	year = {2007},
	publisher = {Cambridge University Press}
}

@book{tentziegler,
	doi = {10.1017/cbo9781139015417},
	url = {https://doi.org/10.1017/cbo9781139015417},
	year = {2009},
	publisher = {Cambridge University Press},
	author = {Katrin Tent and Martin Ziegler},
	title = {A Course in Model Theory}
}

@book{baldwin2017fundamentals,
  title={Fundamentals of stability theory},
  author={Baldwin, John T},
  volume={12},
  year={2017},
  publisher={Cambridge University Press}
}

@book{jech,
	author = {Thomas Jech},
	title = {Set Theory},
	year = {2014},
	publisher = {Springer},
	address =   {Berlin}
}

@book{shelah1990classification,
  title={Classification theory: and the number of non-isomorphic models},
  author={Shelah, Saharon},
  year={1990},
  publisher={Elsevier}
}

@article {MR0297554,
    AUTHOR = {Shelah, Saharon},
     TITLE = {Every two elementarily equivalent models have isomorphic
              ultrapowers},
  JOURNAL = {Israel Journal of Mathematics},
    VOLUME = {10},
      YEAR = {1971},
     PAGES = {224--233},
      ISSN = {0021-2172},
   MRCLASS = {02H99},
  MRNUMBER = {297554},
MRREVIEWER = {John\ L.\ Bell},
       DOI = {10.1007/BF02771574},
       URL = {https://doi.org/10.1007/BF02771574},
}

@article {MR0295904,
    AUTHOR = {Magidor, M.},
     TITLE = {On the role of supercompact and extendible cardinals in logic},
  JOURNAL = {Israel Journal of Mathematics},
    VOLUME = {10},
      YEAR = {1971},
     PAGES = {147--157},
      ISSN = {0021-2172},
   MRCLASS = {02K35},
  MRNUMBER = {295904},
MRREVIEWER = {John\ L.\ Bell},
       DOI = {10.1007/BF02771565},
       URL = {https://doi.org/10.1007/BF02771565},
}

@article {MR0142459,
    AUTHOR = {Frayne, T. and Morel, A. C. and Scott, D. S.},
     TITLE = {Reduced direct products},
   JOURNAL = {Fundamenta Mathematicae},
    VOLUME = {51},
      YEAR = {1962},
     PAGES = {195--228},
      ISSN = {0016-2736,1730-6329},
   MRCLASS = {02.50},
  MRNUMBER = {142459},
       DOI = {10.4064/fm-51-3-195-228},
       URL = {https://doi.org/10.4064/fm-51-3-195-228},
}

@article {MR0148547,
    AUTHOR = {Keisler, H. Jerome},
     TITLE = {Limit ultrapowers},
   JOURNAL = {Transactions of the American Mathematical Society},
    VOLUME = {107},
      YEAR = {1963},
     PAGES = {382--408},
      ISSN = {0002-9947,1088-6850},
   MRCLASS = {02.50},
  MRNUMBER = {148547},
MRREVIEWER = {A.\ Mostowski},
       DOI = {10.2307/1993808},
       URL = {https://doi.org/10.2307/1993808},
}

@book {enderton2001logic,
    AUTHOR = {Enderton, Herbert B.},
     TITLE = {A mathematical introduction to logic},
 PUBLISHER = {Academic Press, Burlington (MA)},
      YEAR = {2001},
      ISBN = {0-12-238452-0},
   MRCLASS = {03-01},
  MRNUMBER = {1801397},
MRREVIEWER = {Yehuda\ Rav},
}

@incollection {vanbenthem2001higher,
    AUTHOR = {van Benthem, Johan and Doets, Kees},
     TITLE = {Higher-order logic},
 BOOKTITLE = {Handbook of philosophical logic},
    VOLUME = {1},
     PAGES = {189--243},
 PUBLISHER = {Springer, Dordrecht},
      YEAR = {2001},
      ISBN = {0-7923-7018-X},
   MRCLASS = {03B15},
  MRNUMBER = {1884625},
}

@InCollection{sep-logic-higher-order,
	author       =	{Väänänen, Jouko},
	title        =	{{Second-order and Higher-order Logic}},
	booktitle    =	{The {Stanford} Encyclopedia of Philosophy},
	editor       =	{Edward N. Zalta},
	howpublished =	{\url{https://plato.stanford.edu/archives/fall2021/entries/logic-higher-order/}},
	year         =	{2021},
	edition      =	{{F}all 2021},
	publisher    =	{Metaphysics Research Lab, Stanford University}
}

@misc{kirby_abstract_2008,
	title = {Abstract {Elementary} {Categories}},
	author = {Kirby, Jonathan},
	month = aug,
	year = {2008},
	note = {Unpublished}
}

@phdthesis{keisler-thesis,
title={Ultraproducts and elementary classes},
  author={Keisler, H Jerome},
  year={1961},
  school={University of California, Berkeley}
}

@article {MR3028798,
    AUTHOR = {Galliani, Pietro},
     TITLE = {General models and entailment semantics for independence
              logic},
   JOURNAL = {Notre Dame Journal of Formal Logic},
    VOLUME = {54},
      YEAR = {2013},
    NUMBER = {2},
     PAGES = {253--275},
      ISSN = {0029-4527,1939-0726},
   MRCLASS = {03B60 (03F03)},
  MRNUMBER = {3028798},
MRREVIEWER = {Juha\ Kontinen},
       DOI = {10.1215/00294527-1960506},
       URL = {https://doi.org/10.1215/00294527-1960506},
}

@article{CODY2014620,
title = {On supercompactness and the continuum function},
journal = {Annals of Pure and Applied Logic},
volume = {165},
number = {2},
pages = {620-630},
year = {2014},
issn = {0168-0072},
doi = {https://doi.org/10.1016/j.apal.2013.09.001},
url = {https://www.sciencedirect.com/science/article/pii/S0168007213001358},
author = {Brent Cody and Menachem Magidor}
}

\end{document}